\documentclass[letterpaper,11pt,cleveref, autoref, thm-restate]{article}
\usepackage{footmisc}
\usepackage[margin=1in]{geometry}
\usepackage{xcolor}
\usepackage{enumitem}
\usepackage{amsthm}
\usepackage{amsmath}
\usepackage{amssymb}
\usepackage{amsfonts}
\usepackage{graphicx}
\usepackage{thm-restate}
\usepackage[colorlinks = true]{hyperref}
\hypersetup{
    linkcolor=black,
    citecolor=black,
    urlcolor=blue}

\usepackage{mathtools}

\usepackage{titlesec}
\titleformat{\subsection}[runin]
       {\normalfont\bfseries}
       {\thesubsection}
       {0.5em}
       {}
       [.]

\allowdisplaybreaks

\newtheorem{thrm}{Theorem}[section]
\newtheorem{lem}[thrm]{Lemma}
\newtheorem{prop}[thrm]{Proposition}

\newtheorem{cor}[thrm]{Corollary}
\newtheorem{obs}[thrm]{Observation}
\theoremstyle{definition}
\newtheorem{defn}[thrm]{Definition}
\theoremstyle{definition}

\theoremstyle{definition}

\theoremstyle{definition}
\newtheorem{fct}[thrm]{Fact}
\theoremstyle{definition}

\newcommand{\Z}{\mathbb{Z}}
\newcommand{\N}{\mathbb{N}}
\newcommand{\Q}{\mathbb{Q}}

\newcommand{\R}{\mathbb{R}}

\newcommand{\A}{\mathbb{A}}

\newcommand{\Rp}{\mathbb{R}_{\geq 0}}
\newcommand{\Rpp}{\mathbb{R}_{>0}}

\newcommand{\Qpp}{\mathbb{Q}_{>0}}
\newcommand{\Zp}{\mathbb{Z}_{\geq 0}}
\newcommand{\Zpp}{\mathbb{Z}_{>0}}

\newcommand{\SL}{\mathsf{SL}}
\newcommand{\GL}{\mathsf{GL}}

\newcommand{\conv}{\operatorname{conv}}
\newcommand{\inte}{\operatorname{int}}
\newcommand{\lat}{\operatorname{Lat}}
\newcommand{\gq}{\Gamma_{\Q}}
\newcommand{\hG}{\widehat{\Gamma}}
\newcommand{\init}{\operatorname{in}}
\newcommand{\mG}{\mathcal{G}}
\newcommand{\sgmG}{\langle \mathcal{G} \rangle}

\newcommand{\mM}{\mathcal{M}}
\newcommand{\mH}{\mathcal{H}}

\newcommand{\mY}{\mathcal{Y}}
\newcommand{\mP}{\mathcal{P}}
\newcommand{\mL}{\mathcal{L}}
\newcommand{\mA}{\mathcal{A}}
\newcommand{\mB}{\mathcal{B}}

\newcommand{\bg}{\boldsymbol{g}}

\newcommand{\bm}{\boldsymbol{m}}
\newcommand{\bn}{\boldsymbol{n}}
\newcommand{\tby}{\widetilde{\boldsymbol{y}}}
\newcommand{\tbf}{\widetilde{\boldsymbol{f}}}
\newcommand{\tf}{\widetilde{f}}
\newcommand{\tc}{\widetilde{c}}
\newcommand{\bF}{\boldsymbol{F}}

\newcommand{\balpha}{\boldsymbol{\alpha}}

\newcommand{\oX}{\mkern 1.5mu\overline{\mkern-1.5mu X \mkern-1.5mu}\mkern 1.5mu}
\newcommand{\bff}{\boldsymbol{f}}

\newcommand{\ApK}{\left(\A^+\right)^K}
\newcommand{\Rns}{\left(\R^n\right)^*}

\newcommand{\ncl}{\operatorname{ncl}}

\newcounter{ProblemCounter}

\begin{document}

\title{Semigroup algorithmic problems in metabelian groups}
\author{Ruiwen Dong\footnote{Department of Computer Science, University of Oxford, Oxford, OX1 3QD, United Kingdom, email: \url{ruiwen.dong@kellogg.ox.ac.uk}}}
\date{}
\maketitle
\begin{abstract}
    We consider semigroup algorithmic problems in finitely generated metabelian groups.
    Our paper focuses on three decision problems introduced by Choffrut and Karhum\"{a}ki (2005): the \emph{Identity Problem} (does a semigroup contain a neutral element?), the \emph{Group Problem} (is a semigroup a group?) and the \emph{Inverse Problem} (does a semigroup contain the inverse of a generator?).
    We show that all three problems are decidable for finitely generated sub-semigroups of finitely generated metabelian groups.
    In particular, we establish a correspondence between polynomial semirings and sub-semigroups of metabelian groups using an interaction of graph theory, convex polytopes, algebraic geometry and number theory.
    
    Since the \emph{Semigroup Membership} problem (does a semigroup contain a given element?) is known to be undecidable in finitely generated metabelian groups, our result completes the decidability characterization of semigroup algorithmic problems in metabelian groups.
\end{abstract}

\newpage

\section{Introduction}
\subsection{Algorithmic problems in groups and semigroups}
In 1911 Max Dehn formulated three basic problems which would become the foundation of computational group theory.
Given a finite presentation of a group $G$, it is asked whether there are algorithms that solve the \emph{Word Problem} (whether an element is the neutral element), the \emph{Conjugacy Problem} (whether two elements are conjugate), and the \emph{Isomorphism Problem} (whether $G$ is isomorphic to another finitely presented group).
All three problem are later shown to be undecidable in general groups~\cite{adyan1955algorithmic,novikov1955algorithmic}, providing the first examples of an undecidable problem not coming from the theory of computation.

Since the 1940s, due to their connection with mathematical logic, \emph{membership problems} became the centre of active research in computational group theory.
For these problems, we work in a fixed group $G$.
The input is a finite set of elements $\mG = \{g_1, \ldots, g_K\}$ in $G$ and a target element $g \in G$.
Denote by $\sgmG$ the semigroup generated by $\mG$, and by $\langle\mG\rangle_{grp}$ the group generated by $\mG$.
\begin{enumerate}[nosep, label = (\roman*)]
    \item \textit{(Semigroup Membership)} decide whether $\sgmG$ contains $g$.
    \item \textit{(Group Membership)} decide whether $\langle\mG\rangle_{grp}$ contains $g$.
    \setcounter{ProblemCounter}{\value{enumi}}
\end{enumerate}
In the seminal work of Markov~\cite{markov1947certain}, it is shown that Semigroup Membership is undecidable for matrix groups of dimension six.
Mikhailova~\cite{mikhailova1966occurrence} later showed undecidability of Group Membership in the group $\SL(4, \Z)$ of $4 \times 4$ integer matrices with determinant one.

As some of the oldest and most well-developed problems of computational algebra, membership problems play an essential role in analysing system dynamics, and has numerous applications in automata theory, program analysis, and interactive proof systems~\cite{beals1993vegas, blondel2005decidable, derksen2005quantum, hrushovski2018polynomial}.
In most classes of groups, Group Membership tends to be much more tractable than Semigroup Membership.
For example, Group Membership is decidable in the class of \emph{polycyclic groups} by a classic result of Kopytov~\cite{kopytov1968solvability}; whereas Semigroup Membership is undecidable even in the subclass of \emph{nilpotent groups}~\cite{roman2022undecidability}.
This gap motivated the introduction of a series of intermediate problems by Choffrut and Karhum\"{a}ki~\cite{choffrut2005some} in 2005:
\begin{enumerate}[nosep, label = (\roman*)]
    \setcounter{enumi}{\value{ProblemCounter}}
    \item \textit{(Identity Problem)} decide whether $\sgmG$ contains the neutral element of $G$.
    \item \textit{(Group Problem)} decide whether $\langle\mG\rangle$ is a group.
    \item \textit{(Inverse Problem)} given $a \in \mG$, decide whether $a^{-1} \in \langle\mG\rangle$.
    \setcounter{ProblemCounter}{\value{enumi}}
\end{enumerate}

Apart from being some of the essential special cases of Semigroup Membership, these intermediate problems are crucial in determining structural properties of semigroups, and motivated the development of numerous tools in the study of semigroups, ranging from automata theory and compressed words~\cite{bell2017identity, bell2010undecidability} to Lie algebra~\cite{https://doi.org/10.48550/arxiv.2208.02164}.
It is not difficult to see that decidability of the Group Problem subsumes decidability of the Identity Problem and the Inverse Problem (see Section~\ref{sec:prelim}).
A recent result of Bell and Potapov showed undecidability of all three problems in $\SL(4, \Z)$~\cite{bell2010undecidability}; whereas the Identity Problem in $\SL(2, \Z)$ is NP-complete~\cite{bell2017identity}.
In~\cite{babai1996multiplicative}, Babai et al.\ famously reduced algorithmic problems in abelian matrix groups to computation on lattices, thus all three problems in abelian matrix groups are decidable in PTIME by solving homogeneous linear Diophantine equations.
However, decidability of these intermediate problems remains open for larger classes of groups, notably nilpotent groups, polycyclic groups and metabelian groups, where decidability of membership problems have definitive answers.

\subsection{Metabelian groups and main result}
In this paper we study algorithmic problems in \emph{metabelian groups}.
Metabelian groups are groups whose \emph{commutator} is abelian.
Recall that for a group $G$, its \emph{commutator} $[G, G]$ is defined as the subgroup of $G$ generated by the elements $g h g^{-1} h^{-1}$, $g, h \in G$.
Developing a complete algorithmic theory for finitely generated metabelian groups has been the focus of intense research since the 1950s.
For a surveys of recent developments, see~\cite{baumslag1994algorithmic,kharlampovich1995algorithmic}.
As the convention in computational group theory, a finitely generated metabelian group $G$ is always given as a part of the input by a \emph{finite metabelian presentation} (see Section~\ref{sec:prelim} for its definition).
Every finitely generated metabelian group admits a finite metabelian presentation, making them a natural target for algorithmic methods~\cite[p.629]{baumslag1994algorithmic}.

Among the classic Max Dehn problems for finitely generated metabelian groups, decidability of the Word Problem is known since the 1950s following the seminal work of Hall~\cite{hall1954finiteness}.
The Conjugacy Problem is shown to be decidable by Noskov~\cite{noskov1982conjugacy}.
The Isomorphism Problem remains an outstanding open problem~\cite{baumslag2017localization}.
We may note that in the hierarchy of solvable groups, metabelian groups (also known as 2-step solvable groups) are on the fringe of decidability. By a celebrated result of Kharlampovich, all three Max Dehn problems are undecidable in 3-step solvable groups (these are groups whose commutator is metabelian)~\cite{kharlampovich1981finitely}, \cite[Theorem~6.17, Section~6.8]{kharlampovich1995algorithmic}.

In finitely generated metabelian groups, decidability of Group Membership is a classic result of Romanovskii~\cite{romanovskii1974some};
whereas Semigroup Membership is undecidable for many instances such as large direct powers of the Heisenberg group~\cite{roman2022undecidability} and the wreath product $\Z \wr \Z$~\cite{lohrey2015rational}.
Decidability of the Identity Problem, the Group Problem and the Inverse Problem remained open.
A recent result by Dong~\cite{dong2023identity} showed decidability of the Group Problem in an important example of metabelian groups, the \emph{wreath product} $\Z \wr \Z$.
This hinted at the possibility of decidability results for other finitely generated metabelian groups.
Our main results solve these open problems:

\begin{restatable}{thrm}{thmmain}\label{thm:main}
    The Group Problem, the Identity Problem and the Inverse Problem are decidable in all finitely generated metabelian groups.
\end{restatable}

\subsection{Related work}
It has been noticed since the work of Hall that metabelian groups have natural connections with polynomials rings.
Indeed, this connection is the key to deciding many \emph{group} algorithmic problems in metabelian groups.
However, a corresponding theory for \emph{semigroups} has yet to be developed.
Recent work by Dong~\cite{dong2023identity} suggested there are connections between certain semigroups, directed graphs and polynomial semirings.
However, such connections are highly sophisticated and a satisfactory characterization is yet to be obtained.
We build on Dong's work to establish a full connection between sub-\emph{semigroups} of metabelian groups and polynomial \emph{semirings}.
Many of our ideas are inspired by~\cite{dong2023identity}, notably the use of \emph{$\mG$-graphs} to describe words over metabelian groups, as well as exploiting the interaction between semigroups, graphs and algebraic geometry.
Here are our main new contributions.
\begin{enumerate}[noitemsep, leftmargin=*, wide, labelindent=0pt]
    \item We introduce a \emph{systematic} way of translating from words over metabelian groups to elements in polynomial semirings. We introduce the notion of ``position polynomials'' to describe the associated $\mG$-graph. These polynomials are powerful enough to describe interesting properties of the graph such as ``full-image'' and symmetry.
    In Dong's work this translation was done using an \emph{ad hoc} method that decomposes a walk over $\Z$ into ``primitive circuits''.
    As noted there, that method could not be generalized to walks over $\Z^n$. The new method we develop here overcomes this difficulty.
    \item The drawback of using ``position polynomials'' instead of decomposition into primitive circuits is the inability to express the connectivity property of the graph.
    Our second main idea is to introduce a new property of the graph called ``face-accessibility''.
    This property is weaker then connectivity but has the advantage of being describable by position polynomials.
    We then show that together with symmetry, face-accessibility is enough to characterize Eulerian graphs up to taking unions of translations.
    This is done using a series of manipulation in convex geometry and graph theory.
    \item Thanks to the two previous ideas we are able to reduce semigroup problems to decision problems over polynomial semirings.
    Our third new idea is a simultaneous generalization of two deep results by Einsiedler, Mouat and Tuncel~\cite[Theorem~1.3]{einsiedler2003does} and by Dong~\cite[Proposition~3.4]{dong2023identity}.
    This is mathematically the deepest part of our paper: it includes highly intricate applications of algebraic geometry tools such as Gr\"{o}bner basis over modules, as well as various ideas from number theory.
\end{enumerate}

\section{Preliminaries}\label{sec:prelim}
\subsection{Words, semigroups and groups}
All omitted proofs of this section can be found in Appendix~\ref{app:prelim}.
Let $G$ be an arbitrary group.
Let $\mG = \{g_1, \ldots, g_K\}$ be a set of elements in $G$.
Considering $\mG$ as an alphabet, denote by $\mG^*$ the set of words over $\mG$.
For an arbitrary word $w = g_{i_1} g_{i_2} \cdots g_{i_m} \in \mG^*$, by multiplying consecutively the elements appearing in $w$, we can evaluate $w$ as an element $\pi(w)$ in $G$.
We say that the word $w$ \emph{represents} the element $\pi(w)$.
The semigroup $\langle \mG \rangle$ generated by $\mG$ is hence the set of elements in $G$ that are represented by \emph{non-empty} words in $\mG^*$.
A word $w$ over the alphabet $\mG$ is called \emph{full-image} if every letter in $\mG$ has at least one occurrence in $w$.

\begin{restatable}[{\cite[Lemma~2.1]{dong2023identity}}]{lem}{lemword}\label{lem:word}
The semigroup $\langle \mG \rangle$ is a group if and only if the neutral element of $G$ is represented by a full-image word over $\mG$.
\end{restatable}

The following lemma is a classic reduction between the algorithmic problems we consider.
By Lemma~\ref{lem:subsume}, we can focus solely on the Group Problem throughout this paper.

\begin{restatable}{lem}{lemsubsume}\label{lem:subsume}
Let $G$ be a group.
If the Group Problem is decidable in $G$, then the Identity Problem and the Inverse Problem are also decidable in $G$.
\end{restatable}

\subsection{Polytopes, Laurent polynomials and modules}
For a detailed reference on convex polytopes, see~\cite{alexandrov2005convex}.
Let $C$ be a (closed) convex polytope. A \emph{face} $F$ of $C$ is the intersection of $C$ with any closed halfspace whose boundary is disjoint from the interior of $C$.
A \emph{strict face} is a face of $C$ that is not the empty set or $C$ itself.
For example, if $C$ is of dimension two, then the strict faces of $C$ are its edges and its vertices.

Let $R$ be a commutative ring (such as $\Z$ or $\R$) or semiring (such as $\N$ or $\Rp$).
Denote by $R[X_1^{\pm}, \ldots, X_n^{\pm}]$ the Laurent polynomial ring or semiring over $R$ with $n$ variables: this is the set of polynomials of variables $X_1, X_1^{-1}, \ldots, X_n, X_n^{-1}$ with coefficients in $R$.
When $n$ is fixed, we denote
\[
R[\oX^{\pm}] \coloneqq R[X_1^{\pm}, \ldots, X_n^{\pm}], \quad R[\oX^{\pm}]^* \coloneqq R[\oX^{\pm}] \setminus \{0\}.
\]
For a vector $a = (a_1, \ldots, a_n) \in \Z^n$, denote by $\oX^a$ the monomial $X_1^{a_1} X_2^{a_2} \cdots X_n^{a_n}$.
Let $\cdot$ denote the dot product in $\R^n$.
Given $f \in R[\oX^{\pm}]$ and a vector $v \in \Rns \coloneqq \R^n \setminus \{0\}$, define the \emph{weighted degree} 
\[
\deg_v(f) \coloneqq \max\{v \cdot  a \mid a \in \Z^n, c_a \neq 0\}, \quad \text{ where } f = \sum c_a \oX^a \neq 0.
\]
Additionally, define $\deg_v(0) = -\infty$ for all $v \in \Rns$.

Let $R$ be a commutative ring. An $R[\oX^{\pm}]$-module is an abelian group $(M, +)$ along with an operation $\cdot \;\colon R[\oX^{\pm}] \times M \rightarrow M$ satisfying $f \cdot (a+b) = f \cdot a + f \cdot b$, $(f + g) \cdot a = f \cdot a + g \cdot a$, $fg \cdot a = f \cdot (g \cdot a)$ and $1 \cdot a = a$.
For example, for any $d \in \N$, $R[\oX^{\pm}]^d$ is an $R[\oX^{\pm}]$-module by $f \cdot (g_1, \ldots, g_d) = (fg_1, \ldots, fg_d)$.

\underline{Throughout this paper, we use the bold symbol $\bff$ to denote a vector $(f_1, \ldots, f_d) \in R[\oX^{\pm}]^d$.}

Given $\bg_1, \ldots, \bg_m \in R[\oX^{\pm}]^d$, we say they \emph{generate} the $R[\oX^{\pm}]$-module $\sum_{i=1}^m R[\oX^{\pm}] \cdot \bg_i \coloneqq \{\sum_{i=1}^m p_i \cdot \bg_i \mid p_1, \ldots, p_m \in R[\oX^{\pm}] \}$.
A module is called \emph{finitely generated} if it can be generated by a finite number of elements.
Given two finitely generated submodules $N, M$ of $R[\oX^{\pm}]^d$ such that $N \subseteq M$, we can define the quotient $M/N \coloneqq \{\overline{m} \mid m \in M\}$ where $\overline{m_1} = \overline{m_2}$ iff $m_1 - m_2 \in N$.
This quotient is also an $R[\oX^{\pm}]$-module.
We say that an $R[\oX^{\pm}]$-module $\mY$ is \emph{finitely presented} if it can be written as a quotient $M/N$ for two finitely generated submodules $N \subseteq M$ of $R[\oX^{\pm}]^d$ for some $d \in \N$.
We call a \emph{finite presentation} of $\mY$ the respective generators of such $M, N$.

\subsection{Representing a metabelian group}
Metabelian groups are usually represented by a \emph{finite metabelian presentation}.
We recall here its formal definition.
Understanding the technical details in the definition is not essential, since we will only be using the more intuitive representation given by Equations~\eqref{eq:defsemi}, \eqref{eq:defsemi2} and Proposition~\ref{prop:metatoZ} throughout this paper.

Let $F_s$ be the free group over $s \geq 2$ generators.
The quotient 
$
M_s \coloneqq F_s/[[F_s, F_s],[F_s, F_s]]
$
is metabelian and is called the \emph{free metabelian group} over $s$ generators.
Let $\{x_1, \ldots, x_s\}$ be the generators of $F_s$, then their equivalence classes $\{\overline{x}_1, \ldots, \overline{x}_s\}$ are the generators of $M_s$.
An element of $M_s$ is represented as a word over $\{\overline{x}_1, \ldots, \overline{x}_s\}$.

\begin{defn}[Finite metabelian presentation]
    Let $G$ be a metabelian group. A \emph{finite metabelian presentation} of $G$ is the generators of a free metabelian group $M_s, s \geq 2$, along with a finite set of elements $r_1, \ldots, r_m$, such that
    $
    G = M_s / \ncl_{M_s}(r_1, \ldots, r_m)
    $.
    Here, $\ncl_{M_s}(r_1, \ldots, r_m)$ denotes that \emph{normal closure} of $\{r_1, \ldots, r_m\}$, that is, the smallest normal subgroup of $M_s$ containing $\{r_1, \ldots, r_m\}$.
\end{defn}

By~\cite[Corollary~1]{hall1954finiteness} or \cite[p.629]{baumslag1994algorithmic}, every finitely generated metabelian group admits a finite metabelian presentation. 

Given a finitely presented $\Z[X_1^{\pm}, \ldots, X_n^{\pm}]$-module $\mY$, define the following semidirect product:
\begin{equation}\label{eq:defsemi}
\mY \rtimes \Z^n \coloneqq \{(y, a) \mid y \in \mY, a \in \Z^n\};
\end{equation}
this is a group where multiplication and inversion are defined by
\begin{equation}\label{eq:defsemi2}
(y, a) \cdot (y', a') = (y + \oX^a \cdot y', a + a'), \quad (y, a)^{-1} = (- \oX^{-a} \cdot y, -a).
\end{equation}
The neutral element of $\mY \rtimes \Z^n$ is $(0, 0)$.
Intuitively, the element $(y, a)$ can be seen as a $2 \times 2$ matrix
$
\begin{pmatrix}
\oX^{a} & y \\
0 & 1 \\
\end{pmatrix}
$, where group multiplication is represented by matrix multiplication.\footnote{When $n = 0$, the polynomial ring $\Z[X_1^{\pm}, \ldots, X_n^{\pm}]$ becomes $\Z$, and the group $\mY \rtimes \Z^n$ degenerates into the $\Z$-module $\mY$, which is an abelian group.}
The following proposition shows that it suffices to solve the Group Problem in groups of the form $\mY \rtimes \Z^n$.

\begin{restatable}{prop}{propmetatoZ}\label{prop:metatoZ}
    Suppose we are given a finite metabelian presentation of a group $G$ as well as a finite set $\mG \subseteq G$.
    One can effectively construct a finitely presented $\Z[X_1^{\pm}, \ldots, X_n^{\pm}]$-module $\mY$ for some $n \in \N$, as well as a subset $\widetilde{\mG}$ of the group $\mY \rtimes \Z^n$, such that $\sgmG$ is a group if and only if $\langle \widetilde{\mG} \rangle$ is a group.
    Furthermore, the constructed set $\widetilde{\mG}$ satisfies $\pi(\langle \widetilde{\mG} \rangle_{grp}) = \Z^n$ under the canonical projection $\pi \colon \mY \rtimes \Z^n \rightarrow \Z^n$.
\end{restatable}
\begin{proof}[Sketch of proof]
    (The full proof is given in Appendix~\ref{app:metatoZ}.)
    By~\cite[Theorem~3.3]{baumslag1994algorithmic}, we can compute a presentation for the group $\sgmG_{grp}$, so without loss of generality we can suppose $G = \sgmG_{grp}$.
    By~\cite[Lemma~3]{baumslag1973subgroups}, $G$ can be embedded as a subgroup of a quotient $\left(\mY \rtimes \Z^n\right)/H$, where $H$ is a subgroup of $\Z^n \leq \mY \rtimes \Z^n$, and elements of $H$ commute with all elements of $\mY \rtimes \Z^n$.
    We can hence suppose $G$ is given as a subgroup of $\left(\mY \rtimes \Z^n\right)/H$ and the generator set $\mG$ is given as $\{g_1 H, \ldots, g_k H\}$ where $g_1, \ldots, g_k \in \mY \rtimes \Z^n$.
    Let $h_1, \ldots, h_M$ be the generators of $H \subseteq \mY \rtimes \Z^n$ as a \emph{semi}group.
    Then $\sgmG$ is a group if and only if the semigroup generated by $\widetilde{\mG} \coloneqq \{g_1, \ldots, g_k, h_1, \ldots, h_M \} \subseteq \mY \rtimes \Z^n$ is a group.
    It is not hard to show from the construction in~\cite[Lemma~3]{baumslag1973subgroups} that $\pi(\langle \widetilde{\mG} \rangle_{grp}) = \Z^n$.
    Finally, it suffices to retrace the proof of \cite[Theorem~3.3]{baumslag1994algorithmic} and \cite[Lemma~3]{baumslag1973subgroups} to show effectiveness of this construction.
\end{proof}

By Proposition~\ref{prop:metatoZ}, we can now focus on solving the Group Problem in $\mY \rtimes \Z^n$.

\subsection{Graph theory and $\mG$-graphs}
We now fix the set of elements 
$
\mG \coloneqq \{(y_1, a_1), \ldots, (y_K, a_K)\} \subseteq \mY \rtimes \Z^n
$.
Similar to \cite[Definition~4.1]{dong2023identity}, we define the notion of $\mG$-graphs.
\begin{defn}[$\mG$-graphs]
A \emph{$\mG$-graph} is a directed multigraph $\Gamma$, whose set of vertices is a finite subset of $\Z^n$, each connected to at least one edge. The edges of $\Gamma$ are each labeled with an index in $\{1, \ldots, K\}$.
Furthermore, if an edge from vertex $v$ to vertex $w$ has label $i$, then $v = w + a_i$.
\end{defn}

For a graph $\Gamma$, we denote by $V(\Gamma)$ its set of vertices and by $E(\Gamma)$ its set of edges.
For a (directed) edge $e$, we denote by $s(e)$ its starting vertex and by $d(e)$ its destination vertex.
We call a graph \emph{Eulerian} if it contains an Euler circuit.
A directed graph is called \emph{symmetric} if for each vertex, its out-degree equals its in-degree.
A directed graph is Eulerian if and only if it is symmetric and connected.

Given $z \in \Z^n$ and a $\mG$-graph $\Gamma$, its \emph{translation} $\Gamma + z$ is a graph obtained by moving everything in $\Gamma$ by a vector $z$. See Figure~\ref{fig:translation} for an illustration.

\begin{figure}[h!]
    \centering
    \begin{minipage}[t]{.47\textwidth}
        \centering
        \includegraphics[width=0.6\textwidth,height=1.0\textheight,keepaspectratio, trim={5.5cm 0.8cm 5cm 0cm},clip]{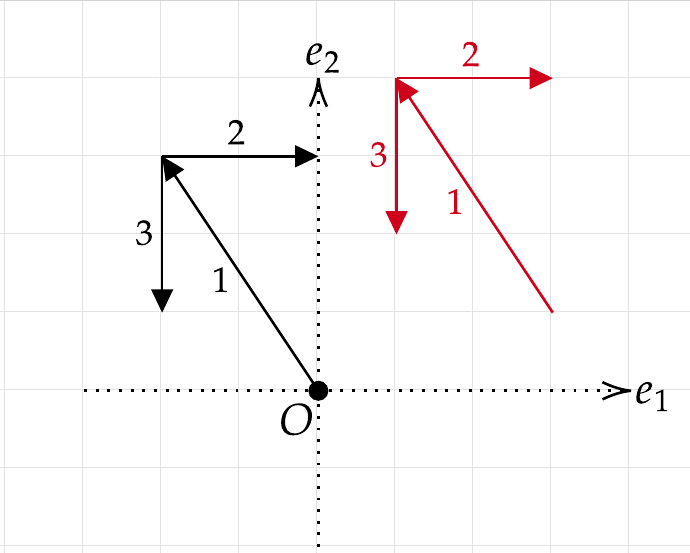}
        \caption{A $\mG$-graph $\Gamma$ (in black) and its translation $\Gamma + (3, 1)$ (in red).}
        \label{fig:translation}
    \end{minipage}
    \hfill
    \begin{minipage}[t]{0.47\textwidth}
        \centering
        \includegraphics[width=0.6\textwidth,height=1.0\textheight,keepaspectratio, trim={5.5cm 0.8cm 5cm 0cm},clip]{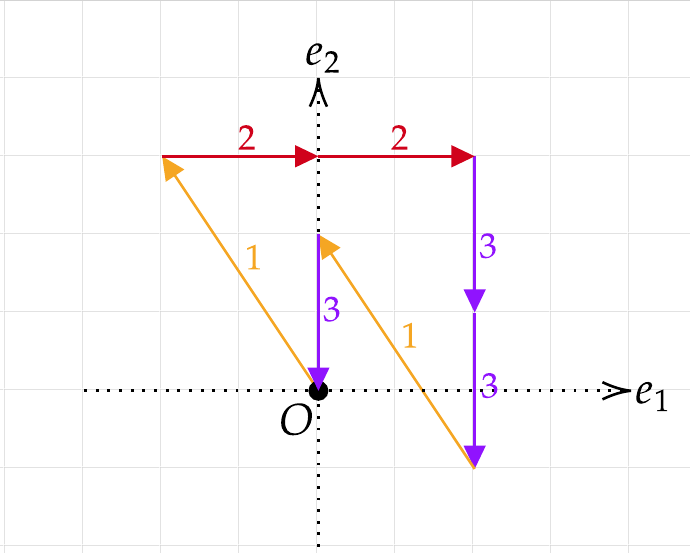}
        \caption{The graph $\Gamma(w)$ where $a_1 = (-2, 3), a_2 = (2, 0), a_3 = (0, -2)$ and $w = (y_1, a_1) (y_2, a_2) (y_2, a_2) (y_3, a_3) (y_3, a_3) (y_1, a_1) (y_3, a_3)$.}
        \label{fig:Gammaw}
    \end{minipage}
\end{figure}

\begin{defn}[Element represented by a $\mG$-graph]
For an edge $e \in E(\Gamma)$, denote by $\ell(e)$ the label of $e$.
We say $\Gamma$ \emph{represents} the following element of $\mY \rtimes \Z^n$:
\begin{equation}\label{eq:edges}
\left(\sum_{e \in E(\Gamma)} \oX^{s(e)} \cdot y_{\ell(e)}, \; \sum_{e \in E(\Gamma)} a_{\ell(e)}\right).
\end{equation}
\end{defn}

For a word $w$ over the alphabet $\mG$, we associate to it a unique $\mG$-graph $\Gamma(w)$, defined as follows.
Write $w = (y_{i_1}, a_{i_1}) (y_{i_2}, a_{i_2}) \cdots (y_{i_p}, a_{i_p})$. For each $j = 0, \ldots, p-1$, we add an edge starting at the vertex $a_{i_1} + \cdots + a_{i_{j}}$, ending at the vertex $a_{i_1} + \cdots + a_{i_{j+1}}$, with the label $i_j$. (If $j = 0$ then the edge starts at $0$ and ends at $a_{i_1}$.)
The graph $\Gamma(w)$ is then obtained by taking the connected component of the vertex $0$.
See Figure~\ref{fig:Gammaw} for an illustration.

\begin{fct}
For a word $w$ over the alphabet $\mG$, the element of $\mY \rtimes \Z^n$ represented by its associated graph $\Gamma(w)$ is equal to the element of $\mY \rtimes \Z^n$ represented by the word $w$.
\end{fct}

By reading the letters in $w$ one by one and tracing the corresponding edges of $\Gamma(w)$, we obtain an Euler path of $\Gamma(w)$.
Furthermore, if the word $w$ represents the neutral element (or any element of the form $(y, 0)$), then this Euler path is an Euler circuit.
Conversely, given an Eulerian $\mG$-graph $\Gamma$ containing the vertex 0, we can follow an Euler circuit starting from 0 and read a word $w$ such that $\Gamma(w) = \Gamma$.

A $\mG$-graph $\Gamma$ is called \emph{full-image} if it contains an edge with label $i$ for every $i \in \{1, \ldots, K\}$.
Note that for a word $w$ over the alphabet $\mG$, its associated graph $\Gamma(w)$ is full-image if and only if the word $w$ is full-image.
Combining Lemma~\ref{lem:word} with the above correspondence between words and Eulerian graphs, we immediately obtain the following lemma. 
Note that we can always translate a $\mG$-graph so that it contains the vertex 0; translating a $\mG$-graph by a vector $z$ multiplies the first entry of its represented element by $\oX^z$.
\begin{restatable}{lem}{lemgrp}\label{lem:grptoeul}
The semigroup $\sgmG$ is a group if and only if there exists a full-image Eulerian $\mG$-graph that represents the neutral element.
\end{restatable}

\section{Proof of main technical result}\label{sec:main}

By Lemma~\ref{lem:subsume} and Proposition~\ref{prop:metatoZ}, our main result (Theorem~\ref{thm:main}) boils down to proving the following technical theorem.

\begin{restatable}{thrm}{thmtec}\label{thm:tec}
    Let $\mY$ be a $\Z[X_1^{\pm}, \ldots, X_n^{\pm}]$-module with a given finite presentation.
    Suppose we are given a finite subset $\mG$ of the semidirect product $\mY \rtimes \Z^n$, such that the subgroup $\sgmG_{grp}$ of $\mY \rtimes \Z^n$ admits the image $\Z^n$ under the canonical projection $\mY \rtimes \Z^n \rightarrow \Z^n$.
    It is decidable whether the semigroup $\sgmG$ is a group.
\end{restatable}

In this section we outline the proof of Theorem~\ref{thm:tec}.
In Subsection~\ref{subsec:grptograph} we introduce the notion of ``face-accessibility'' of a $\mG$-graph to replace the property of being Eulerian.
In Subsection~\ref{subsec:graphtopoly} we introduce ``position polynomials'' to reduce problems on $\mG$-graphs to algorithmic problems over polynomial semirings.
In Subsection~\ref{subsec:locglob} we state a local-global principle concerning our algorithmic problem over polynomial semirings.
In Subsection~\ref{subsec:dec} we state our decidability result over polynomial semirings.
Proofs of several stated theorems will be given in the later Sections~\ref{sec:graph},\ref{sec:locglob} and \ref{sec:dec}.
Other omitted proofs can be found in Appendix~\ref{app:main}.

\subsection{From semigroups to face-accessible graphs}\label{subsec:grptograph}
We now fix the set of elements 
\[
\mG \coloneqq \{(y_1, a_1), \ldots, (y_K, a_K)\} \subseteq \mY \rtimes \Z^n.
\]
By Lemma~\ref{lem:grptoeul}, deciding the Group Problem boils down to finding an Eulerian $\mG$-graph.
However, being Eulerian is hard to characterize as a graph theory property.
In fact, being Eulerian is equivalent to being symmetric and connected. While symmetry is easy to describe \emph{locally} (i.e. at each vertex), connectivity of a graph is a \emph{global} property and is hence hard to describe.
The key idea of this subsection is to introduce a \emph{local} property called ``face-accessibility'' to replace connectivity.

Let $\Gamma$ be a $\mG$-graph. 
Denote by $C$ the convex hull of $V(\Gamma)$.
A strict face $F$ of $C$ is called \emph{accessible} if there exists an edge starting inside the face and ending outside of it.
That is, $F$ is accessible if there is an edge $e \in E(\Gamma)$ starting from a vertex $s(e) \in F \cap V(\Gamma)$ and ending at a vertex $d(e) \in (C \setminus F) \cap V(\Gamma)$.
See Figure~\ref{fig:accessible} for an example.

The graph $\Gamma$ is called \emph{face-accessible} if every strict face of $C$ is accessible.

\begin{figure}[ht]
    \centering
    \includegraphics[width=0.55\textwidth,height=1.0\textheight,keepaspectratio, trim={3cm 0cm 1cm 0cm},clip]{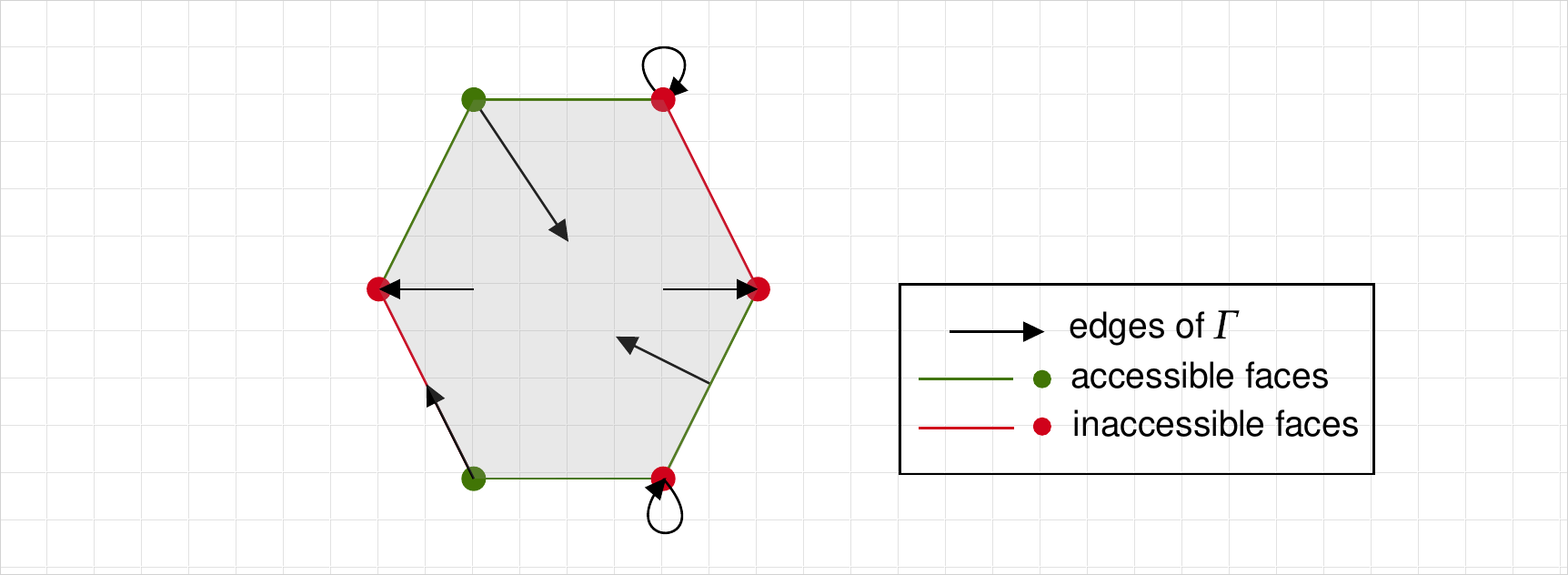}
    \caption{Accessible and inaccessible faces.}
    \label{fig:accessible}
\end{figure}

\begin{obs}\label{obs:contoacc}
    An Eulerian graph is face-accessible.
\end{obs}

On the contrary, a symmetric face-accessible graph is not necessarily Eulerian, as it may not be connected.
Moreover, a symmetric graph need not be face-accessible.
See Figure~\ref{fig:accnoteul} and \ref{fig:symnotacc} for counterexamples.

\begin{figure}[h!]
    \centering
    \begin{minipage}[t]{.47\textwidth}
        \centering
        \includegraphics[width=0.5\textwidth,height=0.6\textheight,keepaspectratio, trim={5cm 0cm 5cm 0cm},clip]{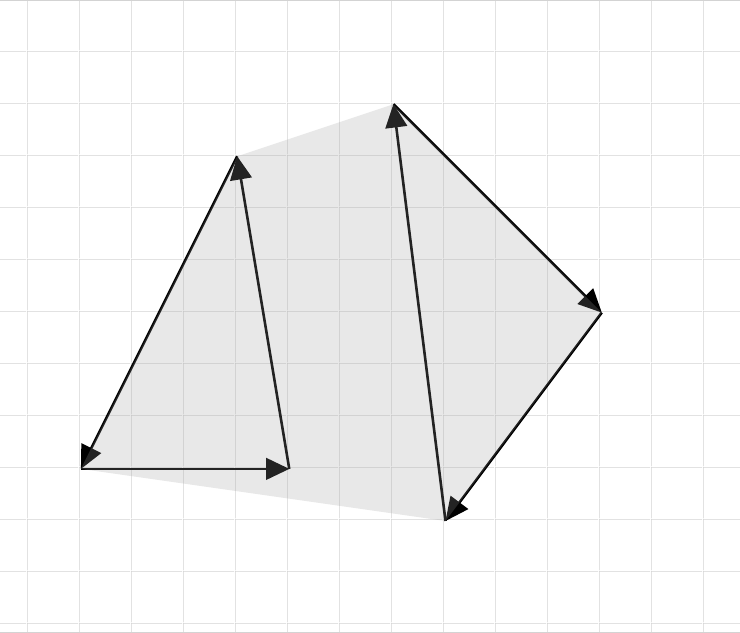}
        \caption{A symmetric face-accessible graph that is not Eulerian (due to connectivity).}
        \label{fig:accnoteul}
    \end{minipage}
    \hfill
    \begin{minipage}[t]{0.47\textwidth}
        \centering
        \includegraphics[width=0.5\textwidth,height=0.6\textheight,keepaspectratio, trim={5cm 0cm 5cm 0cm},clip]{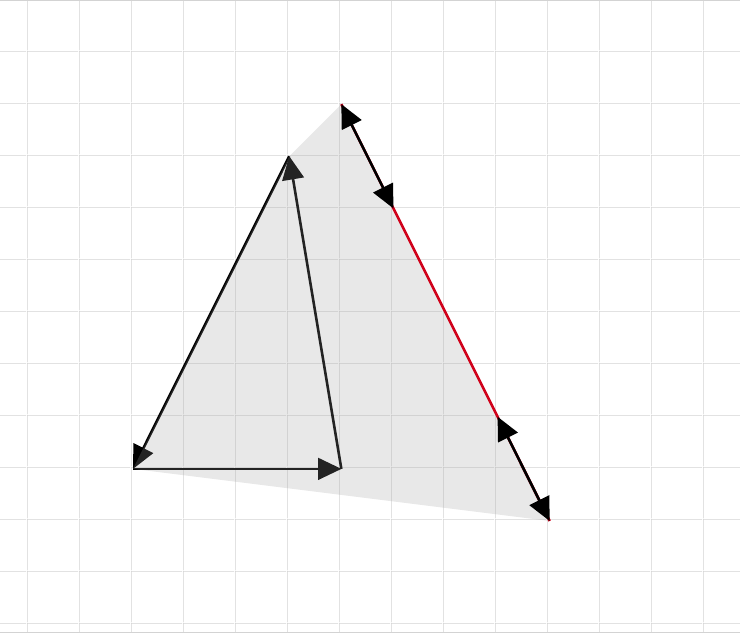}
        \caption{A symmetric graph that is not face-accessible (the red face is not accessible).}
        \label{fig:symnotacc}
    \end{minipage}
\end{figure}

A graph $\Gamma$ with vertices in $\Z^n$ is called $\Z^n$-generating if the set of vectors $\{d(e) - s(e) \mid e \in E(\Gamma)\}$ generates $\Z^n$ as a semigroup.
If $\Gamma$ is symmetric, this semigroup is a group.
Hence, a full-image symmetric $\mG$-graph is $\Z^n$-generating if and only if $\{a_1, \ldots, a_K\}$ generates $\Z^n$ as a group, which is true by the assumption in Theorem~\ref{thm:tec}.

The main theorem of this subsection is the following, it shows that face-accessibility can characterize connectivity up to taking a finite union of translations. 

\begin{restatable}{thrm}{thmacctoeul}\label{thm:acctocon}
Let $\Gamma$ be a $\mG$-graph that is symmetric, face-accessible and $\Z^n$-generating.
Then there exist $z_1, \ldots, z_m \in \Z^n$, such that the union of translations $\hG \coloneqq \bigcup_{i = 1}^m \left( \Gamma + z_i \right)$ is an Eulerian graph.
\end{restatable}

See Figures~\ref{fig:acctocon} and \ref{fig:acctocon2} for an illustration of Theorem~\ref{thm:acctocon}.
The proof of Theorem~\ref{thm:acctocon} uses a sophisticated combination of convex geometry and graph theory, and will be given in Section~\ref{sec:graph}.
Note that the face-accessibility condition in Theorem~\ref{thm:acctocon} is necessary. Figures~\ref{fig:notcon} and \ref{fig:notcon2} show that Theorem~\ref{thm:acctocon} does not hold without face-accessibility.

\begin{figure}[h!]
    \centering
    \begin{minipage}[t]{.47\textwidth}
        \centering
        \includegraphics[width=0.5\textwidth,height=1.0\textheight,keepaspectratio, trim={5.0cm 0cm 4.5cm 0cm},clip]{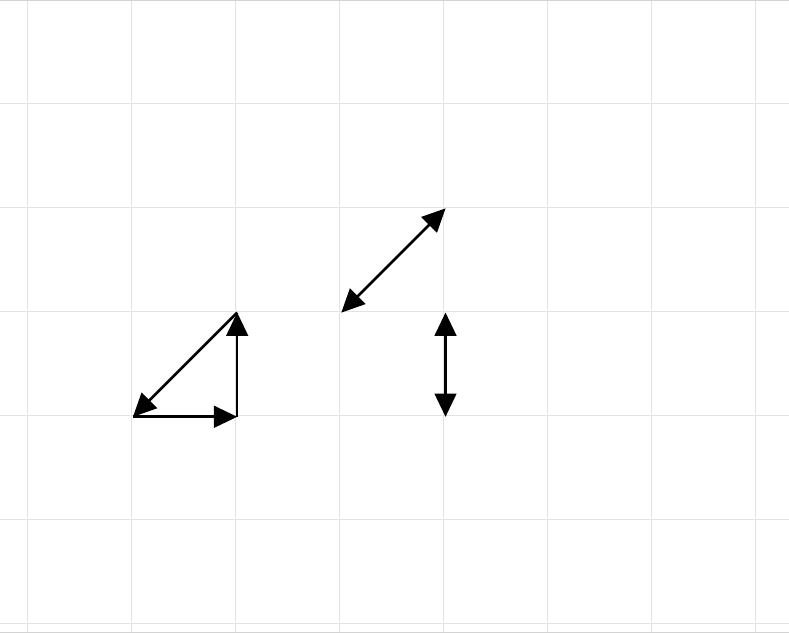}
        \caption{A symmetric face-accessible and $\Z^n$-generating graph $\Gamma$ from Theorem~\ref{thm:acctocon}.}
        \label{fig:acctocon}
    \end{minipage}
    \hfill
    \begin{minipage}[t]{0.47\textwidth}
        \centering
        \includegraphics[width=0.5\textwidth,height=1.0\textheight,keepaspectratio, trim={5.0cm 0cm 4.5cm 0cm},clip]{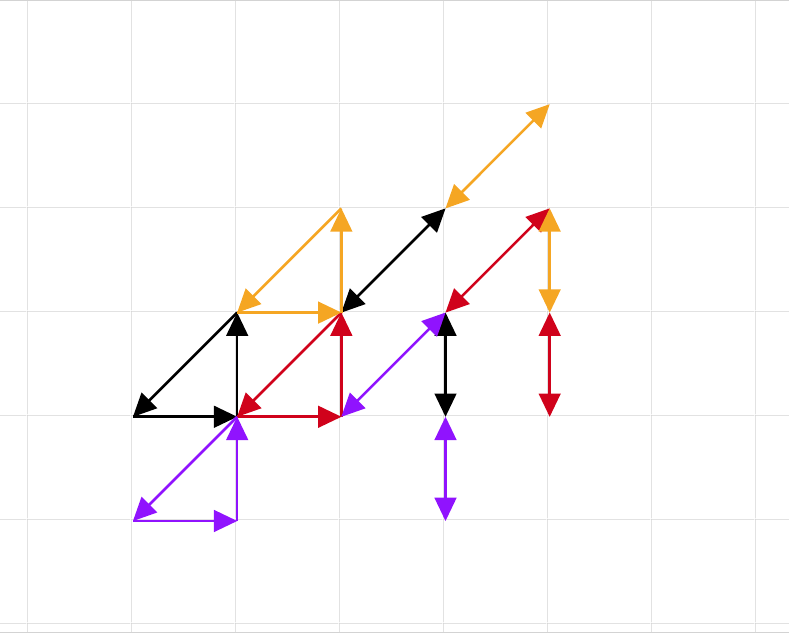}
        \caption{An Eulerian graph $\hG$ constructed in Theorem~\ref{thm:acctocon}, consisting of four translations of $\Gamma$, each noted with a different colour.}
        \label{fig:acctocon2}
    \end{minipage}

    \vspace{0.5cm}
    \begin{minipage}[t]{.47\textwidth}
        \centering
        \includegraphics[width=0.5\textwidth,height=1.0\textheight,keepaspectratio, trim={5.0cm 1cm 4.5cm 1cm},clip]{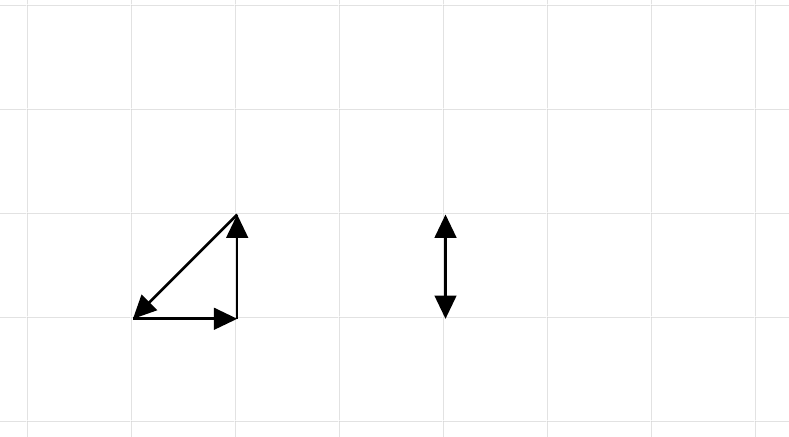}
        \caption{A $\Gamma$ that is not face-accessible.}
        \label{fig:notcon}
    \end{minipage}
    \hfill
    \begin{minipage}[t]{0.47\textwidth}
        \centering
        \includegraphics[width=0.5\textwidth,height=1.0\textheight,keepaspectratio, trim={5.0cm 1cm 4.5cm 1cm},clip]{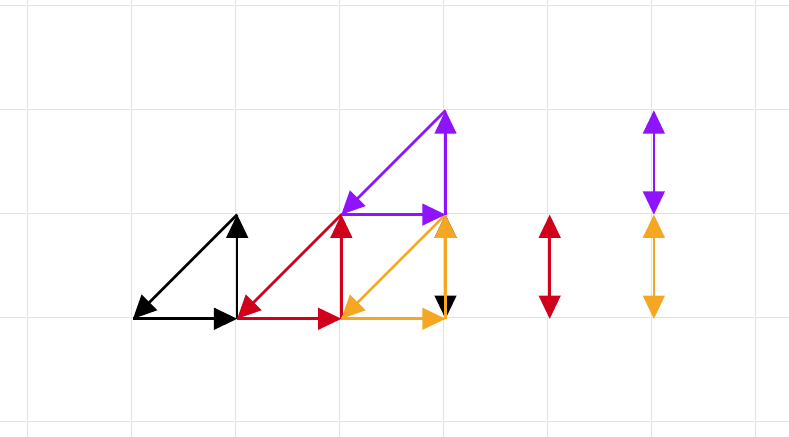}
        \caption{A union of translations of $\Gamma$ cannot be connected.}
        \label{fig:notcon2}
    \end{minipage}
\end{figure}

If $\Gamma$ represents the neutral element, then the union $\hG \coloneqq \bigcup_{i = 1}^m \left( \Gamma + z_i \right)$ represents the element $(\sum_{i = 1}^m \oX^{z_i} \cdot 0 , \sum_{i = 1}^m 0) = (0, 0)$.
Therefore, from Lemma~\ref{lem:grptoeul} and Theorem~\ref{thm:acctocon} we immediately obtain the following.

\begin{restatable}{prop}{propgtoe}\label{prop:grouptoeuler}
The semigroup $\sgmG$ is a group if and only if there exists a full-image symmetric face-accessible $\mG$-graph that represents the neutral element.
\end{restatable}

\subsection{From face-accessible graphs to positive polynomials}\label{subsec:graphtopoly}
In this subsection we introduce ``position polynomials'' to describe a $\mG$-graph, and reduce our graph theory problem to a computational problem over polynomials.
Recall that for an edge $e$ in a $\mG$-graph, $\ell(e)$ denotes the label of $e$, and $s(e) \in \Z^n$ denotes the starting vertex of $e$.
Given a $\mG$-graph $\Gamma$, define its tuple of \emph{position polynomials} $\bff = (f_1, \ldots, f_K)$ in the following way:
\begin{equation}\label{eq:pospoly}
f_i \coloneqq \sum_{e \in E(\Gamma), \ell(e) = i} \oX^{s(e)}, \quad i = 1, \ldots, K.
\end{equation}
That is, $f_i$ is the sum of monomials $\oX^{s}$, where $s$ ranges over the starting vertex of all label $i$ edges in $\Gamma$.
These polynomials have only non-negative coefficients, hence are in $\N[\oX^{\pm}]$. See Figure~\ref{fig:pospoly}.

\begin{figure}[h!]
    \centering
    \begin{minipage}[t]{.47\textwidth}
        \centering
        \includegraphics[width=0.5\textwidth,height=1.0\textheight,keepaspectratio, trim={6.3cm 0.8cm 5cm 0cm},clip]{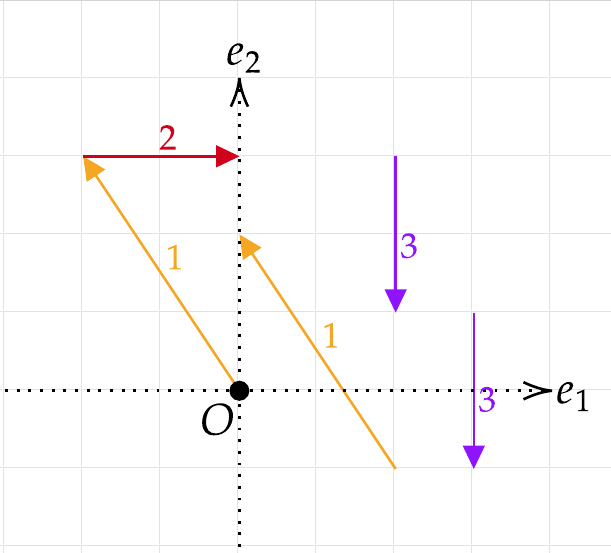}
        \caption{In this example, $f_1 = 1 + X_1^2 X_2^{-1}$, $f_2 = X_1^{-2} X_2^{3}$, $f_3 = X_1^{2} X_2^{3} + X_1^{3} X_2$.}
        \label{fig:pospoly}
    \end{minipage}
    \hfill
    \begin{minipage}[t]{0.47\textwidth}
        \centering
        \includegraphics[width=0.5\textwidth,height=1.0\textheight,keepaspectratio, trim={6.3cm 0.8cm 5cm 0cm},clip]{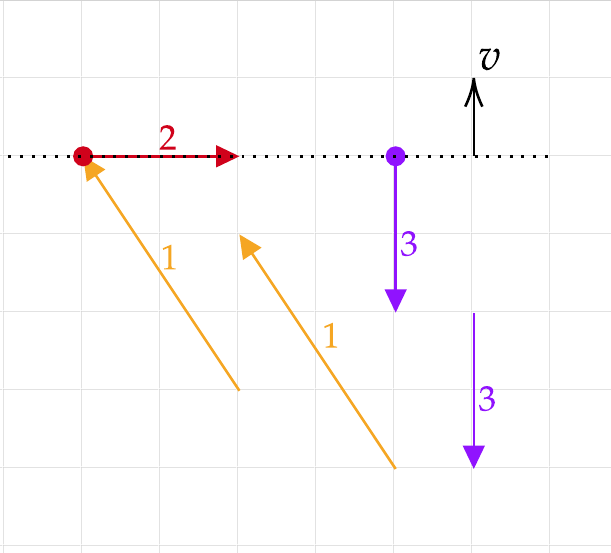}
        \caption{Let $v = (0, 1)$, we have \\ $M_{v}(\{1, 2, 3\}, \bff) = \{2, 3\}$, $O_{v} = \{1, 3\}$.}
        \label{fig:pospoly2}
    \end{minipage}
\end{figure}

Conversely, given any tuple of polynomials $\bff = (f_1, \ldots, f_K) \in \N[\oX^{\pm}]^K$, one can construct a $\mG$-graph $\Gamma$ such that $\bff$ is exactly its position polynomials.
Indeed, for each monomial $c X^{b}$ of $f_i$, we draw $b$ edges of label $i$ starting at vertex $b$.
Note that it is crucial for $f_i$ to be an element of $\N[\oX^{\pm}]$ instead of $\Z[\oX^{\pm}]$, so that all monomials have non-negative coefficients.

For $v \in \Rns$ and $I \subseteq \{1, \ldots, K\}$, define
\[
M_{v}(I, \bff) \coloneqq \left\{i \in I \;\middle|\; \deg_{v}(f_i) = \max_{i' \in I}\{\deg_{v}(f_{i'})\} \right\}.
\]
This is the set of indices $i \in I$ such that $\deg_{v}(f_i)$ is maximal among $i \in I$.
Define
\[
O_{v} \coloneqq \{i \in \{1, \ldots, K\} \mid a_i \not\perp v\}.
\]
This is the set of indices $i \in \{1, \ldots, K\}$ such that $a_i$ is not orthogonal to $v$.
See Figure~\ref{fig:pospoly2}.

The following proposition shows we can completely characterize the graph theoretic properties from Proposition~\ref{prop:grouptoeuler} using position polynomials.
The key point is how we characterize face-accessibility in (iii).
As a comparison, no good characterization of \emph{connectivity} can be obtained from position polynomials.

\begin{restatable}{prop}{propeulertoeq}\label{prop:eulertoeq}
Let $\Gamma$ be a $\mG$-graph with position polynomials $f_i \in \N[\oX^{\pm}], i = 1, \ldots, K$.
\begin{enumerate}[noitemsep, label = (\roman*)]
    \item $\Gamma$ is full-image if and only if 
    $
        f_i \neq 0 \text{ for } i = 1, \ldots, K
    $.
    \item $\Gamma$ is symmetric if and only if
    $
        \sum_{i = 1}^K f_i \cdot (\oX^{a_i} - 1) = 0
    $.
    \item $\Gamma$ is face-accessible if and only if
    \begin{equation}\label{eq:acccond}
        O_{v} \cap M_{v}(\{1, \ldots, K\}, \bff) \neq \emptyset \quad \text{ for every } v \in \Rns.
    \end{equation}
    \item Suppose $\Gamma$ is symmetric. $\Gamma$ represents the neutral element if and only if
    $
        \sum_{i = 1}^K f_i \cdot y_i = 0
    $.
\end{enumerate}
\end{restatable}

\begin{proof}
    (i) and (iv) are easy and follow by direct computation using the definition~\eqref{eq:pospoly} of position polynomials (see Appendix~\ref{app:main} for details).

    (ii) We have 
    \begin{multline*}
        \sum_{i = 1}^K f_i \cdot (\oX^{a_i} - 1) = \sum_{i = 1}^K \sum_{e \in E(\Gamma), \ell(e) = i} \oX^{s(e)} (\oX^{a_i} - 1)
        = \sum_{i = 1}^K \sum_{e \in E(\Gamma), \ell(e) = i} (\oX^{s(e) + a_i} - \oX^{s(e)}) \\
        = \sum_{e \in E(\Gamma)} (\oX^{d(e)} - \oX^{s(e)}) = \sum_{e \in E(\Gamma)} \oX^{d(e)} - \sum_{e \in E(\Gamma)} \oX^{s(e)}.
    \end{multline*}
    This is equal to zero if and only if the in-degree equals the out-degree at every vertex.

    (iii) Let $C$ be the convex hull of $V(\Gamma)$. 
    For every strict face $F$ of $C$ there is a vector $v \in \Rns$ such that $F$ consists of all points $x$ in $C$ where $v \cdot  x$ is maximal.
    Conversely, for every vector $v \in \Rns$, the set $F$ of all points $x$ in $C$ such that $v \cdot  x$ is maximal forms a strict face of $C$.

    Let $F$ be a strict face, then $F$ is accessible if and only if some edge starting in $F$ does not end in $F$.
    Let $e$ be an edge starting in $F$, with label $\ell(e)$.
    Then $v \cdot  s(e)$ is maximal among all $e \in E(\Gamma)$.
    Since the monomial $\oX^{s(e)}$ is contained in $f_{\ell(e)}$, this means $\ell(e) \in M_{v}(\{1, \ldots, K\}, \bff)$.
    
    Observe that $d(e) \in F$ if and only if $a_{\ell(e)} = d(e) - s(e)$ is orthogonal to $v$, which is equivalent to $\ell(e) \not\in O_{v}$.
    Therefore, $F$ is accessible if and only if an edge $e$ exists such that $\ell(e) \in M_{v}(\{1, \ldots, K\}, \bff)$ and $\ell(e) \in O_{v}$;
    that is, $O_{v} \cap M_{v}(\{1, \ldots, K\}, \bff) \neq \emptyset$.

    By the definition of face-accessibility, $\Gamma$ is face-accessible if and only if $O_{v} \cap M_{v}(\{1, \ldots, K\}, \bff) \neq \emptyset$ holds for every $v \in \Rns$.
\end{proof}

Let $\mM_{\Z}$ be the $\Z[\oX^{\pm}]$-module consisting of all $\bff \in \Z[\oX^{\pm}]^K$ satisfying $\sum_{i = 1}^K f_i \cdot (\oX^{a_i} - 1) = 0$ and $\sum_{i = 1}^K f_i \cdot y_i = 0$:
\begin{equation}\label{eq:MZ}
\mM_{\Z} \coloneqq \left\{\bff \in \Z[\oX^{\pm}]^K \;\middle|\; \sum_{i = 1}^K f_i \cdot (\oX^{a_i} - 1) = 0 \text{ and } \sum_{i = 1}^K f_i \cdot y_i = 0 \right\}.
\end{equation}
Then Proposition~\ref{prop:eulertoeq} shows the following: there exists a full-image symmetric face-accessible $\mG$-graph that represents the neutral element if and only if $\mM_{\Z}$ contains an element $\bff \in \left(\N[\oX^{\pm}]^*\right)^K$ satisfying Property~\eqref{eq:acccond}.
Using linear algebra over $\Z[\oX^{\pm}]$, generators of $\mM_{\Z}$ can be effectively computed (a simple proof is given in Appendix~\ref{app:main}):

\begin{restatable}{lem}{lemMZ}\label{lem:MZ}
    A finite set of generators $\bg_1, \ldots, \bg_m \in \Z[\oX^{\pm}]^K$ of the $\Z[\oX^{\pm}]$-module $\mM_{\Z}$ can be computed from $y_i, a_i, i = 1, \ldots, K$.
\end{restatable}

\subsection{Local-global principle for positive polynomials}\label{subsec:locglob}
We now start to construct an algorithm that decides whether $\mM_{\Z}$ contains an element $\bff \in \left(\N[\oX^{\pm}]^*\right)^K$ satisfying Property~\eqref{eq:acccond}.
This problem is highly non-trivial due to the polynomials having coefficients in $\N$ instead of $\Z$.
In fact, solving linear equations over the semiring $\N[\oX^{\pm}]$ is known to be \emph{undecidable}~\cite{narendran1996solving}.
Our key to obtaining a decidability result is to exploit the \emph{homogeneity} of our linear equations.

The first step is to generalize Property~\eqref{eq:acccond}.
Given two sets $I, J \subseteq \{1, \ldots, K\}$, our new goal is to decide whether $\mM_{\Z}$ contains an element $\bff \in \left(\N[\oX^{\pm}]^*\right)^K$ satisfying the following condition.
\begin{equation}\label{eq:gen}
        \left(O_{v} \cup J\right) \cap M_{v}(I, \bff) \neq \emptyset, \quad \text{ for every } v \in \Rns.
\end{equation}
Note that Property~\eqref{eq:acccond} can be considered as a special case of Property~\eqref{eq:gen} with $I = \{1, \ldots, K\}, J = \emptyset$.
This generalization will be crucial to our subsequent decidability result.
Intuitively, edges with labels in $J$ can be considered to be ``going out into an $(n+1)$-th dimension''; and edges with labels outside of $I$ can be considered to ``exist in an $(n+1)$-th dimension''.

The second step is to pass from polynomial rings over $\Z$ to polynomial rings over $\R$ in order to facilitate subsequent usage of analytic methods.
Let $\mM$ be the $\R[\oX^{\pm}]$-submodule of $\R[\oX^{\pm}]^K$ generated by $\bg_1, \ldots, \bg_m$ from Lemma~\ref{lem:MZ}, that is,
\[
\mM \coloneqq \left\{\sum_{j = 1}^m h_j \cdot \bg_j \;\middle|\; h_1, \ldots, h_m \in \R[\oX^{\pm}]\right\}.
\]
\begin{restatable}{lem}{lemM}\label{lem:M}
    There exists an element $\tbf \in \mM_{\Z} \cap \left(\N[\oX^{\pm}]^*\right)^K$ satisfying Property~\eqref{eq:gen}, if and only if there exists $\bff \in \mM \cap \left(\Rp[\oX^{\pm}]^*\right)^K$ satisfying Property~\eqref{eq:gen}.
\end{restatable}

Denote $\A \coloneqq \R[\oX^{\pm}], \A^+ \coloneqq \Rp[\oX^{\pm}]^*$.
Given $f \in \A$ and $v \in \Rns$, the \emph{initial polynomial} of $f$ is defined as the sum of all monomials in $f$ having the maximal degree $\deg_v(\cdot)$:
\[
\init_v(f) \coloneqq \sum\nolimits_{\deg_v(\oX^b) = \deg_v(f)} c_b \oX^b, \quad \text{ where } f = \sum c_b \oX^b.
\]
For $\bff = (f_1, \ldots, f_K) \in \A^K$, we naturally denote $\init_v(\bff) \coloneqq (\init_v(f_1), \ldots, \init_v(f_K)) \in \A^K$.
The key result of this subsection is the following local-global principle, which simultaneously generalizes two deep results of Einsiedler, Mouat and Tuncel~\cite[Theorem~1.3]{einsiedler2003does} and of Dong~\cite[Proposition~3.4]{dong2023identity}.

\begin{restatable}{thrm}{thmlocglob}\label{thm:locglob}
Let $\mM$ be a $\A$-submodule of $\A^K$ and $I, J$ be two subsets of $\{1, \ldots, K\}$.
There exists $\bff \in \mM \cap \ApK$ satisfying 
\begin{equation}\label{eq:globv}
        \left(O_{v} \cup J\right) \cap M_{v}(I, \bff) \neq \emptyset, \quad \text{ for every } v \in \Rns,
\end{equation}
if and only if the two following conditions are satisfied:
\begin{enumerate}[noitemsep, label = \arabic*.]
    \item \label{item:locr} \emph{\textbf{(LocR):}} For every $r \in \Rpp^n$, there exists $\bff_{r} \in \mM$ such that $\bff_{r}(r) \in \Rpp^K$.
    \item \label{item:locinf} \emph{\textbf{(LocInf):}} For every $v \in \Rns$, there exists $\bff_{v} \in \mM$, such that 
    \begin{enumerate}[noitemsep]
        \item $\init_{v}\left(\bff_{v}\right) \in \ApK$.
        \item Denote $I' \coloneqq M_{v}(I, \bff_{v}), J' \coloneqq O_{v} \cup J$.
        We have 
        \begin{equation}\label{eq:locv}
        \left(O_{w} \cup J'\right) \cap M_{w}(I', \init_{v}(\bff_{v})) \neq \emptyset \quad \text{ for every $w \in \Rns$}.
        \end{equation}
    \end{enumerate}
\end{enumerate}
\end{restatable}

The full proof of Theorem~\ref{thm:locglob} is highly non-trivial and is given in Section~\ref{sec:locglob}.
Compared to the cited results~\cite{dong2023identity, einsiedler2003does}, the new element here is the inclusion of Property~\eqref{eq:globv}.
This property is crucial due to our characterization of face-accessibility.
The main difficulty in this generalization is the complex interaction between the sets $M_v, O_v$ and $\bff$, which is absent from the two cited results.

\subsection{Decidability}\label{subsec:dec}
Theorem~\ref{thm:locglob} is the key to an algorithm that finds $\bff \in \mM \cap \ApK$ satisfying Property~\eqref{eq:gen}.
Indeed, we have:
\begin{restatable}{thrm}{thmdec}\label{thm:dec}
Fix $n \in \N$.
Suppose we are given as input a set of generators $\bg_1, \ldots, \bg_m \in \A^K$ with integer coefficients, as well as the vectors $a_1, \ldots, a_K \in \Z^n$ and two subsets $I, J$ of $\{1, \ldots, K\}$.
Denote by $\mM$ be the $\A$-submodule of $\A^K$ generated by $\bg_1, \ldots, \bg_m$.
It is decidable whether there exists $\bff \in \mM \cap \ApK$ satisfying
\begin{equation}\label{eq:deccond}
        \left(O_{v} \cup J\right) \cap M_{v}(I, \bff) \neq \emptyset, \quad \text{ for every } v \in \Rns.
\end{equation}
Here, if $n = 0$ then $\A$ is understood as $\R$, and Property~\eqref{eq:deccond} is considered trivially true.
\end{restatable}
The proof of Theorem~\ref{thm:dec} will be given Section~\ref{sec:dec}.
The main idea is as follows.
We use an induction on the number of variables $n$.
Theorem~\ref{thm:locglob} allows us to reduce the decision problem into verifying two conditions \hyperref[item:locr]{(LocR)} and \hyperref[item:locinf]{(LocInf)}.
Condition~\hyperref[item:locr]{(LocR)} can be decided using the first order theory of reals.
The key part is to show that it suffices to decide Condition~\hyperref[item:locinf]{(LocInf)} for \emph{countably} many $v$.
For each $v$ we can decide \hyperref[item:locinf]{(LocInf)} by a clever application of the induction hypothesis.
We then run two parallel procedures, one enumerates all elements in $\mM_{\Z}$ and checks if any one of them is in $\ApK$ and satisfies Property~\eqref{eq:deccond}, the other enumerates countably many $v \in \Rns$ and checks if Condition~\hyperref[item:locinf]{(LocInf)} is false.
Theorem~\ref{thm:locglob} guarantees that one of the two procedures must terminate. 

Putting together Lemma~\ref{lem:subsume}, Proposition~\ref{prop:grouptoeuler}, Proposition~\ref{prop:eulertoeq}, Lemma~\ref{lem:MZ}, Lemma~\ref{lem:M} and Theorem~\ref{thm:dec}, we obtain our main technical result:

\thmtec*
\begin{proof} 
    Note that tuples of polynomials in $\N[\oX^{\pm}]^K$ have one-to-one correspondence with $\mG$-graphs.
    Therefore, Proposition~\ref{prop:grouptoeuler} and \ref{prop:eulertoeq} show it suffices to decide whether the module $\mM_{\Z}$ (defined in \eqref{eq:MZ}) contains an element $\bff \in \left(\N[\oX^{\pm}]^*\right)^K$ satisfying Property~\eqref{eq:acccond}.
    We use Lemma~\ref{lem:MZ} to compute a basis of $\mM_{\Z}$. Lemma~\ref{lem:M} then shows it suffices to decide whether there exists $\bff \in \mM \cap \left(\Rp[\oX^{\pm}]^*\right)^K$ satisfying Property~\eqref{eq:acccond}.
    But Property~\eqref{eq:acccond} is simply Property~\eqref{eq:globv} with $I = \{1, \ldots, K\}, J = \emptyset$.
    So Theorem~\ref{thm:dec} shows this is decidable.
\end{proof}

Our main result follows from this technical result.

\thmmain*
\begin{proof}
    Let $G$ be a finitely generated metabelian group.
    By Lemma~\ref{lem:subsume} it suffices to decide the Group Problem in $G$.
    Given a finite subset $\mG$ in $G$, we use Proposition~\ref{prop:metatoZ} to construct a a finitely presented $\Z[X_1^{\pm}, \ldots, X_n^{\pm}]$-module $\mY$ for some $n \in \N$, as well as a subset $\widetilde{\mG}$ of the group $\mY \rtimes \Z^n$, such that $\sgmG$ is a group if and only if $\langle \widetilde{\mG} \rangle$ is a group.
    Furthermore, the constructed set $\widetilde{\mG}$ satisfies $\pi(\langle \widetilde{\mG} \rangle_{grp}) = \Z^n$ under the canonical projection $\pi \colon \mY \rtimes \Z^n \rightarrow \Z^n$.
    Theorem~\ref{thm:tec} shows we can decide whether $\langle \widetilde{\mG} \rangle$ is a group.
    Therefore, it is decidable whether $\sgmG$ is a group.
\end{proof}

\section{From face-accessibility to connectivity}\label{sec:graph}
In this section we give the proof of Theorem~\ref{thm:acctocon}.
Let $\Gamma$ be a $\mG$-graph that is symmetric, face-accessible and $\Z^n$-generating.
Note that in a symmetric graph, there exists a path from vertex $v$ to $w$ if and only if there exists a path from $w$ to $v$.
Therefore it will suffice to proof connectivity for the undirected version of the graph $\hG$.

Let $x$ be an arbitrary point in $\R^n$.
Given $c \in \R^n$ and $r \in \Rpp$, denote by $scale(x, c, r)$ the \emph{scaling} of $x$ with centre $c$ by the ratio $r$.
That is,
\[
scale(x, c, r) \coloneqq c + r \cdot (x - c).
\]
Let $S$ be an arbitrary set in $\R^n$, define 
\[
scale(S, c, r) \coloneqq \{scale(x, c, r) \mid x \in S\}.
\]
When the centre is the origin $0$, we simplify the notation by defining
\[
NS \coloneqq scale(S, 0, N)
\]

Let $C$ be the convex hull of $V(\Gamma)$.
Since $\Gamma$ is $\Z^n$-generating, the polytope $C$ is of dimension $n$.
For any $N \in \N$, let $NC \coloneqq scale(C, 0, N)$.
Define 
\[
S_N \coloneqq \{z \in \Z^n \mid z + C \subset NC\}.
\]
That is, $S_N$ is the set of translation vectors $z$ that make $C + z$ stay in $NC$.
Consider the graph 
\[
\Gamma_N \coloneqq \sum_{z \in S_N} (\Gamma + z).
\]
We have $V(\Gamma_N) \subseteq NC$.
Intuitively, $\Gamma_N$ is the union of translations of the graph $\Gamma$ whose convex hull is contained in $NC$.
See Figure~\ref{fig:Gamma2} and \ref{fig:Gamma22} for an illustration.
Our goal is to prove that for some large $N$, the graph $\Gamma_N$ is connected.

\begin{figure}[h!]
    \centering
    \begin{minipage}[t]{.47\textwidth}
        \centering
        \includegraphics[width=1\textwidth,height=0.6\textheight,keepaspectratio, trim={3cm 0cm 3cm 0cm},clip]{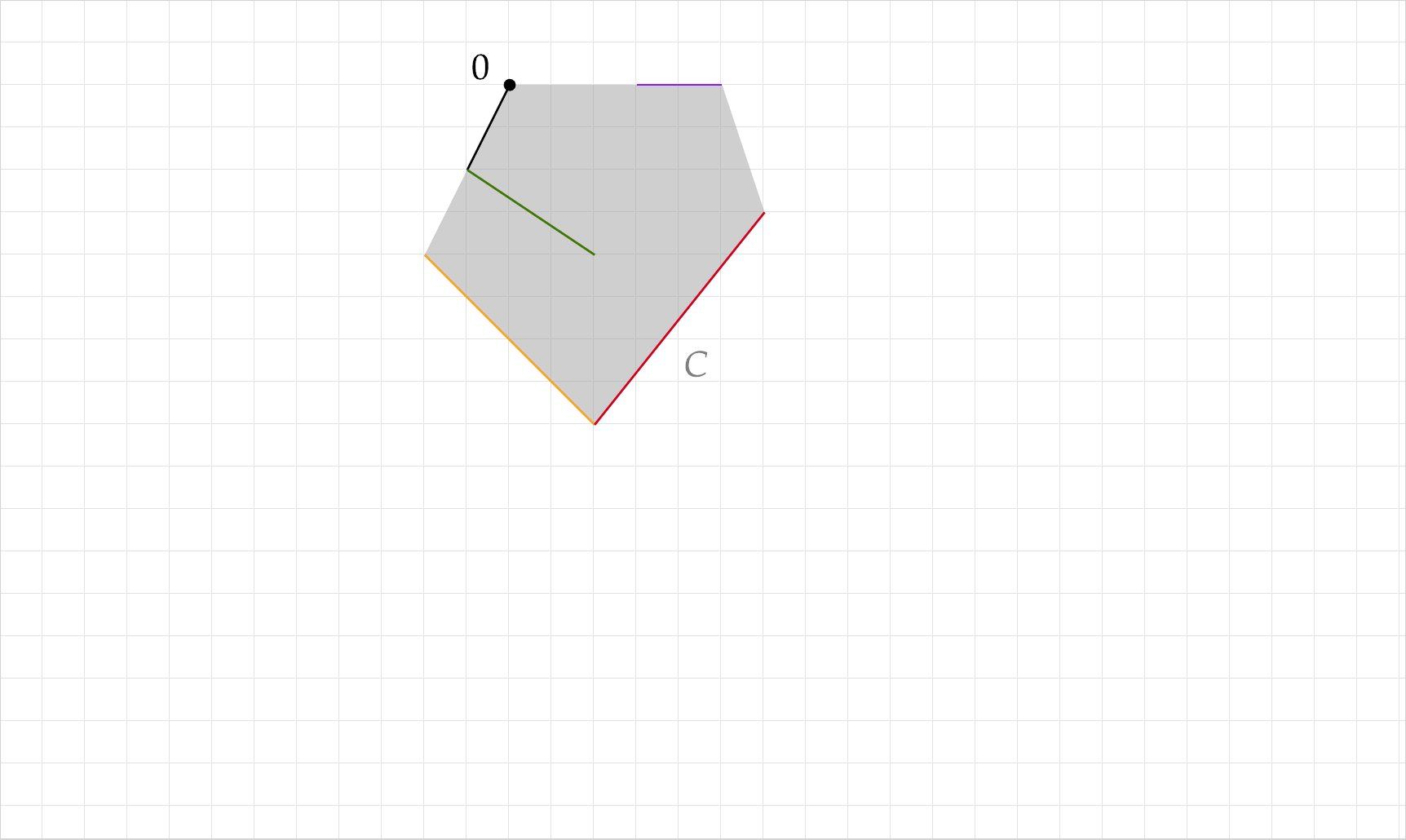}
        \caption{A graph $\Gamma$ with convex hull $C$ covered in grey, and each edge denoted with a different colour.}
        \label{fig:Gamma2}
    \end{minipage}
    \hfill
    \begin{minipage}[t]{0.47\textwidth}
        \centering
        \includegraphics[width=1\textwidth,height=0.6\textheight,keepaspectratio, trim={3cm 0cm 3cm 0cm},clip]{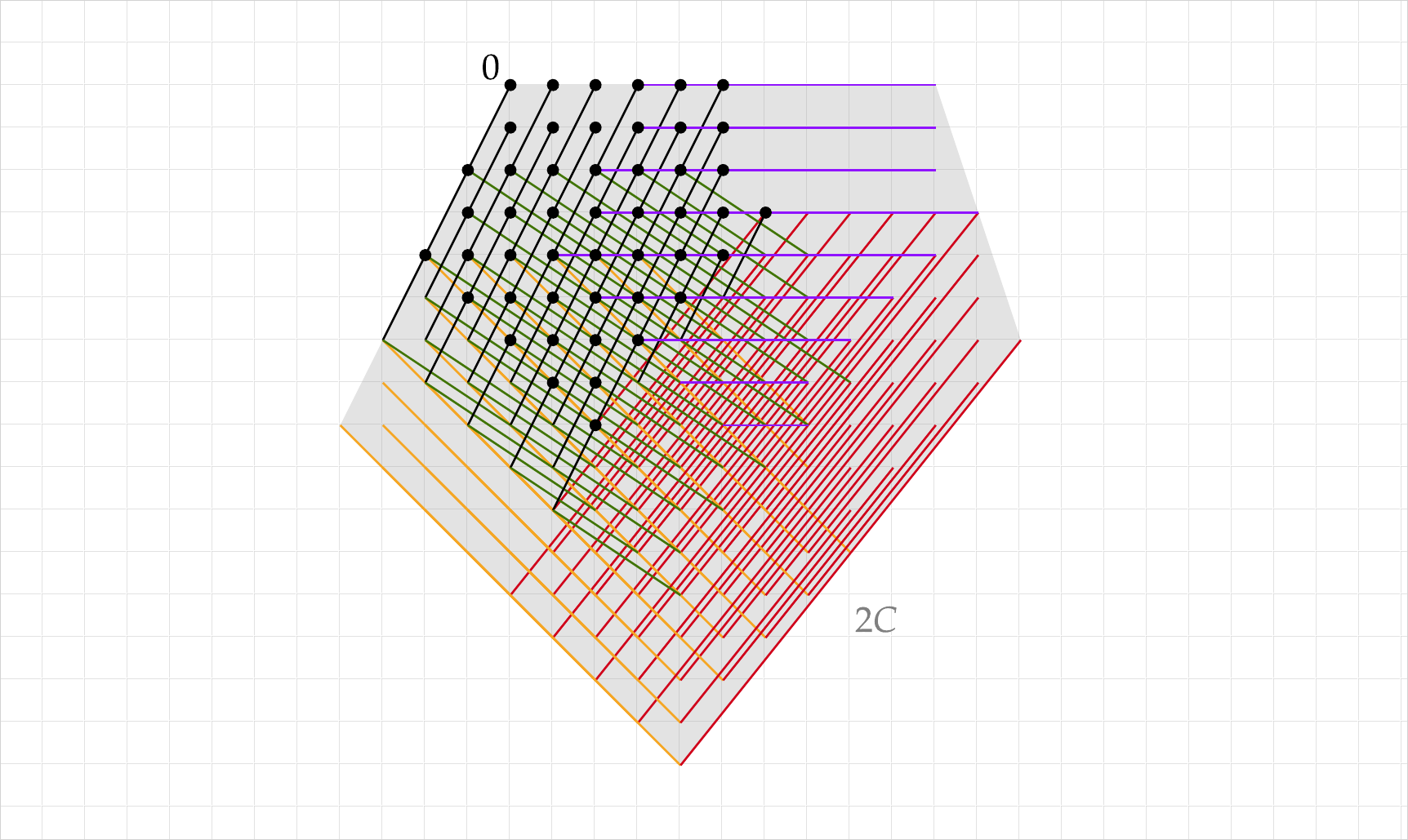}
        \caption{The graph $\Gamma_{N}$ with $N = 2$, consisting of translations of $\Gamma$. The polytope $2C$ is covered in grey.}
        \label{fig:Gamma22}
    \end{minipage}
\end{figure}

First we define the (infinite) graph $\gq$ as follows.
The vertices of $\gq$ are $V(\gq) \coloneqq C \cap \Q^n$. 
The edges of $\gq$ are
\[
E(\gq) \coloneqq \{scale(e, c, r) \mid e \in E(\Gamma), c \in C, r \in (0, 1) \cap \Q\}.
\]
That is, $\gq$ is the union of all scaled versions of $\Gamma$ that completely falls inside $C$.
Intuitively, $\gq$ can be seen as the ``limit'' of $\Gamma_N$ when $N$ tends towards infinity. 
See Figure~\ref{fig:GammaQ} for an illustration.

\begin{figure}[h!]
    \centering
    \begin{minipage}[t]{.47\textwidth}
        \centering
        \includegraphics[width=1\textwidth,height=1.0\textheight,keepaspectratio, trim={2cm 0cm 2cm 0cm},clip]{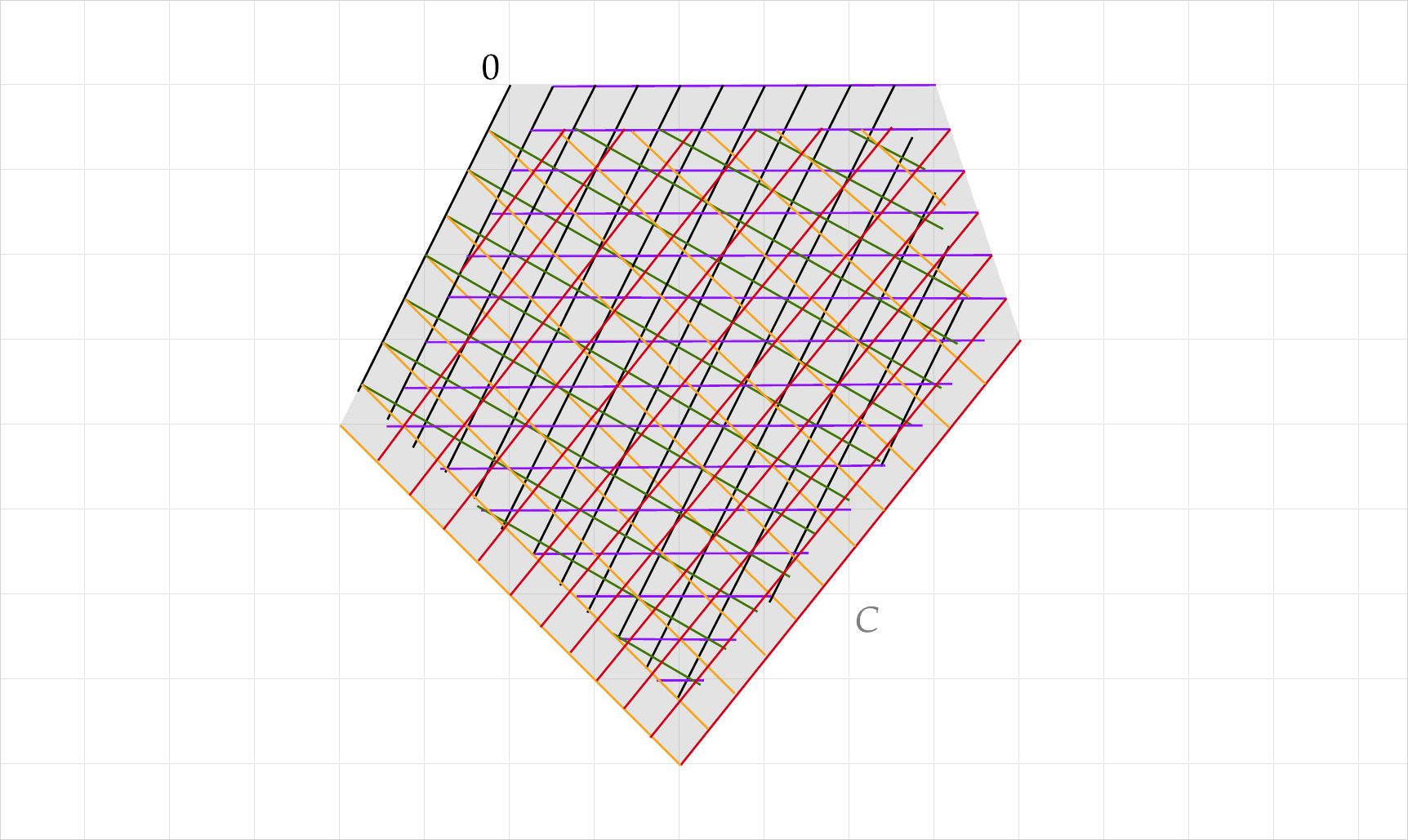}
        \caption{Illustration for $\Gamma_{\Q}$, where $\Gamma$ is as in Figure~\ref{fig:Gamma2}.}
        \label{fig:GammaQ}
        \end{minipage}
    \hfill
    \begin{minipage}[t]{0.47\textwidth}
        \centering
        \includegraphics[width=1\textwidth,height=0.6\textheight,keepaspectratio, trim={0cm 0cm 0cm 0cm},clip]{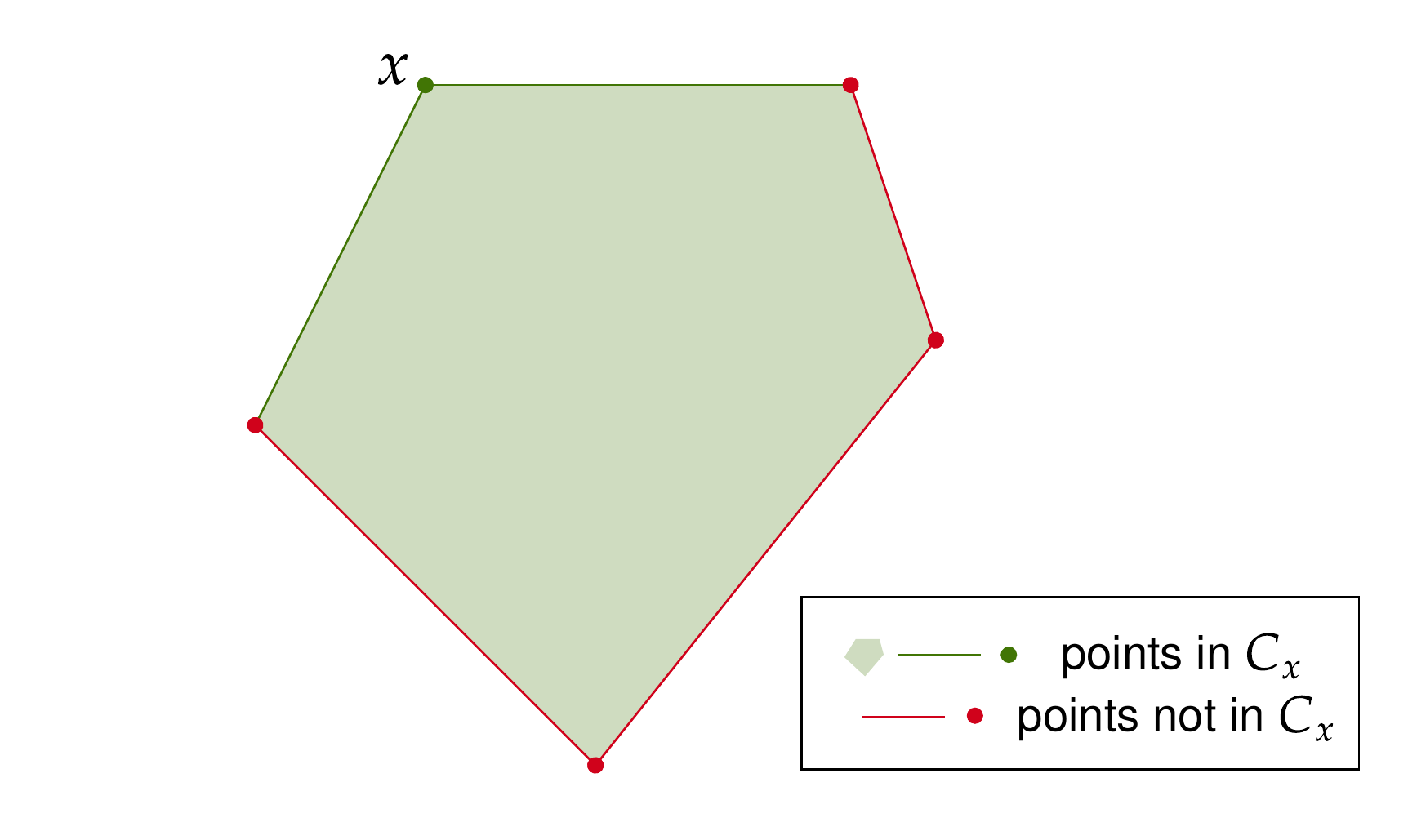}
        \caption{Illustration of the set $C_x$.}
        \label{fig:Cx}
    \end{minipage}
\end{figure}

Define the \emph{face lattice} $\lat(C)$ to be the set of all faces of $C$.
For $F \in \lat(C)$, its \emph{relative interior} $\inte(F)$ is the set of points in $F$ that are not contained in any sub-face of $F$:
\[
\inte(F) \coloneqq \{x \in F \mid \text{ for all faces } F' \subsetneq F, x \not\in F'\}.
\]
Then the relative interiors of faces constitute a partition of $C$:
\[
C = \bigcup_{F \in \lat(C)} \inte(F).
\]

For any point $x \in C$, define $F_x$ to be the face of $C$ such that $x \in \inte(F_x)$.
This is the smallest face containing $x$.
For any face $F$of $C$, we have $x \in F$ if and only if $F_x \subseteq F$.
Define 
\[
C_x \coloneqq \bigcup_{F \in \lat(C), F \supseteq F_x} \inte(F).
\]
That is, the set $C_x$ is the union of the interior of all faces containing $F_x$.
See Figure~\ref{fig:Cx} for an illustration.
This is also the union of the interior of all faces containing $x$.
If $y \in C_x$, then $F_y \supseteq F_x$, so $C_y = \bigcup_{F \in \lat(C), F \subseteq F_y} \inte(F) \subseteq \bigcup_{F \in \lat(C), F \subseteq F_x} \inte(F) = C_x$.

\begin{lem}\label{lem:cx}
    For any $c \in C_x$ and $r \in (0,1)$, we have $scale(C, c, r) \subset C_x$.
\end{lem}
\begin{proof}
    Let $v \in C$, we show that $scale(v, c, r) \in C_x$.
    Since $c \in C_x = \bigcup_{F \in \lat(C), F \supseteq F_x} \inte(F)$, we have $\inte(F_c) \subset C_x$.
    Therefore $F_c \supseteq F_x$.

    Denote by $seg(c, v)$ the closed segment that connects $c$ and $v$, and by $\inte(seg(c, v)) \coloneqq seg(c, v) \setminus \{c, v\}$.
    Let $F$ be the smallest face containing the $seg(c, v)$, then $F \supseteq F_c \supseteq F_x$.
    Hence, $scale(v, c, r) \in \inte(seg(c, v)) \subset \inte(F) \subset C_x$.
\end{proof}

Since $C = \bigcup_{x \in C} C_x$, we have $E(\gq) = \bigcup_{x \in C} E_x(\gq)$, where
\[
E_x(\gq) \coloneqq \{scale(e, c, r) \mid e \in E(\Gamma), c \in C_x, r \in (0, 1) \cap \Q\}.
\]
Every edge in $E_x(\gq)$ is contained in $C_x$ by Lemma~\ref{lem:cx}.

\begin{lem}\label{lem:congq}
    Every vertex $x$ of $\gq$ is connected to $\inte(C)$ by a finite path $P_x$ consisting of edges in $E_x(\gq)$.
\end{lem}
\begin{proof}
    See Figure~\ref{fig:Px} for an illustration of the proof.
    We first show that $x$ is connected by an edge $e_x \in E_x(\gq)$ to some $x' \in C_x$ where $F_{x'} \supsetneq F_x$.

    Since $\Gamma$ is face-accessible, there exists an edge $e \in E(\Gamma)$ connecting $w \in F_x$ and $w' \in C \setminus F_x$.
    Since $x \in \inte(F_x)$ and $w \in F_x$, there exists $\varepsilon \in \Qpp$ such that $c \coloneqq scale(w, x, 1 + \varepsilon) \in \inte(F_x) \subseteq C_x$.
    Then $x = scale(w, c, \frac{\varepsilon}{1+\varepsilon})$, and $x' \coloneqq scale(w', c, \frac{\varepsilon}{1+\varepsilon}) \in C_x$ by Lemma~\ref{lem:cx}.
    We also have $x' \not\in F_x$ since $w' \not\in F_x$, so $F_{x'} \supsetneq F_x$.
    Therefore, the edge $e_x \coloneqq scale(e, c, \frac{\varepsilon}{1+\varepsilon})$ is in $E_x(\gq)$ and connects $x$ and $x'$.

    Note that $x' \in C_x$, so $C_{x'} \subseteq C_x$ and $E_{x'}(\gq) \subseteq E_x(\gq)$.
    Repeating this process for $x'$, we can find a sequence of edges $e_x, e_{x'}, \ldots,$ respectively in $E_x(\gq) \subseteq E_{x'}(\gq) \subseteq \cdots$ that gradually connects $x$ to the interiors of increasingly higher dimensional faces.
    Eventually $x$ is connected to $\inte(C)$ by a path $P_x$.
\end{proof}

\begin{figure}[ht]
    \centering
    \includegraphics[width=0.7\textwidth,height=1.0\textheight,keepaspectratio, trim={3cm 0.5cm 1cm 0.5cm},clip]{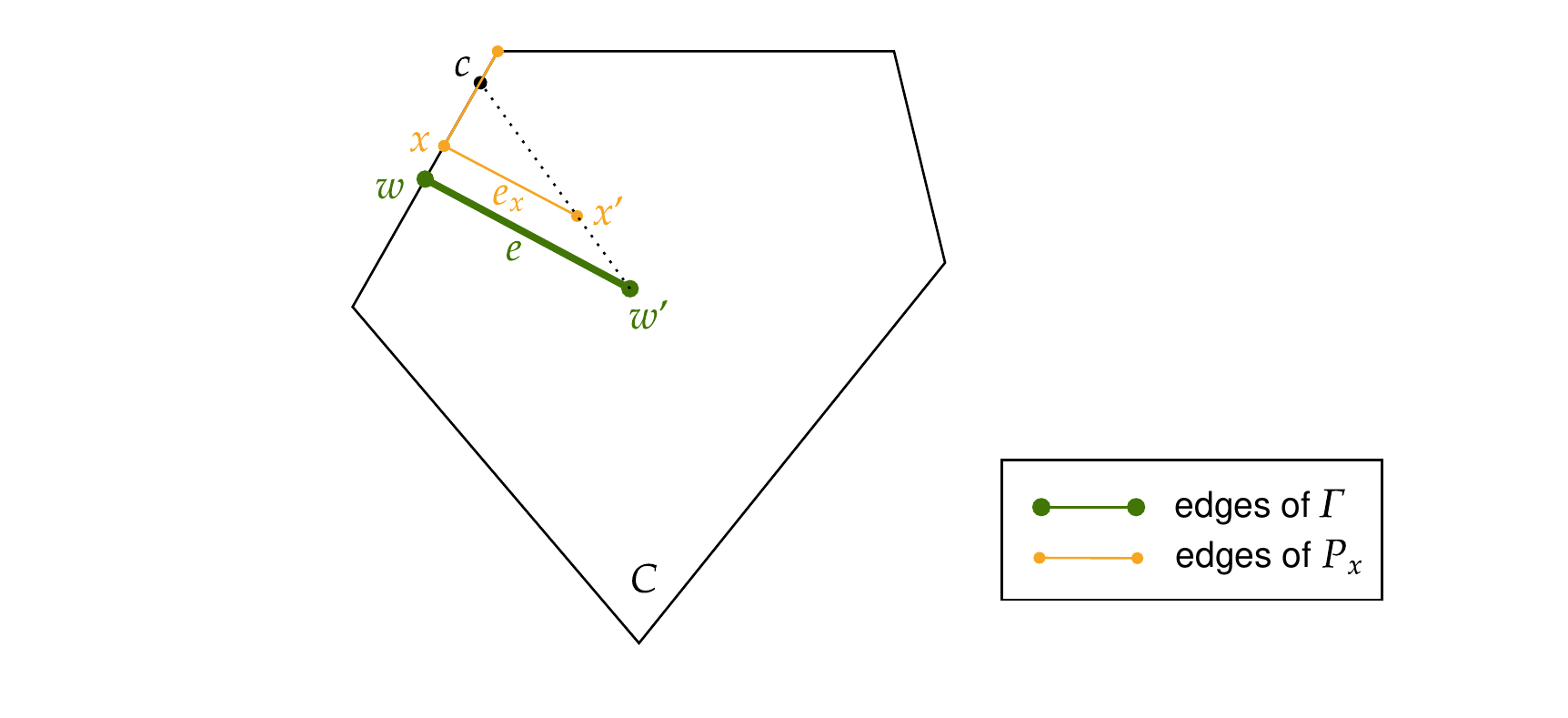}
    \caption{Illustration for Lemma~\ref{lem:congq}.}
    \label{fig:Px}
\end{figure}

Let $x$ be an arbitrary point in $C$.
For each edge $e'$ in $P_x$ (the path defined in Lemma~\ref{lem:congq}), write $e' = scale(e, c, r)$ where $e \in E(\Gamma), c \in C_x, r \in (0,1) \cap \Q$, define the polytope 
\begin{equation}\label{eq:Ce}
    C(e') \coloneqq scale(C, c, r).
\end{equation}
Then $e' \subset C(e') \subset C_x$ by Lemma~\ref{lem:cx}.
Therefore, defining the finite union of polytopes
\[
U_x \coloneqq \bigcup_{e' \in P_x} C(e'),
\]
we have $P_x \subset U_x \subset C_x$ and $U_x$ is compact.
See Figure~\ref{fig:Ux} for an illustration.

\begin{figure}[ht]
        \centering
        \includegraphics[width=0.5\textwidth,height=0.6\textheight,keepaspectratio, trim={3.5cm 0cm 3.5cm 0cm},clip]{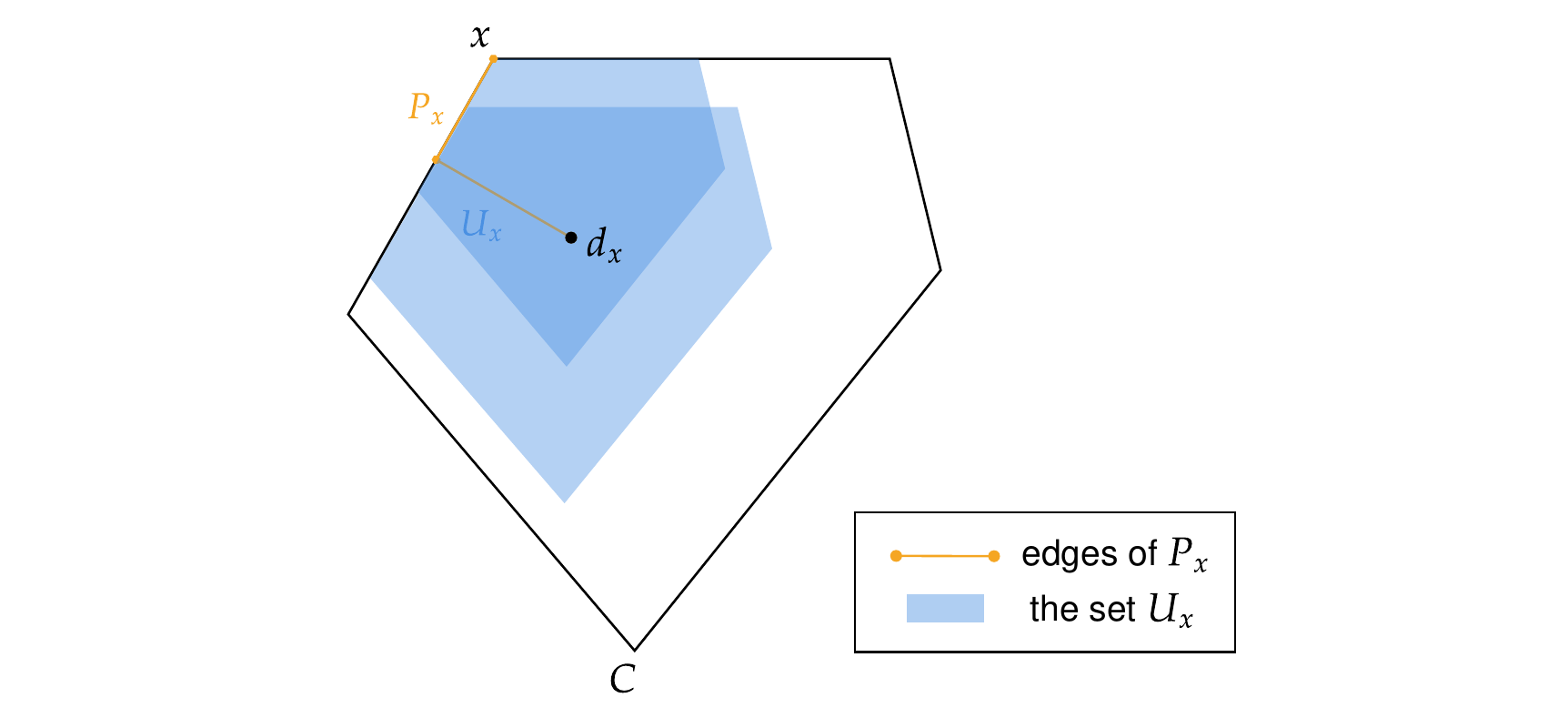}
        \caption{Illustration of $U_x$.}
        \label{fig:Ux}
\end{figure}

\begin{figure}[ht]
    \centering
    \begin{minipage}[t]{.47\textwidth}
        \centering
        \includegraphics[width=1\textwidth,height=0.6\textheight,keepaspectratio, trim={3.5cm 0cm 3.5cm 0cm},clip]{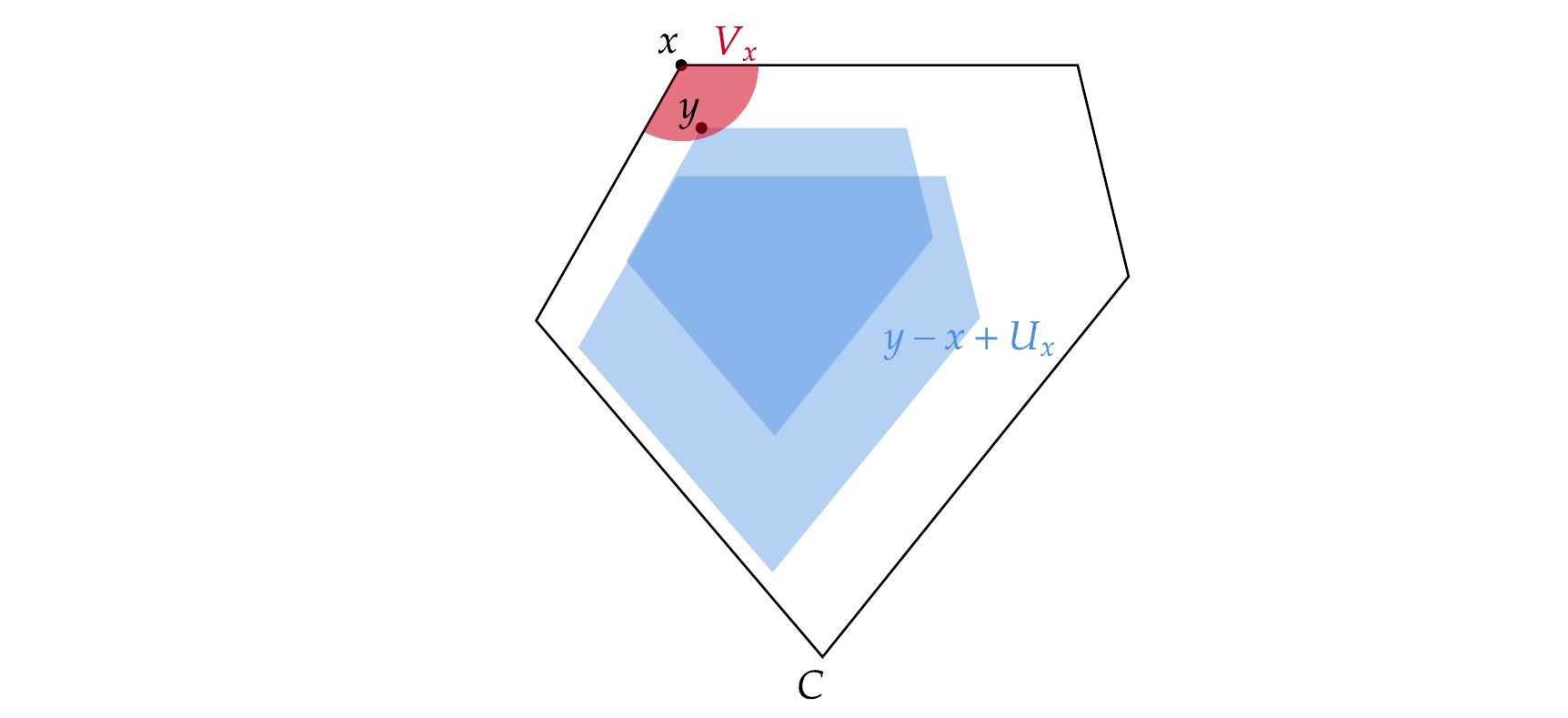}
        \caption{Illustration of $V_x$.}
        \label{fig:Vx}
    \end{minipage}
    \hfill
    \begin{minipage}[t]{0.47\textwidth}
        \centering
        \includegraphics[width=1\textwidth,height=0.6\textheight,keepaspectratio, trim={3.5cm 0cm 3.5cm 0cm},clip]{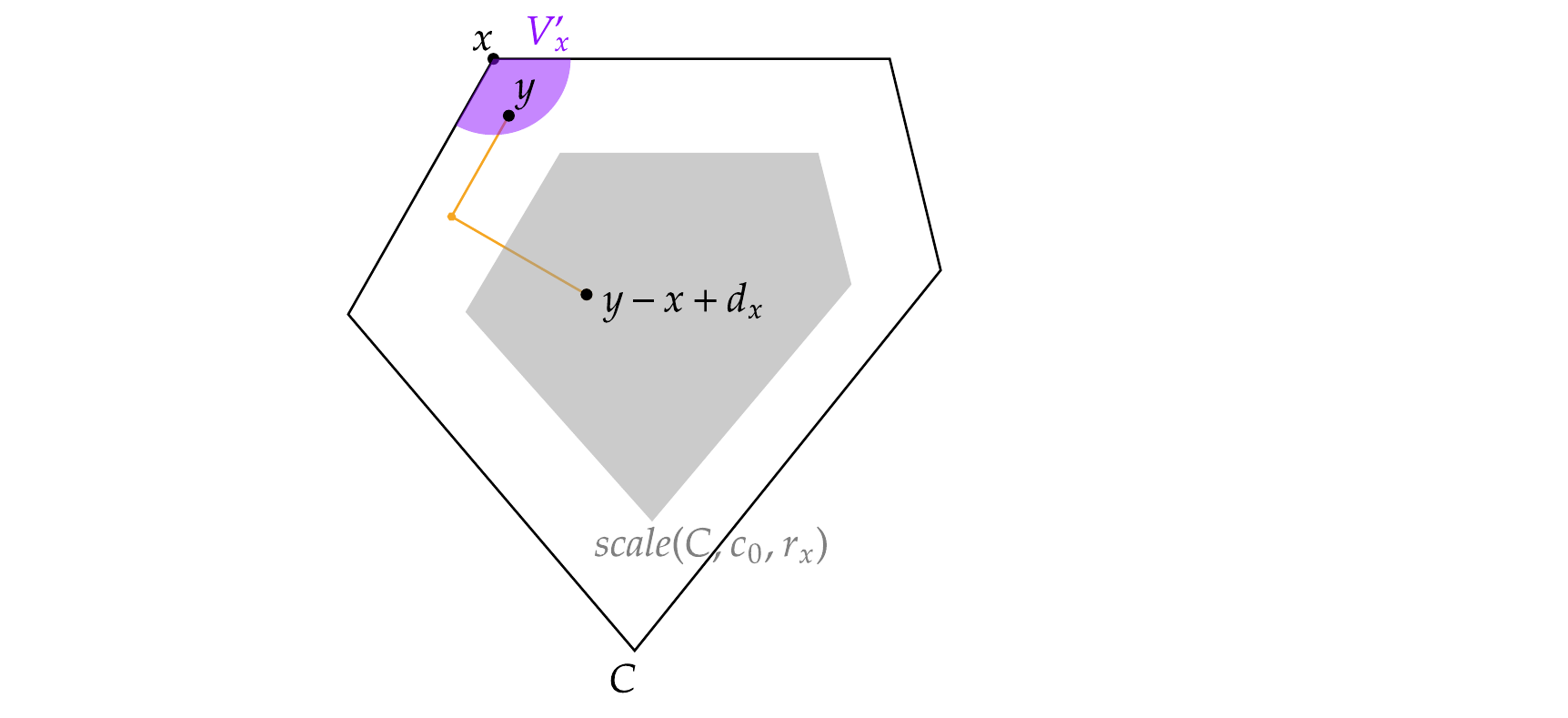}
        \caption{Illustration of $V'_x$.}
        \label{fig:Vxp}
    \end{minipage}
\end{figure}

Consider the topology of $C$ inherited from the Euclidean topology of $\R^n$ (that is, the open subsets of $C$ are of the form $C \cap U$ where $U$ is an open subset of $\R^n$).
Then $C$ is compact under this topology.

\begin{fct}
    For each $x \in C$, the set $C_x$ is an open subset of $C$.
\end{fct}
\begin{proof}
    It suffices to show that $C \setminus C_x$ is closed.
    Indeed, $C \setminus C_x = \bigcup_{F \in \lat(C), x \not\in F} \inte(F)$.
    For any $F \in \lat(C), x \not\in F$, we have $x \not\in F'$ for all faces $F' \subseteq F$.
    Therefore $\bigcup_{F \in \lat(C), x \not\in F} \inte(F) = \bigcup_{F \in \lat(C), x \not\in F} \bigcup_{F' \in \lat(C), F' \subseteq F} \inte(F') = \bigcup_{F \in \lat(C), x \not\in F} F$.
    So $C \setminus C_x = \bigcup_{F \in \lat(C), x \not\in F} F$ is a finite union of (closed) faces, and is hence closed.
\end{proof}

Denote by $d_x \coloneqq d(P_x) \in \inte(C)$ the destination of $P_x$.
Fix a point $c_0$ in the interior of $C$. Then $d_x$ is contained in the interior of $scale(C, c_0, r_x)$ for some rational $0 < r_x < 1$.

\begin{lem}
    \begin{enumerate}[noitemsep, label = (\arabic*)]
        \item There exists an open neighbourhood $V_x \subset C$ of $x$, such that for all $y \in V_x$, we have $(y - x) + U_x \subseteq C_x$.
        \item There exists an open neighbourhood $V'_x \subset C$ of $x$, such that for all $y \in V'_x$, we have $(y - x) + d_x \subseteq scale(C, c_0, r_x)$.
    \end{enumerate}
    See Figures~\ref{fig:Vx} and \ref{fig:Vxp} for illustration.
\end{lem}
\begin{proof}
For $y \in C, r \in \Rpp$, denote by $B(y, r)$ the open ball centered at $y$ with radius $r$. Denote $B_C(y, r) \coloneqq C \cap B(y, r)$; it is an open subset of $C$.

(1) For each $y \in U_x \subset C_x$, since $C_x$ is open in $C$, let $r_y > 0$ be the supremum of real numbers $r$ such that $B_C(y, r) \subset C_x$. The function $f : y \mapsto r_y$ is continuous on $U_x$ since $|r_y - r_{y'}| \leq |y - y'|$.
Since $U_x$ is compact, $f$ attains a minimum $r_{min} > 0$ on $U_x$.
Then let $V_x \coloneqq B_C(x, r_{min}/2)$, we have $(y - x) + U_x \subseteq C_x$.

(2) Since $scale(C, c_0, r_x)$ is an $n$-dimensional polytope, its interior is an open set. 
Since $d_x$ is in the interior of $scale(C, c_0, r_x)$, there exists $\rho_x > 0$ such that $B(d_x, \rho_x)$ is contained in the interior of $scale(C, c_0, r_x)$.
We then simply take $V_x \coloneqq B_C(x, \rho_x)$.
\end{proof}

For each $x \in C$, denote $W_x \coloneqq V_x \cap V'_x$. The open sets $W_x, x \in C$ cover the compact set $C$, so we can choose a finite number of representatives $x_1, \ldots, x_m$ such that $W_{x_1} \cup \cdots \cup W_{x_m} = C$.
Denote respectively by $P_1, \ldots, P_m$ the paths $P_{x_1}, \ldots, P_{x_m}$ and by $U_1, \ldots, U_m$ the sets $U_{x_1}, \ldots, U_{x_m}$.
Let $R \coloneqq \max\{r_{x_1}, \ldots, r_{x_m}\}$.

Therefore, for each $y \in C$, there exists $x_i$ such that $y \in W_{x_i}$; the set $U_i + (y - x_i)$ is contained in $C$, and the path $P_i + (y - x_i)$ leads from $y$ to the interior of $scale(C, c_0, R)$.

For each edge $e' = scale(e, c, r), e \in E(\Gamma), c \in C_x, r \in (0,1) \cap \Q$ in each of the paths $P_1, \ldots, P_m$, let $n(e') \in \N$ be such that $n(e') \cdot r \in \N$ and $n(e') \cdot c \in \Z^n$.
Furthermore, let $n_0 \in \N$ be such that $n_0 \cdot c_0 \in \Z^n$.
Define
\[
N_0 \coloneqq n_0 \cdot \prod_{e' \in P_1 \cup \cdots \cup P_m} n(e').
\]
Then in particular, $N_0 x_i$ has only integer entries for all $i = 1, \ldots, m$; and $N_0 \cdot s(e')$ has only integer entries for all $e' \in P_1 \cup \cdots \cup P_m$.

\begin{lem}\label{lem:coverPv}
    Let $N \in \N$ be such that $N_0 \mid N$.
    Then every vertex in $\Gamma_N$ is connected to some vertex in $scale(NC, N c_0, R) \cap V(\Gamma_N)$.
\end{lem}
\begin{proof}
    See Figure~\ref{fig:GammaN} and \ref{fig:coverPv} for an illustration of the proof.

\begin{figure}[ht!]
    \centering
    \begin{minipage}[t]{0.9\textwidth}
        \centering
        \includegraphics[width=0.7\textwidth,height=0.6\textheight,keepaspectratio, trim={5cm 0cm 4cm 0cm},clip]{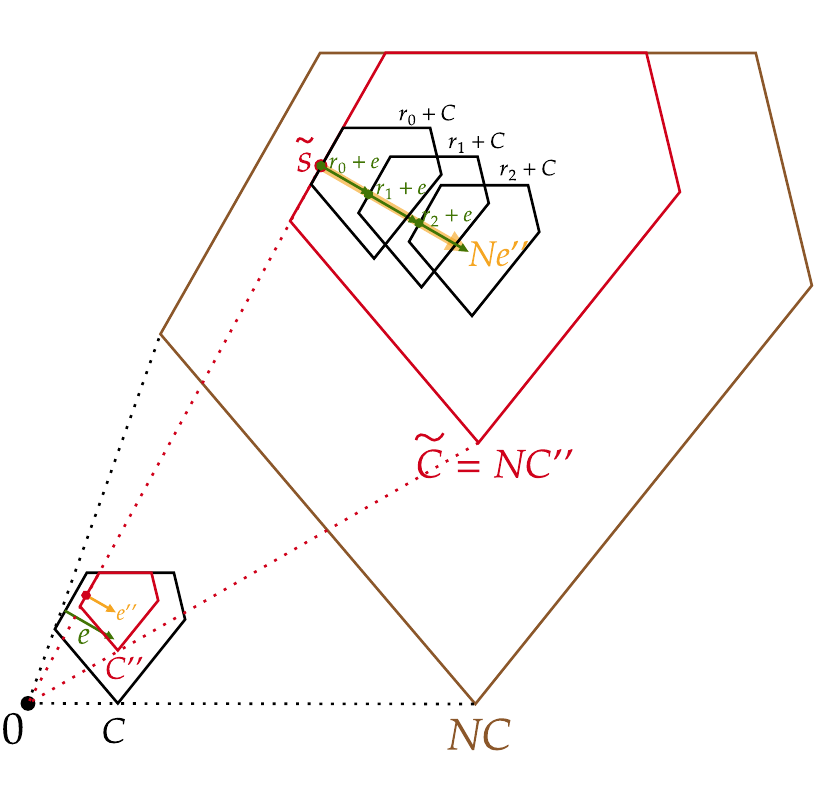}
        \caption{Illustration 1 of Lemma~\ref{lem:coverPv}.}
        \label{fig:GammaN}
    \end{minipage}
    \hfill
    \begin{minipage}[t]{0.9\textwidth}
        \centering
        \includegraphics[width=0.7\textwidth,height=0.6\textheight,keepaspectratio, trim={5cm 0cm 4cm 0cm},clip]{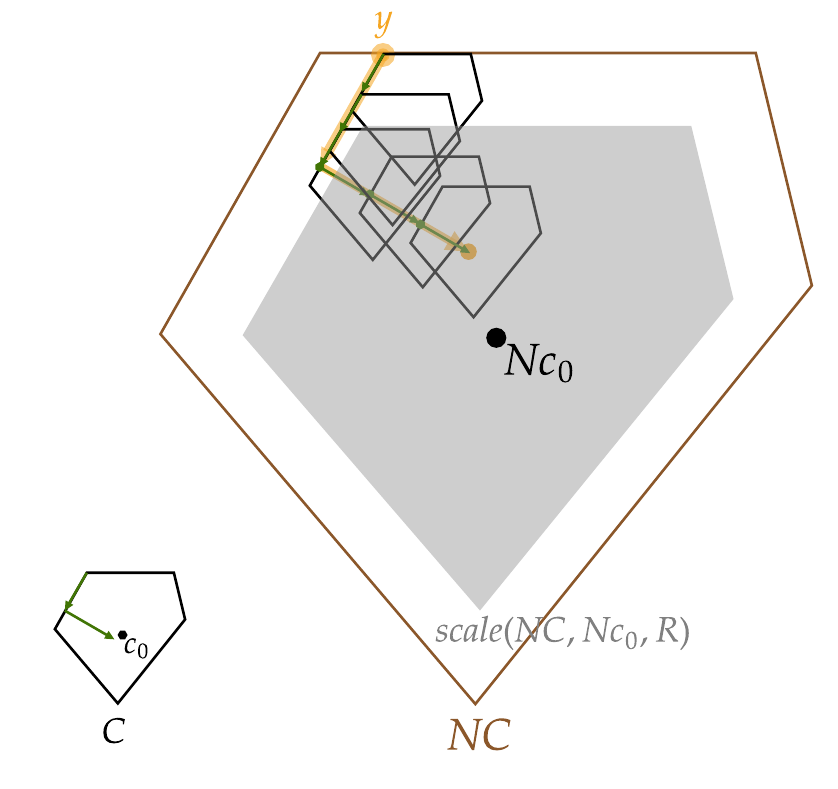}
        \caption{Illustration 2 of Lemma~\ref{lem:coverPv}.}
        \label{fig:coverPv}
    \end{minipage}
\end{figure}

    For each $y \in \conv(\Gamma_N) \subseteq NC$, the point $y' \coloneqq \frac{y}{N} \in \frac{1}{N} C$ is contained in one of $W_{x_1}, \ldots, W_{x_m}$.
    Without loss of generality suppose $y' \in W_{x_1}$, then $y'$ is connected to $(y' - x_1) + d(P_1) \in scale(C, c_0, R)$ by the path $(y' - x_1) + P_1$ in $\gq$.
    We will show that $scale((y' - x_1) + P_1, 0, N)$ is a path in $\Gamma_N$. If this is the case, then it connects $Ny' = y$ to $N(y' - x_1 + d(P_1)) \in scale(NC, N c_0, R)$ and we are done.

    We now show that $scale((y' - x_1) + P_1, 0, N)$ is a path in $\Gamma_N$.
    It suffices to show that for each edge $e' \in P_1$, the segment $scale((y' - x_1) + e', 0, N)$ is a concatenation of edges in $\Gamma_N$.
    Again write $e' = scale(e, c, r), e \in E(\Gamma), c \in C_x, r \in (0,1) \cap \Q$.
    Consider the polytope $C(e')$ defined in~\eqref{eq:Ce}. By definition of the set $U_1$, the translation $C'' \coloneqq (y' - x_1) + C(e')$ is contained in $C$.
    Also, the edge $e'' \coloneqq (y' - x_1) + e'$ is contained in $C''$.
    Therefore, $\widetilde{C} \coloneqq NC''$ is contained in $NC$.
    Note that the relative positive of $s(e'')$ in $C''$ is the same as the relative positive of $s(e)$ in $C$, so the the relative positive of $Ns(e'')$ in $NC''$ is the same as the relative positive of $s(e)$ in $C$.
    Denote $\widetilde{s} \coloneqq Ns(e'') = y - Nx_1 + Ns(e') \in \widetilde{C} \cap \Z^n$.
    For every $k = 0, 1, \ldots, rN - 1$, define
    \[
        r_k \coloneqq k(d(e) - s(e)) + \widetilde{s} - s(e) \in \Z^n.
    \]
    Consider the polytopes $r_0 + C, \ldots, r_{rN-1} + C$.
    For every $k = 0, 1, \ldots, rN - 1$,
    \begin{equation*}
         \widetilde{s} + \frac{krN}{rN - 1}(d(e) - s(e)) 
        \in  \widetilde{s} + rN(e - s(e)) 
        =  \widetilde{s} + N(e'' - s(e''))
        =  Ne'' 
        \subset  NC''
        =  \widetilde{C}. \\
    \end{equation*}
    Therefore
    \[
    scale(\widetilde{C}, \widetilde{s} + \frac{krN}{rN - 1}(d(e) - s(e)), \frac{1}{rN}) \subset \widetilde{C} \subset NC.
    \]
    Hence for $k = 0, 1, \ldots, rN - 1$ we have
    \begin{multline*}
         r_k + C 
        =  k(d(e) - s(e)) + \widetilde{s} - s(e) + C 
        =  k(d(e) - s(e)) + scale(\widetilde{C}, \widetilde{s}, \frac{1}{rN}) \\
        =  scale(\widetilde{C}, \widetilde{s} + \frac{krN}{rN - 1}(d(e) - s(e)), \frac{1}{rN}) 
        \subset  NC,
    \end{multline*}
    where the second equality comes from the fact that the relative positive of $\widetilde{s}$ in $\widetilde{C}$ is the same as the relative positive of $s(e)$ in $C$, so $scale(\widetilde{C}, \widetilde{s}, \frac{1}{rN})$ is a translation of $\frac{1}{rN} \widetilde{C} = C$ where the relative positive of $\widetilde{s}$ is the same as that of $s(e)$ in $C$.
    Thus, $r_k \in S_N$, and $r_k + e$ is an edge of $\Gamma_N$ for $k = 0, 1, \ldots rN-1$.
    Concatenating these $rN$ edges, we obtain the segment from $\widetilde{s}$ to $\widetilde{s} + rN(d(e) - s(e))$, which is exactly $Ne'' = scale((y' - x_1) + e', 0, N)$.
\end{proof}

For $k = 1, \ldots, n$, let $e_k$ denote the $k$-th element in the canonical basis of $\Z^n$.
Then since $\Gamma$ is $\Z^n$-generating, we have that for $k = 1, \ldots, n$, there exists a concatenation $Q_k$ of translations of edges in $E(\Gamma)$ that connects from $0$ to $e_k$.
Let $M_{k}$ be the total length of the edges appearing in $Q_{k}$, and let $M \coloneqq \max_{1 \leq k \leq n}\{M_{k}\}$.

Denote by $\partial C$ the boundary of $C$, that is, the union of strict faces of $C$.
For two sets $S, T \in \R^n$, define their \emph{distance} to be $dist(S, T) \coloneqq \inf\{|t - s| \mid s \in S, t \in T\}$.
The \emph{diameter} of $C$ is defined as $diam(C) \coloneqq \sup\{|t - s| \mid s, t \in C\}$.
Let $N_1 \in \N$ be such that
\begin{equation}\label{eq:defN1}
N_1 \cdot dist(scale(C, c_0, R), \partial C) > M + \sqrt{n} + diam(C).
\end{equation}
Such an $N_1$ exists because $scale(C, c_0, R)$ and $\partial C$ are disjoint closed sets, so their distance is larger than zero.
\begin{lem}\label{lem:mix}
    Let $N \in \N$ be such that $N > N_1$. Every two vertices in $scale(NC, N c_0, R) \cap V(\Gamma_N)$ are connected in $\Gamma_N$.
\end{lem}
\begin{proof}
    See Figure~\ref{fig:mix} for an illustration of the proof.

    Let $v_1, v_2$ be two arbitrary vertices in $scale(NC, N c_0, R) \cap V(\Gamma_N)$.
    There exists a path $P_{\Z^n}(v_1, v_2)$ in the grid $\Z^n$ from $v_1$ to $v_2$, such that each point in $P_{\Z^n}(v_1, v_2)$ is at most of distance $\sqrt{n}$ from $seg(v_1, v_2)$.
    The path $P_{\Z^n}(v_1, v_2)$ consists of translations of the segments $seg(0, e_k), k = 1, \ldots, n$.
    For $k = 1, \ldots, n$, replacing each translation $seg(0, e_k) + z$ of the segment $seg(0, e_k)$ in $P_{\Z^n}(v_1, v_2)$ with the translation $Q_k + z$ of the path $Q_k$, we obtain a path $P_{\Gamma}(v_1, v_2)$.
    We now show that each edge of $P_{\Gamma}(v_1, v_2)$ is in $E(\Gamma_N)$.

    Each point in $P_{\Gamma}(v_1, v_2)$ is at most of distance $\sqrt{n} + M$ from the segment $seg(v_1, v_2) \subset scale(NC, N c_0, R)$, so it is at most of distance $\sqrt{n} + M$ from $scale(NC, N c_0, R)$.
    By the definition~\eqref{eq:defN1} of $N_1$, we have 
    \[
    dist\left(scale(NC, N c_0, R), \partial \left(NC\right)\right) = N \cdot dist(scale(C, c_0, R), \partial C) > M + \sqrt{n} + diam(C).
    \]
    Therefore each point in $P_{\Gamma}(v_1, v_2)$ is at least of distance $diam(C)$ from the boundary $\partial \left(NC\right)$.
    Take an arbitrary edge $e$ in $P_{\Gamma}(v_1, v_2)$, it comes from some translation $\Gamma + z$ of the graph $\Gamma$.
    Therefore $e$ is contained in $C + z$. Since $e$ is of distance at least $diam(C)$ from $\partial \left(NC\right)$, the polytope $C + z$ must be contained in $NC$.
    Hence $z \in S_N$ and so $e$ is an edge of $\Gamma_N$.
    We have shown that each edge of $P_{\Gamma}(v_1, v_2)$ is in $E(\Gamma_N)$.
    Therefore every two vertices in $scale(NC, N c_0, R) \cap V(\Gamma_N)$ are connected in $\Gamma_N$.
\end{proof}

    \begin{figure}[h!]
        \centering
        \includegraphics[width=0.4\textwidth,height=0.6\textheight,keepaspectratio, trim={5cm 0cm 4cm 0cm},clip]{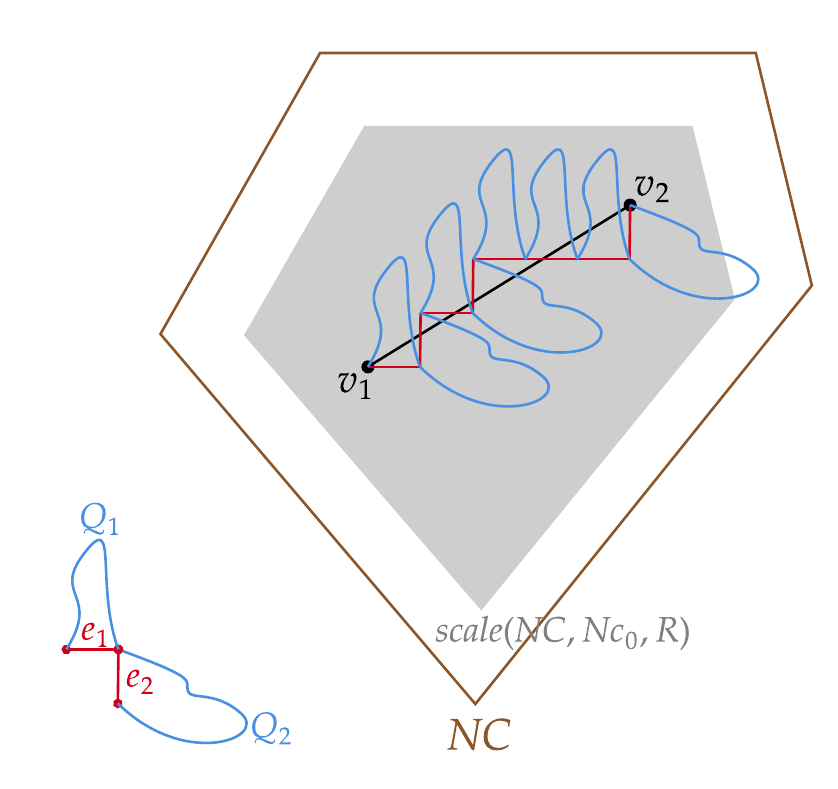}
        \caption{Illustration of Lemma~\ref{lem:mix}.}
        \label{fig:mix}
    \end{figure}

\thmacctoeul*
\begin{proof}
    See Figure~\ref{fig:connect} for an illustration of the proof.
    Let $N \in \N$ be such that $N_0 \mid N$ and $N > N_1$. 
    We show that the graph $\Gamma_N$ is connected.
    Take any two vertices $v, w$ of the graph $\Gamma_N$.
    Since $N_0 \mid N$, Lemma~\ref{lem:coverPv} shows that $v$ and $w$ are respectively connected to two vertices $v_1$ and $v_2$ in $scale(NC, N c_0, R) \cap V(\Gamma_N)$.
    Since $N > N_1$, Lemma~\ref{lem:mix} shows that $v_1$ and $v_2$ are connected in $\Gamma_N$.
    Therefore, $v$ and $w$ are connected in $\Gamma_N$.
\end{proof}

\begin{figure}[h!]
        \centering
        \includegraphics[width=0.4\textwidth,height=0.6\textheight,keepaspectratio, trim={5cm 0cm 4cm 0cm},clip]{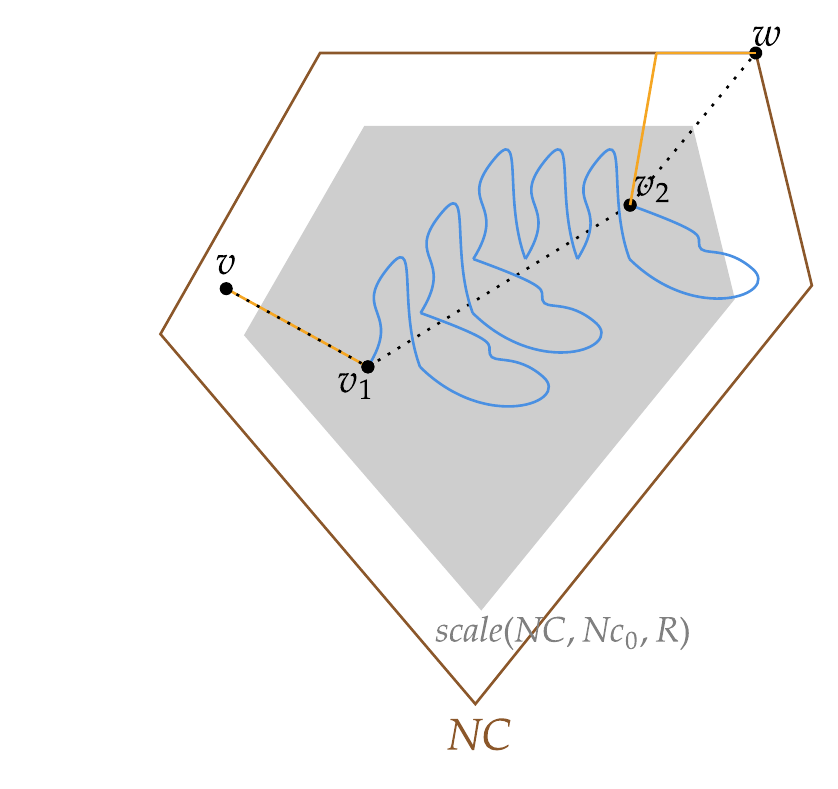}
        \caption{Illustration of proof of Theorem~\ref{thm:acctocon}.}
        \label{fig:connect}
\end{figure}

\section{Proving the local-global principle}\label{sec:locglob}
In this section we prove Theorem~\ref{thm:locglob}:

\thmlocglob*

Theorem~\ref{thm:locglob} can be considered as a generalization of Einsiedler, Mouat and Tuncel's local-global principle~\cite[Theorem~1.3]{einsiedler2003does}.
The difference in our paper is the additional constraint \eqref{eq:globv}.
Recall that this property stems from the face-accessibility constraint in Proposition~\ref{prop:eulertoeq}, and is hence essential for our purpose of studying sub-semigroups of metabelian groups.
Many components of the original proof in \cite{einsiedler2003does} fail when integrating  Property~\eqref{eq:globv}, notably~\cite[Lemma~3.2, Lemma~5.2]{einsiedler2003does}.
In order to take into account this extra property, we need to introduce new arguments to rework many parts of the original proof.
The key new component will be the following Lemma~\ref{lem:tweakinit}, which shows a certain ``continuity'' of Property~\eqref{eq:globv} when changing the direction $v$ by a small amount.

Before stating Lemma~\ref{lem:tweakinit}, we will make a few observations.
Define the quotient 
\[
D_n \coloneqq \Rns / \Rpp.
\]
That is, elements of $D_n$ are of the form $v \Rpp, v \in \Rns$, where $v \Rpp = v' \Rpp$ if and only if $v = r \cdot v'$ for some $r \in \Rpp$.
The quotient $D_n$ can be identified with the unit sphere of dimension $n$ since every $v \Rpp$ is equal to exactly one $v' \Rpp$ with $||v'|| = 1$.
We equip $D_n$ with the standard topology of the unit sphere.
Note that $\init_v(\cdot), M_v(\cdot)$ and $O_v$ are invariant when scaling $v$ by any positive real number.

\begin{restatable}{lem}{lemtweak}\label{lem:tweakinit}
    Fix $v \in \Rns$, a set $I \subseteq \{1, \ldots, K\}$ and $\bff \in \A^K$. 
    There exists an open neighbourhood $U \subseteq D_n$ of $v \Rpp$, such that for every $w \in \Rns$ with $(v + w)\Rpp \in U$, we have
    \begin{equation}\label{eq:vw}
    \init_{v + w}\left(\bff\right) = \init_{w}\left(\init_{v}(\bff)\right), \quad M_{v + w}(I, \bff) = M_{w}(M_{v}(I, \bff), \init_{v}(\bff)), \quad \text{and} \quad O_{v + w} = O_{v} \cup O_{w}.
    \end{equation}
\end{restatable}

\begin{figure}[ht!]
    \centering
    \begin{minipage}[t]{.47\textwidth}
        \centering
        \includegraphics[width=1\textwidth,height=1.0\textheight,keepaspectratio, trim={3.5cm 0cm 2cm 0cm},clip]{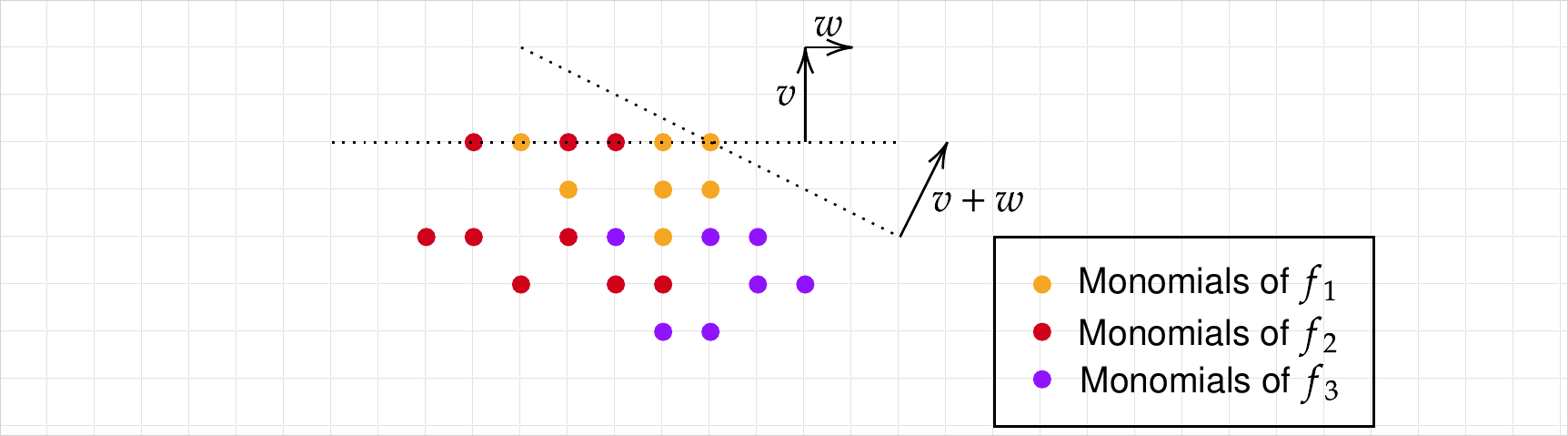}
        \caption{Illustration of Lemma~\ref{lem:tweakinit}. Here, $M_{v}(\{1, 2, 3\}, \bff) = \{1,2\}$.}
        \label{fig:tweak}
    \end{minipage}
    \hfill
    \begin{minipage}[t]{0.47\textwidth}
        \centering
        \includegraphics[width=1\textwidth,height=1.0\textheight,keepaspectratio, trim={3.5cm 0cm 2cm 0cm},clip]{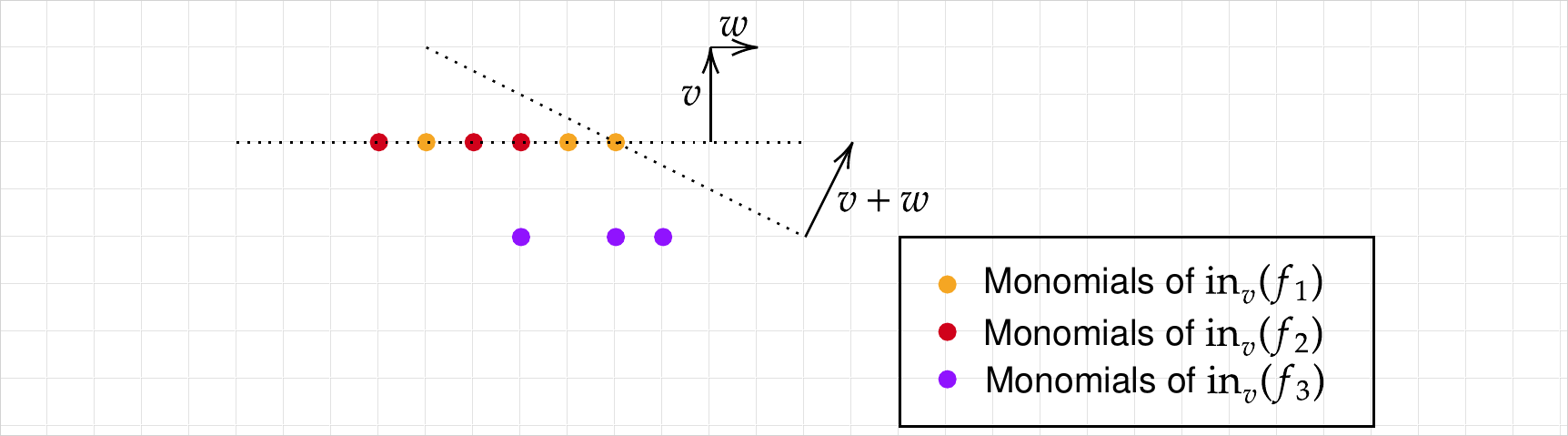}
        \caption{Illustration of Lemma~\ref{lem:tweakinit}. Here, $M_{v+w}(\{1, 2, 3\}, \bff) = \{1\} = M_{w}(\{1, 2\}, \init_v(\bff))$.}
        \label{fig:tweak2}
    \end{minipage}
\end{figure}
\begin{proof}
	See Figures~\ref{fig:tweak} and \ref{fig:tweak2} for an illustration.
    For each $i = 1, \ldots, K$ such that $f_i \neq 0$, write $f_i = g_i + h_i$ where $g_i \coloneqq \init_v(f_i)$ and $\deg_{v}(h_i) < \deg_{v}(g_i)$.
    
    Since $\deg_{v}(h_i)$ and $\deg_{v}(g_i)$ vary continuously when $v \Rpp$ varies in $D_n$, there exists an open neighbourhood $U_i \subseteq D_n$ of $v \Rpp$ such that $\deg_{v'}(h_i) < \deg_{v'}(g_i)$ for every $v'\Rpp \in U_i$.
    Therefore, for $(v + w)\Rpp \in U_i$ we have
    \[
    \init_{v + w}\left(f_i\right) = \init_{v + w}(g_i) = \init_{w}\left(g_i\right) = \init_{w}\left(\init_{v}(f_i)\right).
    \]
    Where $\init_{v + w}(g_i) = \init_{w}\left(g_i\right)$ can be justified as follows.
    For every monomials $c \oX^b$ appearing in $g_i$, we have $\deg_{v+w}(c \oX^b) = v \cdot b + w \cdot b = \deg_v g_i + w \cdot b = \deg_v g_i + \deg_{w}(c \oX^b)$; therefore the monomials in $g_i$ with maximal $\deg_{v+w}$ are exactly those with maximal $\deg_{w}$.

    Note that for each $i \in M_{v}(I, \bff), i' \in I \setminus M_{v}(I, \bff)$, we have $\deg_v(f_i) > \deg_v(f_{i'})$.
    Again by the continuity of $\deg_v$ with respect to $v$, there exists an open neighbourhood $U'$ of $v \Rpp$ such that for all $v'\Rpp \in U'$, we have $\deg_{v'}(f_i) > \deg_{v'}(f_{i'})$ for $i \in M_{v}(I, \bff), i' \in I \setminus M_{v}(I, \bff)$.
    Then for every $(v + w)\Rpp \in \cap_{i \in I, f_i \neq 0} U_i \cap U'$, we have
    \begin{align*}
        & M_{v + w}\left(I, \bff\right) \\
        = & \left\{i \in I \;\middle|\; \deg_{v + w}(f_i) = \max_{i' \in I}\{\deg_{v + w}(f_{i'})\} \right\} \\
        = & \left\{i \in M_{v}(I, \bff) \;\middle|\; \deg_{v + w}(f_i) = \max_{i' \in M_{v}(I, \bff)}\{\deg_{v + w}(f_{i'})\} \right\} \quad \text{ (since $(v+w)\Rpp \in U'$) }\\
        = & \left\{i \in M_{v}(I, \bff) \;\middle|\; \deg_{w}(\init_{v}(f_i)) = \max_{i' \in M_{v}(I, \bff)}\{\deg_{w}(\init_{v}(f_{i'})) \right\} \quad \text{ (since $(v+w)\Rpp \in U_i, i \in I$) } \\
        = & \, M_{w}(M_{v}(I, \bff), \init_{v}(\bff_{v})).
    \end{align*}
    Finally, take any $i \in \{1, \ldots, K\}$. 
    If $a_i \not\perp v$ then there exists an open neighbourhood $U''_i \subseteq D_n$ of $v \Rpp$ such that for every $v'\Rpp \in U''_i$ we have $a_i \not\perp v'$.
    If $a_i \perp v$ then for every $w \in \Rns$ we have $a_i \perp (v + w) \iff a_i \perp w$.
    Take $U'' \coloneqq \cup_{i \in \{1, \ldots, K\}, a_i \not\perp v} U''_i$.
    For all $(v + w)\Rpp \in U''$, we have
    \begin{align*}
        O_{v + w} & = \{i \in \{1, \ldots, K\} \mid a_i \not\perp (v + w)\} \\
        & = \left\{i \in \{1, \ldots, K\} \mid (a_i \not\perp v \text{ and } a_i \not\perp (v+w)) \text{ or } \left(a_i \perp v \text{ and } a_i \not\perp (v+w) \right) \right\} \\
        & = \left\{i \in \{1, \ldots, K\} \mid (a_i \not\perp v) \text{ or } \left(a_i \perp v \text{ and } a_i \not\perp (v+w) \right) \right\} \quad \text{ (since $(v+w)\Rpp \in U''_i$) } \\
        & = \left\{i \in \{1, \ldots, K\} \mid (a_i \not\perp v) \text{ or } \left(a_i \perp v \text{ and } a_i \not\perp w \right) \right\} \\
        & = \left\{i \in \{1, \ldots, K\} \mid (a_i \not\perp v) \text{ or } \left(a_i \not\perp w \right) \right\} \\
        & = O_{v} \cup O_{w}.
    \end{align*}
    We conclude the proof by taking $U = \cap_{i \in I, f_i \neq 0} U_i \cap U' \cap U''$.
\end{proof}

As an illustration of how to integrate Property~\eqref{eq:globv} into the local-global principle using Lemma~\ref{lem:tweakinit}, we first give a proof of the ``only if'' part of Theorem~\ref{thm:locglob}.
This is the easier implication in Theorem~\ref{thm:locglob}.
As a comparison, the ``only if'' parts of the cited results \cite[Theorem~1.3]{einsiedler2003does} and \cite[Proposition~3.4]{dong2023identity} are both immediate.

\begin{proof}[Proof of ``only if'' part of Theorem~\ref{thm:locglob}]
    Suppose we have $\bff \in \mM \cap \ApK$ satisfying Property~\eqref{eq:globv}.
    For \hyperref[item:locr]{(LocR)}, simply take $\bff_{r} \coloneqq \bff$ for all $r \in \Rpp^n$, then $\bff(r) \in \Rpp^K$.
    As for \hyperref[item:locinf]{(LocInf)}, for every $v \in \Rns$ we show that $\bff_{v} \coloneqq \bff$ satisfies Properties \hyperref[item:locinf]{(LocInf)}(a) and (b). 
    Property \hyperref[item:locinf]{(LocInf)}(a) is satisfied by the definition of $\bff$.
    We now show Property \hyperref[item:locinf]{(LocInf)}(b).
    
    When $w \in v \Rpp$, we have $O_{w} \cup J' = O_{v} \cup J$ and $M_{w}(I', \init_{v}(\bff)) = M_{v}(M_{v}(I, \bff), \init_{v}(\bff)) = M_{v}(I, \bff)$, so Property~\hyperref[item:locinf]{(LocInf)}(b) is equivalent to $\left(O_{v} \cup J\right) \cap M_{v}(I, \bff) \neq \emptyset$.
    This is satisfied by the definition of $\bff$.
    
    When $w \not\in v \Rpp$, let $U \subseteq D_n$ be the open neighbourhood of $v \Rpp$ defined in Lemma~\ref{lem:tweakinit}.
    Scaling $w$ by a small enough positive real we can suppose $(v + w)\Rpp \in U$.
    We have
    \[
    \init_{v + w}\left(\bff\right) = \init_{w}\left(\init_{v}(\bff)\right), \quad M_{v + w}(I, \bff) = M_{w}(I', \init_{v}(\bff)), \quad \text{and} \quad O_{v + w} = O_{v} \cup O_{w},
    \]
    where $I' = M_{v}(I, \bff)$.
    Therefore
    $
    \left(O_{v + w} \cup J\right) \cap M_{v + w}(I, \bff) = \left(O_{w} \cup O_{v} \cup J\right) \cap M_{w}(I', \init_{v}(\bff))
     = \left(O_{w} \cup J'\right) \cap M_{w}(I', \init_{v}(\bff)).
    $
    Since $\bff$ satisfies Property~\eqref{eq:globv}, we have $\left(O_{v + w} \cup J\right) \cap M_{v + w}(I, \bff) \neq \emptyset$.
    Therefore we also have $\left(O_{w} \cup J'\right) \cap M_{w}(I', \init_{v}(\bff)) \neq \emptyset$ for all $w \not\in v \Rpp$.
\end{proof}

We now start working towards proving the ``if'' part of Theorem~\ref{thm:locglob}.
The main idea is a ``gluing'' procedure inspired by the original proof of Einsiedler et al.~\cite{einsiedler2003does}.
The following lemma is the foundation of this gluing argument.

\begin{lem}\label{lem:Uv}
    Suppose $v \in \Rns$ and $\bff_{v} \in \mM$ satisfy properties \hyperref[item:locinf]{(LocInf)}(a) and (b) of Theorem~\ref{thm:locglob}.
    Then there exists an open neighbourhood $U_{v} \subseteq D_n$ of $v \Rpp$ such that for every $v'\Rpp \in U_v$,
    \begin{enumerate}[noitemsep, label = (\roman*)]
        \item $\init_{v'}\left(\bff_{v}\right) \in \ApK$.
        \item $(O_{v'} \cup J) \cap M_{v'}(I, \bff_{v}) \neq \emptyset$.
    \end{enumerate}
\end{lem}
\begin{proof}
    We use Lemma~\ref{lem:tweakinit} on $v, I$ and $\bff_{v}$ to obtain an open neighbourhood $U_{v} \subseteq D_n$ of $v \Rpp$,
    where for all $(v + w)\Rpp \in U_{v}$ we have
    \[
    \init_{v + w}\left(\bff_{v}\right) = \init_{w}\left(\init_{v}(\bff_{v})\right), \; M_{v + w}(I, \bff_{v}) = M_{w}(M_{v}(I, \bff_{v}), \init_{v}(\bff_{v})), \; \text{and} \; O_{v + w} = O_{v} \cup O_{w}.
    \] 
    Note that $\init_{v}(\bff_{v}) \in \ApK$ by Property~\hyperref[item:locinf]{(LocInf)}(a) of $\bff_{v}$.
    Since taking the initial polynomial of any polynomial in $\A^+$ yields an element of $\A^+$, we have $\init_{v + w}\left(\bff_{v}\right) = \init_{w}\left(\init_{v}(\bff_{v})\right) \in \ApK$.
     Furthermore, $(O_{v + w} \cup J) \cap M_{v + w}(I, \bff_{v}) = \left(O_{w} \cup J'\right) \cap M_{w}(I', \init_{v}(\bff_{v}))$, which is non-empty by Property~\hyperref[item:locinf]{(LocInf)}(b) of $\bff_{v}$.
    Therefore, both (i) and (ii) are satisfied for $v'\Rpp \in U_{v}$.
\end{proof}

The following lemma is a strengthening of \cite[Lemma~5.2]{einsiedler2003does}.
\begin{lem}\label{lem:finf}
    Suppose Condition \hyperref[item:locinf]{(LocInf)} of Theorem~\ref{thm:locglob} is satisfied.
    Then there exists $\bff \in \mM$ that satisfies 
    \begin{enumerate}[noitemsep, label = (\roman*)]
        \item $\init_{v}\left(\bff\right) \in \ApK$ for all $v \in \Rns$.
        \item $(O_{v} \cup J) \cap M_{v}(I, \bff) \neq \emptyset$ for all $v \in \Rns$.
    \end{enumerate}
\end{lem}
\begin{proof}
    The main steps of our proof follow that of \cite[Lemma~5.2]{einsiedler2003does}.
    See Figure~\ref{fig:finf} for an illustration of this proof.
    For each $v \in \Rns, ||v|| = 1$, let $\bff_{v} \in \mM$ satisfy Properties \hyperref[item:locinf]{(LocInf)}(a) and (b).
    Let $U_{v} \subseteq D_n$ be an open neighbourhood of $v\Rpp$ defined in Lemma~\ref{lem:Uv}.
    The sets $\{U_{v} \mid v \in \Rns, ||v|| = 1\}$ form an open cover of the compact set $D_n$.
    We identify $D_n$ with the unit sphere in $\R^n$, and consider the metric on $D_n$ inherited from $\R^n$.

    Let $2 \lambda < 1$ be a Lebesgue number~\cite[Lemma~7.2]{munkres1974topology} of the open covering $\{U_{v} \mid v \in \Rns, ||v|| = 1\}$, meaning every ball of radius $\lambda$ in $D_n$ is contained in some $U_{v}$.
    Take a finite collection of balls of radius $\lambda$ which cover $D_n$, and label their centers $v_1, \ldots, v_m$. 
    Note that each ball $B(v_j, \lambda), j = 1, \ldots, m$, is contained in some $U_{v'_j}$.
    Let $\bF_j \coloneqq \bff_{v'_j}$, so that $\init_{v}(\bF_j) \in \ApK$ and $(O_{v} \cup J) \cap M_{v}(I, \bF_j) \neq \emptyset$ for all $v \in B(v_j, \lambda)$ (by Lemma~\ref{lem:Uv}). 
    Let $2 \kappa$ be a Lebesgue number for the cover $\{B(v_j, \lambda) \mid j = 1, \ldots, m\}$ of $D_n$. 
    Then for any $v \in D_n$ there exists $j \in \{1, \ldots, m\}$ such that $B(v, \kappa) \subset B(v_j, \lambda)$ and, in particular, $||v_j - v|| < \lambda - \kappa$.

    Let $\delta$ be the infimum of
    \[
    \left\{v \cdot  v_j - v \cdot  v_{j'} \;\middle|\; ||v|| = 1, ||v_j - v|| < \lambda - \kappa, ||v_{j'} - v|| \geq \lambda, j, j' = 1, \ldots, m \right\}.
    \]
    Note that $\delta \geq \kappa(\lambda - \frac{\kappa}{2})$ since for all $v, w, w'$ of norm one we have
    \[
    v \cdot  w - v \cdot  w' = \frac{1}{2} \left(||w' - v||^2 - ||w - v||^2\right).
    \]
    Choose $r$ large enough so that $\deg_{v}(F_{j, i}) < \frac{\delta}{2} r - \frac{\sqrt{n}}{2}$ for all $v \in D_n$, $i = 1, \ldots, K$ and $j = 1, \ldots, m$.

    For $j = 1, \ldots, m$ pick $w_j \in \Z^n$ such that $||w_j - r v_j|| \leq \frac{\sqrt{n}}{2}$.
    Let
    \[
    \bff \coloneqq \sum_{j = 1}^m \oX^{w_j} \bF_j.
    \]
    We show that $\bff$ satisfies both conditions (i) and (ii).
    Consider any $v \in \Rns$ with norm one.
    Let $j \in \{1, \ldots, m\}$ be such that $||v - v_j|| < \lambda - \kappa$.
    For a $j' \in \{1, \ldots, m\}$ with $||v - v_{j'}|| \geq \lambda$ we have
    \begin{align}\label{eq:maxmin}
    \max_{1 \leq i \leq K}\{\deg_{v}(\oX^{w_{j'}} F_{j', i})\} & = v \cdot w_{j'} + \max_{1 \leq i \leq K}\{\deg_{v}(F_{j', i})\} \nonumber\\
    & < v \cdot w_{j'} + \frac{\delta}{2} r - \frac{\sqrt{n}}{2} \nonumber\\
    & \leq rv \cdot  v_{j'} + \frac{\delta}{2} r \nonumber \\
    & \leq rv \cdot  v_{j} - r ||v|| \cdot ||v_{j} - v_{j'}|| + \frac{\delta}{2} r \nonumber \\
    & \leq rv \cdot  v_{j} - \frac{\delta}{2} r \nonumber \\
    & \leq v \cdot w_{j} + \frac{\sqrt{n}}{2} - \frac{\delta}{2} r \nonumber\\
    & < v \cdot w_{j} + \min_{1 \leq i \leq K}\{\deg_{v}(F_{j, i})\} \nonumber\\
    & = \min_{1 \leq i \leq K}\{\deg_{v}(\oX^{w_j} F_{j, i})\}.
    \end{align}
    For the remaining indices $j'$ with $||v - v_{j'}|| < \lambda$ we already know that $\init_{v}(\bF_{j'}) \in \ApK$.
    Since there can be no cancellation with those initial parts, we get
    \[
    \init_{v}(\bff) = \init_{v}\left(\sum_{j' \colon ||v - v_{j'}|| < \lambda} \oX^{w_{j'}} \bF_{j'} \right) \in \ApK.
    \]
    Therefore $\bff$ satisfies condition (i).
    For condition (ii), take any $v \in \Rns$, let $j' \in \{1, \ldots, m\}$ be such that $\max_{i \in I}\{\deg_{v}(\oX^{w_{j'}} F_{j', i})\} = \max_{1 \leq j \leq m} \max_{i \in I}\{\deg_{v}(\oX^{w_{j}} F_{j, i})\}$.
    We must have $||v - v_{j'}|| < \lambda$.
    Indeed, if we had $||v - v_{j'}|| \geq \lambda$ then there exists $j \in \{1, \ldots, m\}$ be such that $||v - v_j|| < \lambda - \kappa$ and the contradiction $\max_{i \in I}\{\deg_{v}(\oX^{w_{j'}} F_{j', i})\} < \max_{i \in I}\{\deg_{v}(\oX^{w_j} F_{j, i})\}$ follows from Inequality~\eqref{eq:maxmin}.
    
    We will now show $M_{v}(I, \bF_{j'}) \subseteq M_{v}(I, \bff)$.
    Take any $i' \in M_{v}(I, \bF_{j'})$, we show $i' \in M_{v}(I, \bff)$.
    
    On one hand, since $||v - v_{j'}|| < \lambda$, so $\init_{v}(\bF_{j'}) \in \ApK$.
    We have
    \begin{multline}\label{eq:maxi}
        \deg_{v}(\oX^{w_{j'}} F_{j', i'}) = \max_{i \in I}\{\deg_{v}(\oX^{w_{j'}} F_{j', i})\} = \max_{1 \leq j \leq m} \max_{i \in I} \{\deg_{v}(\oX^{w_{j}} F_{j, i})\} \\
        = \max_{i \in I} \max_{j : ||v - v_j|| < \lambda} \{\deg_{v}(\oX^{w_{j}} F_{j, i})\} = \max_{i \in I} \deg_{v}\left(\sum_{j : ||v - v_j|| < \lambda} \oX^{w_j} F_{j,i}\right) = \max_{i \in I} \deg_{v}\left(\sum_{j = 1}^m \oX^{w_j} F_{j,i}\right) \\
        = \max_{i \in I} \deg_{v}(f_i),
    \end{multline}
    since there can be no cancellation when summing $\init_{v}(\oX^{w_j} F_{j,i}) \in \A^+$ for $j, ||v - v_j|| < \lambda$.

    On the other hand,
    \[
        \deg_{v}(\oX^{w_{j'}} F_{j', i'}) = \max_{1 \leq j \leq m} \max_{i \in I} \{\deg_{v}(\oX^{w_{j}} F_{j, i})\} \geq \max_{1 \leq j \leq m}\{\deg_{v}(\oX^{w_{j}} F_{j, i'})\}.
    \]
    So $\deg_{v}(\oX^{w_{j'}} F_{j', i'}) = \max_{1 \leq j \leq m}\{\deg_{v}(\oX^{w_{j}} F_{j, i'})\}$.
    Hence,
    \begin{multline*}
    \deg_{v}(\oX^{w_{j'}} F_{j', i'}) = \max_{1 \leq j \leq m}\{\deg_{v}(\oX^{w_{j}} F_{j, i'})\} = \max_{j : ||v - v_j|| < \lambda}\{\deg_{v}(\oX^{w_{j}} F_{j, i'})\} \\
    = \deg_{v}\left( \sum_{j : ||v - v_j|| < \lambda}\oX^{w_{j}} F_{j, i'}\right)
    = \deg_{v}\left( \sum_{j = 1}^m \oX^{w_{j}} F_{j, i'}\right) = \deg_{v}(f_{i'})
    \end{multline*}
    since there can be no cancellation when summing $\init_{v}(\oX^{w_{j}} F_{j, i'}) \in \A^+$ for $j, ||v - v_j|| < \lambda$.
    
    Hence $\deg_{v}(f_{i'}) = \deg_{v}(\oX^{w_{j'}} F_{j', i'}) = \max_{i \in I} \deg_{v}(f_i)$, which yields $i' \in M_{v}(I, \bff)$.
    Since this holds for all $i' \in M_{v}(I, \bF_{j'})$, we have shown $M_{v}(I, \bF_{j'}) \subseteq M_{v}(I, \bff)$.
    Thus, $(O_{v} \cup J) \cap M_{v}(I, \bff) \supseteq (O_{v} \cup J) \cap M_{v}(I, \bF_{j'}) \neq \emptyset$.
    Therefore $\bff$ satisfies condition (ii).
\end{proof}

\begin{figure}[h!]
        \centering
        \includegraphics[width=1\textwidth,height=0.6\textheight,keepaspectratio, trim={0cm 0cm 0cm 0cm},clip]{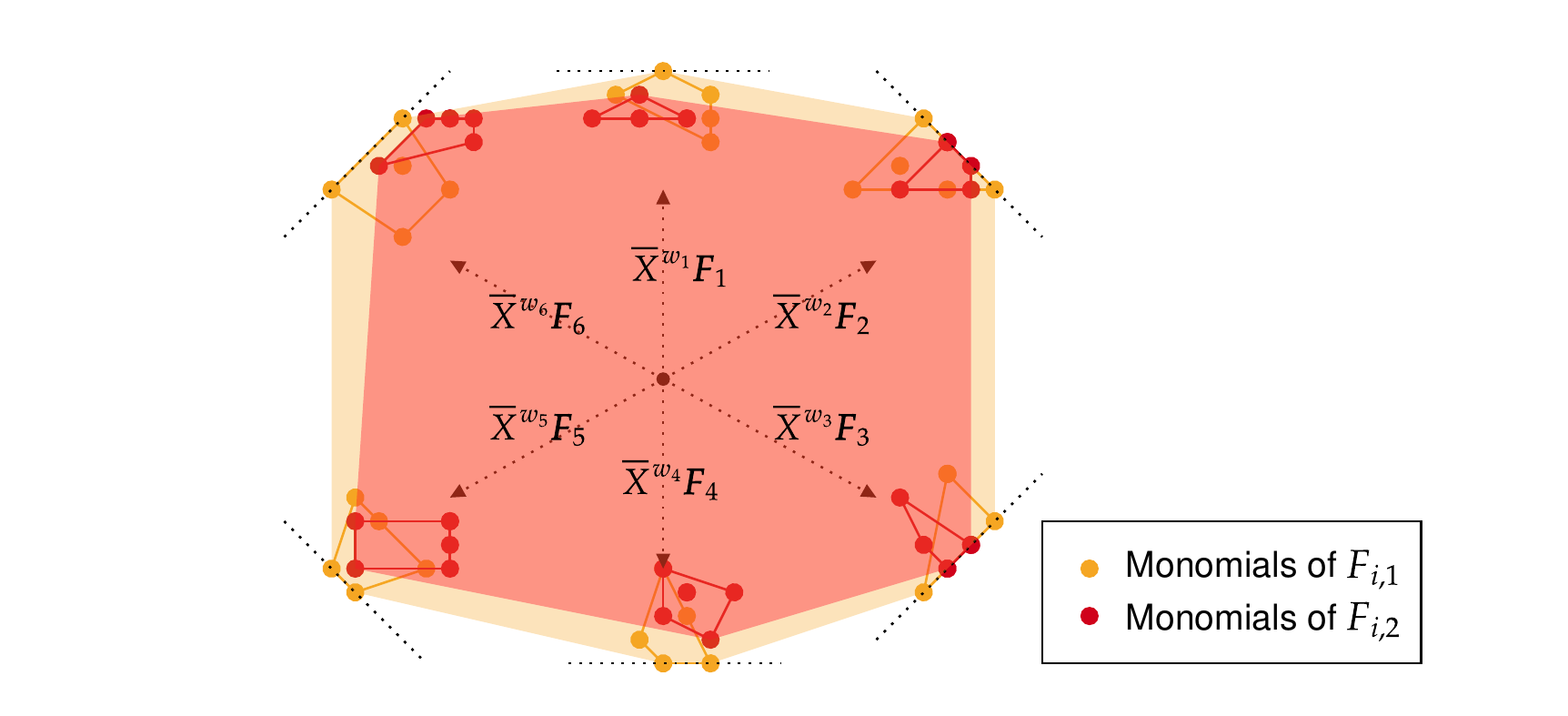}
        \caption{Illustration of proof of Lemma~\ref{lem:finf}. In this example, $a_1 = a_2 = (0,0), I = \{1, 2\}, J = \{1\}$. That is, condition~(ii) requires $1 \in M_v(\{1, 2\}, \bff)$ for all $v \in \Rns$}.
        \label{fig:finf}
\end{figure}

Denote by $\bff_{\infty}$ the element $\bff \in \mM$ obtained in Lemma~\ref{lem:finf}.
Since $\init_{v}\left(\bff_{\infty}\right) \in \ApK$ for all $v \in \Rns$, there exists $c > 1$ such that $\bff_{\infty}(x) \in \Rpp^K$ for all $x \in \Rpp^n \setminus [1/c, c]^n$.
Define the compact set $C \coloneqq [1/(4nc), 4nc]^n$.

\begin{lem}\label{lem:fC}
    Let $\mM$ be an $\A$-submodule of $\A^K$ and $C \subset \Rpp^n$ be a compact set.
    Suppose for all $r \in C$ there exists $\bff_{r} \in \mM$ with $\bff_{r}(r) \in \Rpp^K$.
    Then there exists $\bff \in \mM$ such that 
    $
        \bff(x) \in \Rpp^{n}
    $
    for all $x \in C$.
\end{lem}
\begin{proof}
    For each $r \in C$, by the continuity of polynomial functions, there is an open ball $B(r, b_{r})$, centered at $r$, with radius $b_{r}$, such that $\bff_{r}(x) \in \Rpp^{n}$ for all $x \in B(r, b_{r})$.
    
    Consider the open cover $B(r, \frac{b_{r}}{2}), r \in C$ of the set $C$.
    Since $C$ is compact, there is a finite subcover, which we denote by $B(r_1, \frac{b_{r_1}}{2}), \cdots, B(r_m, \frac{b_{r_m}}{2})$.

    Fix any small enough $\delta > 0$. 
    For each $1 \leq i \leq m$, there exists a  polynomial $q_i \in \A$ such that $|q_i (x)| < \delta$ for all $x \in C \setminus B(r_i, b_{r_i})$ and $|q_i (x)| > 1 - \delta$ for all $x \in B(r_i, \frac{b_{r_i}}{2})$.
    Therefore, for a small enough $\delta$, the sum $\bff \coloneqq \sum_{i = 1}^m q_i \cdot \bff_{r_i} \in \mM$ satisfies $\bff(x) \in \Rpp^{n}$ for all $x \in C$.
\end{proof}

Denote by $\bff_{C}$ the element $\bff \in \mM$ obtained in Lemma~\ref{lem:fC}.
We also need the following theorem from Handelman:
\begin{thrm}[{Handelman's Theorem~\cite{de2001handelman}, \cite[V.6.\ Theorem C]{handelman1985positive}}]\label{lem:Handelman}
    Let $f \in \A$ be a polynomial.
    There exists $g \in \A^+$ such that $fg \in \A^+$ if and only if the two following conditions are satisfied:
    \begin{enumerate}[noitemsep, label = (\roman*)]
        \item For all $r \in \Rpp^n$, we have $f(r) > 0$.
        \item For all $v \in \Rns$ and $r \in \Rpp^n$, we have $\init_{v}(f)(r) > 0$.
    \end{enumerate}
\end{thrm}

\begin{cor}\label{cor:Handelman}
    Let $\bff \in \A^{K}$.
    There exists $g \in \A^+$ such that $g \bff \in \left(\A^+\right)^{K}$ if and only if the two following conditions are satisfied:
    \begin{enumerate}[noitemsep, label = (\roman*)]
        \item For all $r \in \Rpp^n$, we have $\bff(r) \in \Rpp^{K}$.
        \item For all $v \in \Rns$ and $r \in \Rpp^n$, we have $\init_{v}(\bff)(r) \in \Rpp^{K}$.
    \end{enumerate}
\end{cor}
\begin{proof}
    If there exists $g \in \A^+$ such that $g \cdot \bff \in \left(\A^+\right)^{K}$, then obviously (i) and (ii) are satisfied.

    One the other hand, let $\bff \in \A^{k}$ satisfying (i) and (ii).
    By Handelman's theorem (Theorem~\ref{lem:Handelman}), there exist $g_1, \ldots, g_{K} \in \A^{+}$ such that $f_1 g_1 \in \A^+, \ldots, f_{K} g_{K} \in \A^+$.
    Let $g \coloneqq g_1 g_2 \cdots g_{n}$, then $g \bff \in \left(\A^+\right)^{K}$.
\end{proof}

We are now ready to prove the ``if'' part of Theorem~\ref{thm:locglob} by ``gluing'' together the elements $\bff_{\infty}, \bff_C \in \mM$ obtained respectively in Lemma~\ref{lem:finf} and \ref{lem:fC}.
\begin{proof}[Proof of ``if'' part of Theorem~\ref{thm:locglob}]
    Let $\bff_{\infty}, \bff_C \in \mM$ be the elements obtained respectively in Lemma~\ref{lem:finf} and \ref{lem:fC}.
    Define the  polynomial
    \[
    q \coloneqq \frac{1}{2nc}\sum_{i = 1}^n(X_i + X_i^{-1}) \in \A.
    \]
    Let $\epsilon > 0$ be such that 
    \begin{equation}\label{eq:ineqepapp}
        \epsilon \cdot \bff_{\infty}(x) + \bff_{C}(x) \in \Rpp^{n}
    \end{equation}
    for all $x \in C$.
    Such a $\epsilon$ exists by the compactness of $C$.
    We claim that there exists $N \in \N$ such that the vector $\bff \coloneqq \epsilon q^N \cdot \bff_{\infty} + \bff_{C}$ satisfies Conditions~(i) and (ii) in Corollary~\ref{cor:Handelman} simultaneously.

    Let $M \in \N$ be such that $\deg_{v}(f_{\infty, i}) + M \cdot \min_{||w|| = 1} \deg_w(q) > \deg_{v}(f_{C, i})$ for all $v \in \Rns, ||v|| = 1$ and $i = 1, \ldots, K$.
    Such an $M$ exists by the compactness of the unit sphere.
    Let $\bg \coloneqq \epsilon q^M \cdot \bff_{\infty} + \bff_{C}$.
    Then for all $v \in \Rns, i = 1, \ldots, K,$ we have $\deg_v(\epsilon q^M \cdot f_{\infty, i}) = M \cdot \deg_v(q) + \deg_v(f_{\infty, i}) > \deg_v(f_{C, i})$.
    Therefore $\init_{v}(\bg) = \init_{v}(\bff_{\infty}) \in \ApK$ for all $v \in \Rns$.
    Therefore, there exists another compact set $[1/d, d]^n \supset C$ such that $\bg(x) \in \Rpp^{K}$ for all $x \in \Rpp^n \setminus [1/d, d]^n$.
    Since $[1/d, d]^n \supset C = [1/(4nc), 4nc]^n$, we have $d \geq 4nc$.
    Since $[1/d, d]^n$ is compact, there exists $N > M$ such that
    \begin{equation}\label{eq:ineq2Napp}
    \epsilon f_{\infty, i}(x) \cdot 2^{N} + f_{C, i}(x) > 0
    \end{equation}
    for all $i = 1, \ldots, K$, and all $x \in [1/d, d]^n$.
    We prove that for this $N$, the vector $\bff \coloneqq \epsilon q^N \cdot \bff_{\infty} + \bff_{C}$ satisfies Conditions~(i) and (ii) in Corollary~\ref{cor:Handelman}  simultaneously.

    Fix any $i \in \{1, \ldots, K\}$.
    For every $x \in \Rpp^n \setminus [1/d, d]^n$, we have $q(x) > \frac{d}{2nc} \geq 1$, so
    \[
    f_{i}(x) \coloneqq \epsilon q(x)^N \cdot f_{\infty, i}(x) + f_{C, i}(x) \geq \epsilon q(x)^M \cdot f_{\infty, i}(x) + f_{C, i}(x) = g_{i}(x) > 0.
    \]
    
    For every $x \in [1/d, d]^n \setminus C$, we have $x_{i'} \geq 4nc$ for at least one $i' \in \{1, \ldots, K\}$, so
    \[
    f_{i}(x) = \epsilon f_{\infty, i}(x) \cdot \left(\sum_{i' = 1}^n\frac{x_{i'} + x_{i'}^{-1}}{2nc}\right)^N + f_{C, i}(x) \geq \epsilon f_{\infty, i}(x) \cdot 2^N + f_{C, i}(x) > 0
    \]
    by Inequality~\eqref{eq:ineq2Napp} and $\sum_{i = 1}^n(x_i + x_i^{-1}) > x_{i'} \geq 4nc$.
    
    For every $x \in C \setminus [1/c, c]^n$,
    \[
    f_{i}(x) = \epsilon q(x)^N \cdot f_{\infty, i}(x) + f_{C, i}(x) > 0
    \]
    since $f_{\infty, i}(x) > 0$ for all $x \not\in [1/c, c]^n$ and $f_{C, i}(x) > 0$ for all $x \in C$.
    
    For every $x \in [1/c, c]^n$,
    \[
    f_{i}(x) = \epsilon f_{\infty, i}(x) \cdot \left(\sum_{i = 1}^n\frac{x_i + x_i^{-1}}{2nc}\right)^N + f_{C, i}(x) \geq \min\{\epsilon f_{\infty, i}(x), 0\} + f_{C, i}(x) > 0
    \]
    by $\sum_{i = 1}^n(x_i + x_i^{-1}) < 2nc$.
    In fact, if $f_{\infty, i}(x) \geq 0$ then $f_{\infty, i}(x) \cdot \left(\sum_{i = 1}^n\frac{x_i + x_i^{-1}}{2nc}\right)^N \geq 0$, otherwise $\left(\sum_{i = 1}^n\frac{x_i + x_i^{-1}}{2nc}\right)^N \leq 1$ so $f_{\infty, i}(x) \cdot \left(\sum_{i = 1}^n\frac{x_i + x_i^{-1}}{2nc}\right)^N \geq f_{\infty, i}(x)$.
    Therefore, for every $x \in \Rpp^n$, we have $f_i(x) > 0$.
    In other words, $\bff$ satisfies Conditions~(i) in Corollary~\ref{cor:Handelman}.
    
    Furthermore, since $N > M$ we have $\deg_{v}(q^N \cdot f_{\infty, i}) > \deg_{v}(f_{C, i})$ for $i = 1, \ldots, K, v \in \Rns$.
    Hence $\init_{v}(\bff) = \init_{v}(\bff_{\infty}) \in \ApK$ and $M_{v}(I, \bff) = M_{v}(I, \bff_{\infty})$ for all $v \in \Rns$.
    Therefore, $\bff$ satisfies Conditions~(ii) in Corollary~\ref{cor:Handelman}.
    
    Therefore, by Corollary~\ref{cor:Handelman}, we have find $g \in \A^+$ such that $g \bff \in \ApK$.
    We have at the same time $g \bff \in \mM$ as well as $\left(O_{v} \cup J\right) \cap M_{v}(I, g \bff) = \left(O_{v} \cup J\right) \cap M_{v}(I, \bff) = \left(O_{v} \cup J\right) \cap M_{v}(I, \bff_{\infty}) \neq \emptyset$ for all $v \in \Rns$.
    We have thus found the required element $\bff_v \coloneqq g \bff$.
\end{proof}

\section{Decidability of local conditions}\label{sec:dec}

This section is dedicated to the proof of Theorem~\ref{thm:dec}.
By the local-global principle (Theorem~\ref{thm:locglob}), this amounts to showing decidability of the two ``local'' Conditions~\hyperref[item:locr]{(LocR)} and \hyperref[item:locinf]{(LocInf)}.

\subsection{Decidability of local condition at positive reals \hyperref[item:locr]{(LocR)}}
In this subsection we show that the Condition~\hyperref[item:locr]{(LocR)} of Theorem~\ref{thm:locglob} is decidable.
Let $\mM$ be a $\A$-submodule of $\A^K$.

\begin{lem}\label{lem:equivr}
    Let $\bg_1, \ldots, \bg_m$ be the generators for $\mM$.
    Condition~\hyperref[item:locr]{(LocR)} of Theorem~\ref{thm:locglob} is equivalent to the following:
    \begin{enumerate}[noitemsep]
        \item[1.] \emph{\textbf{(LocRLin)}}\label{item:locrp} For every $r \in \Rpp^n$, there exist $x_1, \ldots, x_m \in \R$ such that $\sum_{i = 1}^m x_i \bg_i(r) \in \Rpp^K$.
    \end{enumerate}
\end{lem}
\begin{proof}
    \hyperref[item:locr]{(LocR)}$\implies$\hyperref[item:locrp]{(LocRLin)}:
    Suppose Condition~\hyperref[item:locr]{(LocR)} of Theorem~\ref{thm:locglob} is true.
    For $r \in \Rpp^n$, Condition~\hyperref[item:locr]{(LocR)} shows there exist $p_1, \ldots, p_m \in \A$ such that $\sum_{i = 1}^m p_i \bg_i = \bff_{r}$ with $\bff_{r}(r) \in \Rpp^K$.
    Then letting $x_1 \coloneqq p_1(r), \ldots, x_m \coloneqq p_m(r)$ we have $\sum_{i = 1}^m x_i \bg_i(r) = \bff_{r}(r) \in \Rpp^K$.

    \hyperref[item:locrp]{(LocRLin)}$\implies$\hyperref[item:locr]{(LocR)}:
    Suppose Condition~\hyperref[item:locrp]{(LocRLin)} is true.
    For any $r \in \Rpp^n$, Condition~\hyperref[item:locrp]{(LocRLin)} shows there exist $x_1, \ldots, x_m \in \Rpp$ such that $\sum_{i = 1}^m x_i \bg_i(r) \in \Rpp^K$.
    Then $\bff_{r} \coloneqq \sum_{i = 1}^m x_i \bg_i \in \mM$ satisfies $\bff_{r}(r) \in \Rpp^K$.
\end{proof}

\begin{prop}\label{prop:decr}
    Given the generators $\bg_1, \ldots, \bg_m$ for $\mM$, it is decidable whether Condition~\hyperref[item:locr]{(LocR)} of Theorem~\ref{thm:locglob} is satisfied.
\end{prop}
\begin{proof}
    By Lemma~\ref{lem:equivr}, it suffices to decide Condition~\hyperref[item:locrp]{(LocRLin)}.
    This is expressible in the first order theory of the reals:
    \[
    \forall r_1 > 0, \cdots \forall r_n > 0, \exists x_1 \cdots \exists x_m, \left(\sum_{i = 1}^m x_i g_{i, 1}(r_1, \ldots, r_n) > 0\right) \land \cdots \land \left(\sum_{i = 1}^m x_i g_{i, K}(r_1, \ldots, r_n) > 0\right).
    \]
    By Tarski's theorem~\cite{Tarski1949}, the truth of this sentence is decidable.
\end{proof}

\subsection{Local condition at infinity: from \hyperref[item:locinf]{(LocInf)} to shifted initials \hyperref[item:locshift]{(LocInfShift)}}

In this subsection we introduce the \emph{shifted initials}, in order to replace Condition~\hyperref[item:locinf]{(LocInf)} of Theorem~\ref{thm:locglob} with a new Condition~\hyperref[item:locshift]{(LocInfShift)}.
Our definition follows that of~\cite[Section~1]{einsiedler2003does}.

Suppose we are given $\bff \in \A^K$, $v \in \Rns$ and $\balpha = (\alpha_1, \ldots, \alpha_{K}) \in \R^{K}$.
Then the \emph{shifted initials}
$\init_{v, \alpha}(\bff) = (\init_{v, \alpha}(\bff)_1, \ldots, \init_{v, \alpha}(\bff)_K)$ is defined as
\begin{align*}
\init_{v, \alpha}(\bff)_i \coloneqq 
\begin{cases}
\init_{v}(f_i) & \text{ if } \deg_v(f_i) + \alpha_i = \max_{1 \leq i' \leq K} \{\deg_v(f_{i'}) + \alpha_{i'}\},\\
0 & \text{ if } \deg_v(f_i) + \alpha_i < \max_{1 \leq i' \leq K} \{\deg_v(f_{i'}) + \alpha_{i'}\}. \\
\end{cases}
\end{align*}

\begin{lem}\label{lem:initpos}
    Let $\bff \in \A^K$ and $v \in \Rns$.
    We have $\init_v(\bff) \in \ApK$ if and only if there exists $\balpha \in \R^{K}$ such that $\init_{v, \balpha}(\bff) \in \ApK$.
    Furthermore, in this case, $\init_v(\bff) = \init_{v, \balpha}(\bff)$ and $\alpha_1 + \deg_v(f_1) = \cdots = \alpha_K + \deg_v(f_K)$.
\end{lem}
\begin{proof}
If $\init_{v}\left(\bff\right) \in \ApK$, then let
$
\alpha_1 \coloneqq - \deg_v(f_1), \ldots, \alpha_K \coloneqq - \deg_v(f_K)
$.
We have $\deg_v(f_1) + \alpha_1 = \cdots = \deg_v(f_K) + \alpha_K = \max_{1 \leq i \leq K} \{\deg_v(f_i) + \alpha_i\}$, so $\init_{v, \balpha}\left(\bff\right) = \init_{v}\left(\bff\right) \in \ApK$.

If $\init_{v, \balpha}\left(\bff\right) \in \ApK$, then $\init_{v}\left(f_i\right) = \init_{v, \balpha}(\bff)_i \in \A^+$ for $i = 1, \ldots, K$.
Therefore $\init_{v}\left(\bff\right) = \init_{v, \balpha}(\bff) \in \ApK$.

Furthermore, in this case, since $\init_{v, \balpha}\left(\bff\right)_i \neq 0$ for $i = 1, \ldots, K$, we have $\deg_v(f_i) + \alpha_i = \max_{1 \leq i' \leq K} \{\deg_v(f_{i'}) + \alpha_{i'}\}$.
Hence $\alpha_1 + \deg_v(f_1) = \cdots = \alpha_K + \deg_v(f_K)$.
\end{proof}

Given $v = (v_1, \ldots, v_n) \in \Rns$, define the following set of real numbers:
\[
\sum_{k = 1}^n \Z v_k \coloneqq \left\{\sum_{k = 1}^n z_k v_k \;\middle|\; z_1, \ldots, z_n \in \Z \right\}.
\]
Then for every $f \in \A$, we have $\deg_{v}(f) \in \sum_{k = 1}^n \Z v_k$.

\begin{prop}\label{prop:inftoshift}
Condition~\emph{\hyperref[item:locinf]{(LocInf)}} of Theorem~\ref{thm:locglob} is equivalent to the following:

\begin{enumerate}[noitemsep]
    \item[2.] \label{item:locshift} \emph{\textbf{(LocInfShift):}}  For every $v \in \Rns$, there exists $\bff \in \mM$ as well as $\balpha \in \left(\sum_{k = 1}^n \Z v_k\right)^K$ satisfying the following properties: 
    \begin{enumerate}[noitemsep]
        \item $\init_{v, \balpha}\left(\bff\right) \in \ApK$.
        \item Denote $I' \coloneqq \{i \in I \mid \alpha_{i} = \min_{i' \in I} \alpha_{i'}\}, J' \coloneqq O_{v} \cup J$.
        We have 
        \[
        \left(O_{w} \cup J'\right) \cap M_{w}(I', \init_{v, \balpha}(\bff)) \neq \emptyset \quad \text{ for every $w \in \Rns$}.
        \]
    \end{enumerate}
\end{enumerate}
\end{prop}
\begin{proof}
    \textbf{\hyperref[item:locinf]{(LocInf)}}$\implies$\textbf{\hyperref[item:locshift]{(LocInfShift)}}.
    Suppose Condition~\hyperref[item:locinf]{(LocInf)} of Theorem~\ref{thm:locglob} is true.
    Fix a vector $v \in \Rns$.
    Then there exists $\bff \in \mM$, such that $\init_{v}\left(\bff\right) \in \ApK$ satisfies Property~\hyperref[item:locinf]{(LocInf)}(b).
    As in Lemma~\ref{lem:initpos}, we can let $\alpha_i \coloneqq - \deg_v(f_{i})$ for $i = 1, \ldots, K$, then $\init_{v, \balpha}\left(\bff\right) = \init_{v}\left(\bff\right) \in \ApK$, satisfying \hyperref[item:locshift]{(LocInfShift)}(a).
    Furthermore, we have $\balpha \in \left(\sum_{k = 1}^n \Z v_k\right)^K$ by the definition of $\alpha_i = - \deg_v(f_{i})$.
    Finally, $\{i \in I \mid \alpha_{i} = \min_{i' \in I} \alpha_{i'}\} = \{i \in I \mid \deg_v(f_{i}) = \max_{i' \in I} \deg_v(f_{i'})\} = M_v(I, \bff)$, so \hyperref[item:locinf]{(LocInf)}(b) implies \hyperref[item:locshift]{(LocInfShift)}(b).

    \textbf{\hyperref[item:locshift]{(LocInfShift)}}$\implies$\textbf{\hyperref[item:locinf]{(LocInf)}}.
    Suppose Condition~\hyperref[item:locshift]{(LocInfShift)} is true.
    Fix a vector $v \in \Rns$.
    Then there exists $\bff \in \mM$ as well as $\balpha \in \Rns$, such that $\init_{v, \balpha}\left(\bff\right) \in \ApK$ satisfies Property~\hyperref[item:locshift]{(LocInfShift)}(b).
    By Lemma~\ref{lem:initpos}, we have $\init_{v}\left(\bff\right) = \init_{v, \balpha}\left(\bff\right) \in \ApK$, and $\alpha_1 + \deg_v(f_{1}) = \cdots = \alpha_K + \deg_v(f_{K})$.
    Therefore we have $\{i \in I \mid \alpha_{i} = \min_{i' \in I} \alpha_{i'}\} = \{i \in I \mid \deg_v(f_{i}) = \min_{i' \in I} \deg_v(f_{i'})\} = M_v(I, \bff)$, so \hyperref[item:locshift]{(LocInfShift)}(b) implies \hyperref[item:locinf]{(LocInf)}(b).
\end{proof}

\subsection{Dimension reduction: a special case}
In this and the following subsections we will further reduce Condition~\hyperref[item:locshift]{(LocInfShift)} to a Condition~\hyperref[item:locd]{(LocInfD)} (see Proposition~\ref{prop:shifttod}).
In this subsection we first consider the special case where the vector $v \in \Rns$ in Condition~\hyperref[item:locshift]{(LocInfShift)} is of the form $(0, \ldots, 0, v_{d+1}, \ldots, v_n)$, where $v_{d+1}, \ldots, v_n \in \R$ are $\Q$-linearly independent.

We now define the \emph{super Gr\"{o}bner basis} of an $\A$-module $\mM$.
Note that our definition is different from~\cite[Section~2]{einsiedler2003does}, although the intuition is the same.
Let $v \in \Rns, \balpha \in \R^K$.
Define $\init_{v, \alpha}(\mM)$ to be the $\A$-module generated by the elements $\init_{v, \alpha}(\bff), \bff \in \mM$:
\[
\init_{v, \alpha}(\mM) \coloneqq \sum_{\bff \in \mM} \A \cdot \init_{v, \alpha}(\bff) = \left\{\sum_{j = 1}^q p_j \cdot \init_{v, \alpha}(\bff_j) \;\middle|\; q \in \N, p_1, \ldots p_q \in \A, \bff_1, \ldots, \bff_q \in \mM\right\}.
\]

\begin{defn}[Super Gr\"{o}bner basis]\label{def:grob}
    A set of generators $\bg_1, \ldots, \bg_m$ for the module $\mM$ is called a \emph{super Gr\"{o}bner basis} if for all $v \in \Rns, \balpha \in \R^K$, the set $\{\init_{v, \alpha}(\bg_1), \ldots, \init_{v, \alpha}(\bg_m)\}$ generates $\init_{v, \alpha}(\mM)$ as an $\A$-module.
\end{defn}

\begin{lem}[{Reformulation of~\cite[Lemma~2.1]{einsiedler2003does}}]\label{lem:grob}
    Suppose we are given an arbitrary set of generators for a module $\mM$. Then a super Gr\"{o}bner basis of $\mM$ is effectively computable.
\end{lem}
\begin{proof}
	Let $\R[\oX] \coloneqq \R[X_1, \ldots, X_n]$ be the regular polynomial ring over $n$ variables (note that for the moment, we are not considering the Laurent polynomial ring $\A = \R[\oX^{\pm}]$).
    Let $e_1, \ldots, e_K$ be the canonical $\R[\oX]$-basis of $\R[\oX]^K$.
    A \emph{monomial} of $\R[\oX]^K$ is an element of the form $\oX^u e_i$ for some $u \in \Zp^n, i \in \{1, \ldots, K\}$.
    A \emph{term order} on the monomials of $\R[\oX]^K$ is a total order $\prec$ satisfying 
	\begin{enumerate}[noitemsep, label = (\roman*)]
	\item $e_i \prec \oX^u e_i$,
	\item $\oX^u e_i \prec \oX^{u'} e_{i'} \implies \oX^{u+w} e_i \prec \oX^{u'+w} e_{i'}$,
	\end{enumerate}	    
    for all $i, i' \in \{1, \ldots, K\}$ and $u, u', w \in \Zp^n$.
    
    Let $N$ be an $\R[\oX]$-submodule of $\R[\oX]^K$. An element $\bff \in \R[\oX]^K$ can be written uniquely as a sum $\sum_{u, i} c_{u, i} \oX^u e_i$ with coefficients $c_{u, i}$ in $\R$.
    Among the finitely many monomials of $\R[\oX]^K$ that have nonzero coefficients in this sum, the one that is maximal according to the term order $\prec$ is denoted $\init_{\prec}(\bff)$.
    Define $\init_{\prec}(N)$, the \emph{initial module} of $N$ with respect to $\prec$, to be the $\R[\oX]$-module generated by all $\init_{\prec}(\bff), \bff \in N$.
    We say that the elements $\bff_1, \ldots, \bff_l \in N$ form a \emph{Gr\"{o}bner basis} for $N$ with respect to $\prec$ if $\init_{\prec}(N)$ is generated as an $\R[\oX]$-module by $\init_{\prec}(\bff_1), \ldots, \init_{\prec}(\bff_l)$.
    
    A \emph{universal Gr\"{o}bner basis} of $N$ is given by elements $\bff_1, \ldots, \bff_l$ that form a Gr\"{o}bner basis of $N$ with respect to \emph{every} term order.
    A universal Gr\"{o}bner basis of $N$ exists and can be effectively computed from a set of generators of $N$ as described in~\cite[Lemma~2.1]{einsiedler2003does}.
    
    Fix an $\R[\oX^{\pm}]$-submodule $\mM$ of $\left(\R[\oX^{\pm}]\right)^K$.
    Find $\bff_1, \ldots, \bff_l$ that generate $\mM$ as an $\R[\oX^{\pm}]$-module.
    Let $\delta = (\delta_1, \ldots, \delta_n) \in \{-1, 1\}^n$.
    Pick $u \in \Z^n$ such that 
    \[
    \oX^u \bff_1, \ldots, \oX^u \bff_l \in \left(\R[X_1^{\delta_1}, \ldots, X_n^{\delta_n}]\right)^K,
    \]
    and let $\bff_{\delta, 1}, \ldots, \bff_{\delta, l_{\delta}}$ be a universal Gr\"{o}bner basis for the $\R[X_1^{\delta_1}, \ldots, X_n^{\delta_n}]$-submodule generated by $\oX^u \bff_1, \ldots, \oX^u \bff_l$.
    List the union of $\{\bff_{\delta, 1}, \ldots, \bff_{\delta, l_{\delta}}\}$ over $\delta \in \{-1, 1\}^n$ as $\bg_1, \ldots, \bg_m$.
    By~\cite[Lemma~2.1]{einsiedler2003does}, this is a super Gr\"{o}bner basis for $\mM$.
\end{proof}

It is easy to see the following from the proof: if the given generators for $\mM$ contain only polynomials with integer coefficients, then Lemma~\ref{lem:grob} computes a super Gr\"{o}bner basis containing only polynomials with integer coefficients.

Let $0 \leq d \leq n-1$ be an integer.
From now on we denote 
\[
\A_d \coloneqq \R[X_1^{\pm}, \ldots, X_d^{\pm}], \quad \A_d^+ \coloneqq \Rp[X_1^{\pm}, \ldots, X_d^{\pm}]^*.
\]
In particular, $\A_0 = \R, \A_0^+ = \Rpp$.

As stated in the beginning of this subsection, we now consider the vectors $v \in \Rns$ with the special form $(0, \ldots, 0, v_{d+1}, \ldots, v_n)$ where $v_{d+1}, \ldots, v_n$ are $\Q$-linearly independent.
The following lemma can be seens as a generalization of~\cite[Lemma~6.2]{einsiedler2003does}.

\begin{lem}[{Generalization of~\cite[Lemma~6.2]{einsiedler2003does}}]\label{lem:ing}
    Let $\bg_1, \ldots, \bg_m$ be a super Gr\"{o}bner basis of $\mM$.
    Let $v = (0, \ldots, 0, v_{d+1}, \ldots, v_n) \in \Rns$ be such that $0 \leq d \leq n-1$ and $v_{d+1}, \ldots, v_n$ are $\Q$-linearly independent.
    Let $\balpha \in \R^K$.
    Then there exists $b_i \in \{0\}^d \times \Z^{n-d}$ and $c_j \in \{0\}^d \times \Z^{n-d}$ such that $\oX^{b_i} \oX^{c_j} \init_{v, \balpha}(\bg_j)_i \in \A_d$ for $i = 1, \ldots, K$ and $j = 1, \ldots, m$.
\end{lem}
\begin{proof}
    Let $j \in \{1, \ldots, m\}, i \in \{1, \ldots, K\},$ be such that $\init_{v, \balpha}(\bg_j)_i \neq 0$.
    Since $v_{d+1}, \ldots, v_n$ are $\Q$-linearly independent, there exists an open neighbourhood $U \subseteq \R^{n-d}$ of $(v_{d+1}, \ldots, v_n)$, such that every $v' \in \{0\}^d \times U$ satisfies $\init_{v'}(\init_{v, \balpha}(\bg_j)_i) = \init_{v}(\init_{v, \balpha}(\bg_j)_i)$.
    This shows that $\oX^{z_{ij}} \init_{v, \balpha}(\bg_j)_i \in \A_d$ for some $z_{ij} \in \{0\}^d \times \Z^{n-d}$.

    Letting $F \coloneqq \{(i, j) \mid \init_{v, \balpha}(\bg_j)_i \neq 0\}$, this defines $z_{ij}$ for all $(i, j) \in F$.
    Note that for $(i, j) \in F$ we have
    \[
    \max_{1 \leq i' \leq K} \{\deg_{v}(\bg_{j, i'}) + \balpha_{i'}\} = - v \cdot z_{ij} + \alpha_i.
    \]
    Considering a sequence
    \begin{equation}\label{eq:allowseq}
        (i_0, j_0), (i_1, j_0), (i_1, j_1), \ldots, (i_l, j_{l-1}), (i_l, j_l), (i_0, j_l)
    \end{equation}
    in $F$, and writing $i_{l+1} = i_0$, we find that
    \[
    0 = \sum_{s = 0}^l \left(\max_{1 \leq i' \leq K} \{\deg_{v}(\bg_{j, i'}) + \balpha_{i'}\} - \max_{1 \leq i' \leq K} \{\deg_{v}(\bg_{j, i'}) + \balpha_{i'}\}\right) = \sum_{s = 0}^l \left(v \cdot z_{i_{s+1} j_s} - v \cdot z_{i_s j_s}\right).
    \]
    Since $v = (0, \ldots, 0, v_{d+1}, \ldots, v_n)$ with $v_{d+1}, \ldots, v_n$ being $\Q$-linearly independent and $z_{ij} \in \{0\}^d \times \Z^{n-d}$,
    the above equation yields
    \begin{equation}\label{eq:sumzero}
    \sum_{s = 0}^l \left(z_{i_{s+1} j_s} - z_{i_s j_s}\right) = 0
    \end{equation}
    for every allowed sequence~\eqref{eq:allowseq} in $F$.

    We now extend $z_{ij}$ and Equation~\eqref{eq:sumzero} to all pairs $(i, j) \in \{1, \ldots, K\} \times \{1, \ldots, m\}$.
    Assume $z_{ij}$ is already defined on a set $E \supseteq F$ and \eqref{eq:sumzero} is valid on $E$.
    Pick $(i,j) \not\in E$.
    If there exists a sequence
    \[
    (i_1, j), (i_1, j_1), \ldots, (i_l, j_{l-1}), (i_l, j_l), (i, j_l) \in E,
    \]
    we put $i_{l+1} = i$ and define
    \[
    z_{ij} \coloneqq z_{i_1 j} - \sum_{s = 1}^l \left(z_{i_s j_s} - z_{i_{s+1} j_s}\right).
    \]
    One easily verifies that Equation~\eqref{eq:sumzero} then holds for every allowed sequence~\eqref{eq:allowseq} in $E \cup \{(i,j)\}$.

    If there is no sequence
    \[
    (i_1, j_0), (i_1, j_1), \ldots, (i_l, j_{l-1}), (i_l, j_l), (i_0, j_l)
    \]
    in $E$, we can take $z_{ij}$ to be any element of $\{0\}^d \times \Z^{n-d}$ and have Equation~\eqref{eq:sumzero} hold for all sequences~\eqref{eq:allowseq} in $E \cup \{(i, j)\}$.

    Having thus extended $z_{ij}$ to all pairs $(i, j) \in \{1, \ldots, K\} \times \{1, \ldots, m\}$, we define
    \[
    b_i \coloneqq z_{i1},
    \]
    and
    \[
    c_j \coloneqq z_{ij} - z_{i1},
    \]
    which is independent of $i$ thanks to Equation~\eqref{eq:sumzero}. Indeed, using Equation~\eqref{eq:sumzero} on the allowed sequence $(i, j), (i', j), (i', 1), (i, 1)$ we get $z_{ij} - z_{i1} = z_{i'j} - z_{i'1}$.
    Hence, $z_{ij} = b_i + c_j$ and the lemma follows.
\end{proof}

Suppose $v \in \Rns$ is such that $v = (0, \ldots, 0, v_{d+1}, \ldots, v_n)$ with $v_{d+1}, \ldots, v_n$ being $\Q$-linearly independent.
Let $\balpha \in \R^K$.
For each $j = 1, \ldots, m$, define $\init_{v, \balpha}^{d}(\bg_j) = (\init_{v, \balpha}^{d}(\bg_j)_1, \ldots, \init_{v, \balpha}^{d}(\bg_j)_K)$ where
\[
\init_{v, \balpha}^{d}(\bg_j)_i \coloneqq \oX^{b_i} \oX^{c_j} \init_{v, \balpha}(\bg_j)_i \in \A_d, \quad i = 1, \ldots, K.
\]
Here, $b_i$ and $c_j$ are defined as in Lemma~\ref{lem:ing}.
Note that the vectors $b_i, c_j \in \{0\}^d \times \Z^{n-d}$ are not necessarily uniquely determined. However, when $d, v, \balpha$ are fixed, the polynomials $\init_{v, \balpha}^{d}(\bg_j)_i$ are uniquely determined by $\bg_1, \ldots, \bg_m$.
In fact, by Lemma~\ref{lem:ing} each $\init_{v, \balpha}(\bg_j)_i$ can be uniquely written as $\oX^{s} \cdot p$ for some $\oX^{s} \in \R[X_{d+1}^{\pm}, \ldots, X_{n}^{\pm}]$ and $p \in \R[X_{1}^{\pm}, \ldots, X_{d}^{\pm}]$.
Therefore $\init_{v, \balpha}^{d}(\bg_j)_i$ is uniquely determined as the polynomial $p$ in the decomposition.

Note that if $\init_{v, \balpha}^{d}(\bg_j)_i \neq 0$ then $\deg_{v}( g_{j, i}) = - v \cdot (b_i + c_j)$, otherwise $\deg_{v}(g_{j, i}) < - v \cdot (b_i + c_j)$.
In both cases,
\begin{equation}\label{eq:vbc}
    \deg_{v}(g_{j, i}) \leq - v \cdot (b_i + c_j),
\end{equation}
where the equality holds if and only if $\init_{v, \balpha}^{d}(\bg_j)_i \neq 0$.

\begin{lem}\label{lem:align}
    If there exists $j \in \{1, \ldots, m\}$ such that $\init^{d}_{v, \balpha}(\bg_{j})_i \neq 0, \init^{d}_{v, \balpha}(\bg_{j})_{i'} \neq 0$, then $\alpha_i - v \cdot b_i = \alpha_{i'} - v \cdot b_{i'}$.
\end{lem}
\begin{proof}
    If $\init^{d}_{v, \balpha}(\bg_{j})_i \neq 0, \init^{d}_{v, \balpha}(\bg_{j})_{i'} \neq 0$, then $\init_{v, \balpha}(\bg_{j})_i \neq 0, \init_{v, \balpha}(\bg_{j})_{i'} \neq 0$, so $\deg_{v}(\init_{v, \balpha}(\bg_{j})_i) + \alpha_i = \deg_{v}(\init_{v, \balpha}(\bg_{j})_{i'}) + \alpha_{i'}$.

    But $\deg_{v}(\init_{v, \balpha}(\bg_{j})_i) = \deg_{v}(\oX^{-b_i} \oX^{-c_j} \cdot \init_{v, \balpha}^{d}(\bg_{j})_i) =  - v \cdot(b_i + c_j)$.
    Deriving the same equation for $i'$ we have $- v \cdot(b_i + c_j) + \alpha_i = - v \cdot(b_{i'} + c_j) + \alpha_{i'}$.
    This yields $\alpha_i - v \cdot b_i = \alpha_{i'} - v \cdot b_{i'}$.
\end{proof}

Define by $\init_{v, \balpha}^{d}(\mM)$ the $\A_d$-module generated by $\init_{v, \balpha}^{d}(\bg_1), \ldots, \init_{v, \balpha}^{d}(\bg_m) \in \A_d^K$:
\[
\init_{v, \balpha}^{d}(\mM) \coloneqq \sum_{j = 1}^m \A_d \cdot \init_{v, \balpha}^{d}(\bg_j) =\left\{\sum_{j = 1}^m p_j \cdot \init_{v, \balpha}^{d}(\bg_j) \;\middle|\; p_1, \ldots p_m \in \A_d \right\}.
\]
A key component of proving the original result of Einsiedler et al.~\cite{einsiedler2003does} is \cite[Lemma~3.2]{einsiedler2003does}, which shows that $\init_{v, \balpha}^{d}(\mM) \cap \left(\A_d^+\right)^K \neq \emptyset$ implies $\mM \cap \ApK \neq \emptyset$.
However, this fails when we additionally impose Property~\eqref{eq:globv}: an element in $\mM \cap \ApK$ satisfying Property~\eqref{eq:globv} might not be obtained from an element in $\init_{v, \balpha}^{d}(\mM) \cap \left(\A_d^+\right)^K$ satisfying a similar property.
Indeed, if we directly apply \cite[Lemma~3.2]{einsiedler2003does} to our situation, the main caveat would be in the last paragraph of the proof, where different ``levels'' of polynomials are combined together to create a positive element.
This no longer works if we add in degree constraints.
The following lemma shows that \cite[Lemma~3.2]{einsiedler2003does} can still be made partially compatible with Property~\eqref{eq:globv}, if we impose the additional constraint $\balpha \in \left(\sum_{k = d+1}^n \Z v_k\right)^K$.

\begin{lem}\label{lem:dimred}
    Let $\bg_1, \ldots, \bg_m$ be a super Gr\"{o}bner basis of $\mM$.
    Let  $v = (0, \ldots, 0, v_{d+1}, \ldots, v_n)$ be such that $0 \leq d \leq n-1$ and $v_{d+1}, \ldots, v_n$ are $\Q$-linearly independent.
    Let $\balpha \in \left(\sum_{k = d+1}^n \Z v_k\right)^K$.
    Denote $I' \coloneqq \{i \in I \mid \alpha_{i} = \min_{i' \in I} \alpha_{i'}\}, J' \coloneqq O_{v} \cup J$.
    Denote by $\pi_d \colon \Z^n \rightarrow \Z^d$ the projection onto the first $d$ coordinates.
    For every $u \in (\R^{d})^*$, define $O'_u \coloneqq \{i \in \{1, \ldots, K\} \mid \pi_d(a_i) \not\perp u\}$.
    Then the two following conditions are equivalent:
        \begin{enumerate}[label = (\roman*)]
            \item (Condition in \emph{\hyperref[item:locshift]{(LocInfShift)}}): There exists $\bff \in \mM$ such that $\init_{v, \balpha}\left(\bff\right) \in \ApK$ and
            \begin{equation}\label{eq:condirrshift}
            \left(O_{w} \cup J'\right) \cap M_{w}(I', \init_{v, \balpha}(\bff)) \neq \emptyset \quad \text{ for every $w \in \Rns$}.
            \end{equation}
            \item We have $J' \cap I' \neq \emptyset$, and there exists $\bff^d \in \init_{v, \balpha}^{d}(\mM) \cap \left(\A_d^+\right)^K$, such that
            \begin{equation}\label{eq:condirr2}
            \left(O'_{u} \cup J'\right) \cap M_{u}(I', \bff^d) \neq \emptyset \quad \text{ for every $u \in (\R^{d})^*$}.
            \end{equation}
        \end{enumerate}
    When $d = 0$, Property~\eqref{eq:condirr2} is considered trivially true.
\end{lem}
\begin{proof}
    \textbf{(i)$\implies$(ii).}
    Suppose (i) holds. Let $\bff \in \mM \cap \ApK$ satisfy \eqref{eq:condirrshift}.
    We now show (ii).
    
    The property $J' \cap I' \neq \emptyset$ follows from $\init_{v, \balpha}\left(\bff\right) \in \ApK$ and \eqref{eq:condirrshift} by taking $w \coloneqq v$.
    Indeed, we have $O_v \cup J' = O_v \cup J$ and 
    \begin{multline*}
    M_{v}(I', \init_{v, \balpha}(\bff)) = \left\{i \in I' \;\middle|\; \deg_v(\init_{v, \balpha}(\bff)_i) = \max_{i' \in I'} \deg_v(\init_{v, \balpha}(\bff)_{i'})\right\} \\
    = \left\{i \in I' \;\middle|\; - \alpha_i = \max_{i' \in I'} (- \alpha_{i'})\right\} = I',
    \end{multline*}
    where the second equality comes from $\init_{v, \balpha}\left(\bff\right) \in \ApK$.
    Therefore, Property~\eqref{eq:condirrshift} yields $J' \cap I' \neq \emptyset$ by taking $w \coloneqq v$.
    
    Since $\bg_1, \ldots, \bg_m$ is a super Gr\"{o}bner basis, we can write 
    \begin{equation}\label{eq:suminit}
    \init_{v, \balpha}(\bff) = \sum_{j = 1}^m h_j \cdot \init_{v, \balpha}(\bg_j).
    \end{equation}
    Let 
    \[
    S \coloneqq \left\{1 \leq j \leq m \;\middle|\; \deg_v(h_j) + \max_{1 \leq i \leq K}(\deg_{v}(\bg_{j,i}) + \alpha_i) \text{ is maximal} \right\}.
    \]
    Without loss of generality suppose $\sum_{j \in S} \init_v(h_j) \cdot \init_{v, \balpha}(\bg_j) \neq 0$, otherwise we can replace each $h_j$ with $j \in S$ by $h_j - \init_v(h_j)$ while \eqref{eq:suminit} still holds.
    Then we have 
    \[
    \init_{v, \balpha}(\bff) = \init_{v, \balpha}\left(\sum_{j = 1}^m h_j \cdot \init_{v, \balpha}(\bg_j)\right) = \sum_{j \in S} \init_v(h_j) \cdot \init_{v, \balpha}(\bg_j).
    \]
    Indeed, by the definition of the shifted initial $\init_{v, \balpha}$, the right hand side above are the only elements that can contribute to the shifted initial of the sum in the middle.
    Hence, without loss of generality we can suppose $S = \{1, \ldots, m\}$ and $h_j = \init_v(h_j)$ for all $j = 1, \ldots, m$.
    Denote $D \coloneqq \deg_v(h_j) + \max_{1 \leq i' \leq K}(\deg_{v}(\bg_{j,i'}) + \alpha_{i'})$; this does not depend on the choice of $j \in \{1, \ldots, m\}$.


    Since $v = (0, \ldots, 0, v_{d+1}, \ldots, v_n)$ such that $v_{d+1}, \ldots, v_n$ are $\Q$-linearly independent and $h_j = \init_v(h_j)$,
    we can write $h_j = \oX^{z_{j}} p_{j}$, where $p_{j} \in \A_d$ and $z_{j} \in \{0\}^d \times \Z^{n-d}$.
    Note that $D = \deg_v(h_j) + \max_{1 \leq i' \leq K}(\deg_{v}(\bg_{j,i'}) + \alpha_{i'}) = v \cdot z_j - v \cdot (b_i + c_j) + \alpha_i$ for all $(i, j)$ satisfying $\init^d_{v, \balpha}(\bg_j)_i \neq 0$.
    Since $\balpha \in \left(\sum_{k = d+1}^n \Z v_k\right)^K$, each $\alpha_i, i = 1, \ldots, K,$ can be written as $\alpha_i = v \cdot z'_i$ for some $z'_i \in \{0\}^d \times \Z^{n-d}$.
    So $D = v \cdot (z_j - b_i - c_j + z'_i)$ for all $(i, j)$ satisfying $\init^d_{v, \balpha}(\bg_j)_i \neq 0$.
    
    By the $\Q$-linear independence of the entries of $v$, there exists a single $z \in \{0\}^d \times \Z^{n-d}$ such that $z = z_j - b_i - c_j + z'_i$ for all $(i, j)$ satisfying $\init^d_{v, \balpha}(\bg_j)_i \neq 0$.
    Then 
    \begin{equation*}
    \init_{v, \balpha}(\bff)_i = \sum_{j = 1}^m \oX^{z_j} p_{j} \init_{v, \balpha}(\bg_j)_i = \sum_{j = 1}^m p_{j} \oX^{z_j -b_i - c_j} \init_{v, \balpha}^{d}(\bg_j)_i
    = \oX^{z - z'_i} \sum_{j = 1}^m p_{j} \init_{v, \balpha}^{d}(\bg_j)_i
    \end{equation*}
    Let
    \[
    \bff^d \coloneqq \sum_{j = 1}^m p_{j} \init_{v, \balpha}^{d}(\bg_j) \in \init_{v, \balpha}^{d}(\mM).
    \]
    Then for $i = 1, \ldots, K$,
    \[
    f^d_i = \oX^{z'_i -z} \init_{v, \balpha}(\bff)_i \in \A^+ \cap \A_d = \A_d^+.
    \]

    Therefore $\bff^d \in \left(\A_d^+\right)^K$.
    It is left to show that $\bff^d$ satisfies Property~\eqref{eq:condirr2}.
    If $d = 0$ Property~\eqref{eq:condirr2} is trivially true.
    Suppose $d \geq 1$.
    For each $u \in (\R^{d})^*$, let $w \coloneqq (u, 0^{n-d})$ in~\eqref{eq:condirrshift}.
    Then
    \[
    \left(O_{w} \cup J'\right) \cap M_{w}(I', \init_{v, \balpha}(\bff)) \neq \emptyset.
    \]

    Also, for each $i \in \{1, \ldots, K\}$, because $(z - z'_i) \in \{0\}^d \times \Z^{n-d}$ and $w \in \Z^d \times \{0\}^{n-d}$ we have $\deg_w(\init_{v, \balpha}(\bff)_i) = \deg_w(\oX^{z - z'_i} \bff^d_i) = \deg_u(\bff^d_i)$.
    Hence,
    \begin{multline*}
    M_{w}(I', \init_{v, \balpha}(\bff)) = \left\{i \in I' \;\middle|\; \deg_w(\init_{v, \balpha}(\bff)_i) = \max_{i' \in I'} \deg_w(\init_{v, \balpha}(\bff)_{i'})\right\} \\
    = \left\{i \in I' \;\middle|\; \deg_u(f^d_i) = \max_{i' \in I'} \deg_u(f^d_{i'})\right\} = M_{u}(I', \bff^d).
    \end{multline*}
    Also, since the last $n-d$ entries of $v$ are $\Q$-linearly independent, we have $a_i \perp v$ if and only if $a_i \in \sum_{k = 1}^d \Z e_k$, and
    \begin{multline}\label{eq:Ou}
        O'_u \cup J' = O'_u \cup O_v \cup J = \{i \mid \neg (\pi_d(a_i) \perp u \land a_i \perp v) \} \cup J  \\
        = \left\{i \;\middle|\; \neg \left(\pi_d(a_i) \perp u \land a_i \in \sum_{i = 1}^d \Z e_i\right) \right\} \cup J = \left\{i \;\middle|\; \neg \left(a_i \perp w \land a_i \in \sum_{i = 1}^d \Z e_i \right) \right\} \cup J \\
        = \{i \mid \neg (a_i \perp w \land a_i \perp v) \} \cup J = O_w \cup O_v \cup J = O_w \cup J'
    \end{multline}
    Therefore $\left(O'_{u} \cup J'\right) \cap M_{u}(I', \bff^d) = \left(O_{w} \cup J'\right) \cap M_{w}(I', \init_{v, \balpha}(\bff))$, which is non-empty by \eqref{eq:condirrshift}.
    We have thus shown that $\bff^d$ satisfies Property~\eqref{eq:condirr2}.

    \textbf{(ii)$\implies$(i).}
    Suppose (ii) holds.
    Write $\bff^d = \sum_{j = 1}^m p_j \init_{v, \balpha}^{d}(\bg_j)$ where $p_j \in \A_d$ for $j = 1, \ldots, m$.

    For each $i \in \{1, \ldots, K\}$, write $\alpha_i = v \cdot z_i$ for some $z_i \in \{0\}^d \times \Z^{n-d}$.
    By the $\Q$-linear independence of the entries of $v$, such a $z_i$ is unique.

    For each $j \in \{1, \ldots, m\}$, take any $i_j \in \{1, \ldots, K\}$ such that $\init_{v, \balpha}^{d}(\bg_j)_{i_j} \neq 0$, note that the vector $c_j + b_{i_j} - z_{i_j}$ does not depend on the choice of $i_j$.
    Indeed, take any other $i'_j \in \{1, \ldots, K\}$ such that $\init_{v, \balpha}^{d}(\bg_j)_{i'_j} \neq 0$, then by Lemma~\ref{lem:align} we have $\alpha_{i_j} - v \cdot b_{i_j} = \alpha_{i'_j} - v \cdot b_{i'_j}$.
    Since $\alpha_{i_j} = v \cdot z_{i_j}, \alpha_{i'_j} = v \cdot z_{i'_j}$ we have $v \cdot (z_{i_j} - b_{i_j}) =  v \cdot (z_{i'_j} - b_{i'_j})$.
    By the $\Q$-linear independence of the entries of $v$, we have $z_{i_j} - b_{i_j} = z_{i'_j} - b_{i'_j}$.
    So the vector $c_j + b_{i_j} - z_{i_j}$ does not depend on the choice of $i_j$.
    
    For all $i \in \{1, \ldots, K\}$, we have
    \begin{equation}\label{eq:va}
    \deg_v(g_{j, i}) + \alpha_{i} \leq \deg_v(g_{j, i_j}) + \alpha_{i_j},
    \end{equation}
    where the equality holds if and only if $\init_{v, \balpha}^d(\bg_j)_{i} \neq 0$.

    Take
    \[
    \bff \coloneqq \sum_{j = 1}^m \oX^{c_j + b_{i_j} - z_{i_j}} p_j \cdot \bg_j \in \mM.
    \]
    For each $i \in \{1, \ldots, K\}$ and $j \in \{1, \ldots, m\}$, we have
    \begin{multline}\label{eq:vg}
        \deg_v(\oX^{c_j + b_{i_j} - z_{i_j}} p_j g_{ji}) + \alpha_i = v\cdot (c_j + b_{i_j}) - v\cdot z_{i_j} + \deg_v(g_{j,i}) + \alpha_i \\
        \leq - \deg_v(g_{j, i_j}) - \alpha_{i_j} + \deg_v(g_{j,i}) + \alpha_i \leq 0
    \end{multline}
    The first inequality comes from \eqref{eq:vbc} and $\alpha_{i_j} = v \cdot z_{i_j}$, while the second inequality comes from \eqref{eq:va}.
    Furthermore, the equality in \eqref{eq:vg} holds if and only if $\init^d_{v, \balpha}(\bg_j)_{i} \neq 0$ by the equality conditions in \eqref{eq:vbc} and \eqref{eq:va}.

    Hence, for each $i \in \{1, \ldots, K\}$ we have
    \begin{multline}\label{eq:zd}
        \init_{v, \balpha}(\bff)_i = \sum_{j: \init_{v, \balpha}^d(\bg_j)_{i} \neq 0} \oX^{c_j + b_{i_j} - z_{i_j}} \init_v(p_j g_{j, i}) = \sum_{j: \init_{v, \balpha}^d(\bg_j)_{i} \neq 0} \oX^{c_j + b_{i} - z_{i}} p_j \init_v(g_{j, i}) \\
        = \sum_{j: \init_{v, \balpha}(\bg_j)_{i} \neq 0} \oX^{c_j + b_{i} - z_{i}} p_j \init_{v, \balpha}(\bg_{j})_i = \sum_{j=1}^m \oX^{c_j + b_{i} - z_{i}} p_j \init_{v, \balpha}(\bg_{j})_i = \oX^{- z_{i}} \sum_{j=1}^m p_j \init^{d}_{v, \balpha}(\bg_{j})_i \\
        = \oX^{- z_{i}} f^d_i \in \A^+.
    \end{multline}
    In the first equality, the initial polynomials do not cancel each other because their sum is $\oX^{- z_{i}} f^d_i \in \A_d^+$.
    Therefore $\init_{v, \balpha}(\bff) \in \ApK$.

    We now prove Property~\eqref{eq:condirrshift}.
    Recall $I' \coloneqq \{i \in I \mid \alpha_{i} = \min_{i' \in I} \alpha_{i'}\}$.
    For $i, i' \in I'$, we have $v \cdot z_i = \alpha_i = \alpha_{i'} = v \cdot z_{i'}$.
    By the $\Q$-linear independence of the entries of $v$, we have
    \begin{equation}\label{eq:iiprime}
        z_i = z_{i'} \text{ for all } i, i' \in I'.
    \end{equation}
    Take any $w \in \Rns$.
    
    If $w \in \sum_{k = d+1}^n \R e_{k}$, then 
    \begin{align*}
        M_w(I', \init_{v, \balpha}(\bff)) & = \left\{i \in I' \;\middle|\;  \deg_w(\init_{v, \balpha}(\bff)_i) = \max_{i' \in I'} \deg_w(\init_{v, \balpha}(\bff)_{i'}) \right\} \\
        & = \left\{i \in I' \;\middle|\;  - w \cdot z_i = \max_{i' \in I'} \{- w \cdot z_{i'}\} \right\} \quad \quad \text{ (by \eqref{eq:zd})} \\
        & = I'. \quad \quad \text{ (by \eqref{eq:iiprime})} 
    \end{align*}
    So 
    \[
    \left(O_w \cup J'\right) \cap M_w(I', \init_{v, \balpha}(\bff)) = \left(O_w \cup J'\right) \cap I' \supseteq J' \cap I' \neq \emptyset.
    \]

    If $w \not\in \sum_{k = d+1}^n \R e_{k}$ and $d \geq 1$, write $w = w' + u$ where $w' \in\sum_{k = d+1}^n \R e_{k}$ and $u \in \sum_{k = 1}^d \R e_{k}$.
    Then
    \begin{align*}
        & M_w(I', \init_{v, \balpha}(\bff)) \\
        = & \left\{i \in I' \;\middle|\;  \deg_w(\init_{v, \balpha}(\bff)_i) = \max_{i' \in I'} \deg_w(\init_{v, \balpha}(\bff)_{i'}) \right\} \\
        = & \left\{i \in I' \;\middle|\;  - w' \cdot z_i + \deg_u(f^d_i) = \max_{i' \in I'} \{- w' \cdot z_{i'} + \deg_u(f^d_{i'})\} \right\} \quad \quad \text{ (by \eqref{eq:zd})} \\
        = & \left\{i \in I' \;\middle|\;  \deg_u(f^d_i) = \max_{i' \in I'} \{\deg_u(f^d_{i'})\} \right\} \quad \quad \text{ (by \eqref{eq:iiprime})}\\
        = &\; M_{u}(I', \bff^d).
    \end{align*}
    Since the last $n-d$ entries of $v$ are $\Q$-linearly independent, we have $a_i \not\in O_v$ if and only if $a_i \in \sum_{k = 1}^d \Z e_k$.
    Hence, $O'_u \cup J' = O_w \cup J'$ as in \eqref{eq:Ou}.
    Therefore,
    \[
    \left(O_w \cup J'\right) \cap M_w(I', \init_{v, \balpha}(\bff)) = \left(O'_u \cup J\right) \cap M_{u}(I', \bff^d) \neq \emptyset.
    \]
    
    If $w \not\in \sum_{k = d+1}^n \R e_{k}$ and $d = 0$.
    We have $f^{d}_i \in \R$ for $i = 1, \ldots, K$, so
    \begin{align*}
        M_w(I', \init_{v, \balpha}(\bff)) & = \left\{i \in I' \;\middle|\;  \deg_w(\init_{v, \balpha}(\bff)_i) = \max_{i' \in I'} \deg_w(\init_{v, \balpha}(\bff)_{i'}) \right\} \\
        & = \left\{i \in I' \;\middle|\;  - w' \cdot z_i = \max_{i' \in I'} \{- w' \cdot z_{i'}\} \right\} \quad \quad \text{ (by \eqref{eq:zd})} \\
        & = I' \quad \quad \text{ (by \eqref{eq:iiprime})}
    \end{align*}
    So $\left(O_{w} \cup J'\right) \cap M_{w}(I', \init_{v, \balpha}(\bff)) = \left(O_{w} \cup J'\right) \cap I' \supset J' \cap I' \neq \emptyset$.
    
    This proves Property~\eqref{eq:condirrshift}.
\end{proof}

\subsection{Dimension reduction: the general case}
This subsection continues the work of the previous one.
Our goal is to Condition~\hyperref[item:locshift]{(LocInfShift)} to a Condition~\hyperref[item:locd]{(LocInfD)} (see Proposition~\ref{prop:shifttod}).
In the previous subsection we considered the special case where the vector $v \in \Rns$ in Condition~\hyperref[item:locshift]{(LocInfShift)} is of the form $(0, \ldots, 0, v_{d+1}, \ldots, v_n)$.
In this subsection we consider the general case.
The key idea when dealing with the general case of $v \in \Rns$ is the following \emph{coordinate change}.

Given a matrix $A = (a_{ij})_{1 \leq i, j \leq n} \in \GL(n, \Z)$, define the new variables $X'_1, \ldots, X'_n$ where $X'_i \coloneqq X_1^{a_{i1}} X_2^{a_{i2}} \cdots X_n^{a_{in}}$.
Then 
\[
\R[X_1, \ldots, X_n] = \R[X'_1, \ldots, X'_n].
\]
In other words, we can define the ring automorphism
\[
\varphi_A \colon \A \rightarrow \A, \quad X_i \mapsto X_1^{a_{i1}} X_2^{a_{i1}} \cdots X_n^{a_{in}},
\]
such that $\varphi_A(\oX^{b}) = \oX^{b A}$.
The automorphism $\varphi_A$ extends entry-wise to $\A^K \rightarrow \A^K$.

For each $A \in \GL(n, \Z)$, denote by $A^{- \top}$ the inverse of its transpose.
Then $(v A^{-\top}) \cdot (b A) = v \cdot b$ for all $v \in \Rns, b \in \Z^n$.
Hence, for any $f \in \A$ we have $\init_{v A^{-\top}}(\varphi_A(f)) = \varphi_A(\init_{v}(f))$, and for any $\bff \in \A^K$ we have $M_v(I, \bff) = M_{v A^{-\top}}(I, \varphi_A(\bff))$.
Furthermore, if we replace the vectors $a_1, \ldots, a_K \in \Z^n$ by the vectors $a_1 A, \ldots, a_K A \in \Z^n$, then the set $O_v$ becomes $O_{v A^{-\top}}$.
It is easy to verify that if $\bg_1, \ldots, \bg_m$ is a super Gr\"{o}bner basis for $\mM$, then $\varphi_A(\bg_1), \ldots, \varphi_A(\bg_m)$ is still a super Gr\"{o}bner basis for $\varphi_A(\mM) \coloneqq \{\varphi_A(\bff) \mid \bff \in \mM\}$. 

Let $v \in \Rns$ and let $A \in \GL(n, \Z)$ be such that $v A^{-\top} = (0, \ldots, 0, v_{d+1}, \ldots, v_n)$ where $v_{d+1}, \ldots, v_n$ are $\Q$-linearly independent.
Then as in the previous section we define the module $\init_{v A^{-\top}, \balpha}^{d}(\varphi_A(\mM))$ to be the module generated by $\init_{v A^{-\top}, \balpha}^{d}(\varphi_A(\bg_1)), \ldots, \init_{v A^{-\top}, \balpha}^{d}(\varphi_A(\bg_m))$.

The above observation shows the following.
Fix $v \in \Rns$ in \hyperref[item:locshift]{(LocInfShift)} of Theorem~\ref{thm:locglob}.
Given any change of coordinates $A \in \GL(n, \Z)$, we can simultaneously (right-)multiply $A^{-\top}$ to $v$ and $A$ to all $a_1, \ldots, a_K$, while applying $\varphi_A$ to the super Gr\"{o}bner basis $\bg_1, \ldots, \bg_m$ of $\mM$.
Then the original properties \hyperref[item:locshift]{(LocInfShift)}(a)(b) are satisfied by $\bff$ if and only if they are satisfied by $\varphi_A(\bff)$ after the change of coordinates.
We will use this observation to reduce the general case for $v$ to the special case considered in the previous subsection.

\begin{fct}\label{fct:A}
    For every $v \in \Rns$, there exists $A \in \GL(n, \Z)$ such that $v A^{-\top} = (0, \ldots, 0, v_{d+1}, \ldots, v_n)$ with $v_{d+1}, \ldots, v_n$ being $\Q$-linearly independent.
\end{fct}
\begin{proof}
    It suffices to show the following.
    If $v \in \Rns$ is of the form $(0, \ldots, 0, v_{d}, \ldots, v_n)$, $1 \leq d \leq n$ where $v_{d}, \ldots, v_n$ being $\Q$-linearly \emph{dependent}, then we can find a matrix $A_d \in \GL(n, \Z)$ such that $v A_d = (0, \ldots, 0, v_{d+1}, \ldots, v_n)$.
    Indeed, if this is true, then we can find a series of matrices $A_{r}, \ldots, A_{r+s}$ such that $v A_{r} A_{r+1} \cdots A_{r+s}$ is of the form $(0, \ldots, 0, v_{r+s+1}, \ldots, v_n)$ with $v_{r+s+1}, \ldots, v_n$ being $\Q$-linearly \emph{independent}. We would then let $A \coloneqq \left(A_{r} A_{r+1} \cdots A_{r+s}\right)^{-\top}$.

    Suppose now that $v = (0, \ldots, 0, v_{d}, \ldots, v_n)$, $1 \leq d \leq n$ where $v_{d}, \ldots, v_n$ are $\Q$-linearly dependent.
    Let $z_d, \ldots, z_n \in \Q$, not all zero, be such that $z_d v_d + \cdots + z_n v_n = 0$.
    Multiplying them by a common denominator we can suppose $z_d, \ldots, z_n \in \Z$.
    Using Gaussian pivoting, we can find a matrix $\widetilde{A_d} \in \GL(n-d+1, \Z)$ such that $(z_d, \ldots, z_n) \widetilde{A_d} = (z, 0, \ldots, 0)$ for some $z \in \Z^*$.
    Then we have 
    \begin{multline*}
    0 = (z_d, \ldots, z_n) \cdot (v_{d}, \ldots, v_n) = (z_d, \ldots, z_n) \widetilde{A_d} \cdot (v_{d}, \ldots, v_n) \widetilde{A_d}^{-\top} \\
    = (z, 0, \ldots, 0) \cdot (v_{d}, \ldots, v_n) \widetilde{A_d}^{-\top}.
    \end{multline*}
    Therefore $(v_{d}, \ldots, v_n) \widetilde{A_d}^{-\top}$ is of the form $(0, v'_{d+1}, \ldots, v'_n)$.
    We then let $A_d \coloneqq diag(I_{d-1}, \widetilde{A_d}^{-\top})$.
    That is, $A_d$ is the block diagonal matrix consisting of the block $I_{d-1}$ of $(d-1)$-dimensional identity matrix and the block $\widetilde{A_d}^{-\top}$ of $(n-d+1)$-dimensional matrix.
    Then $v A_d = (0, \ldots, 0, v'_{d+1}, \ldots, v'_n)$.
\end{proof}

\begin{prop}\label{prop:shifttod}
    Condition~\emph{\hyperref[item:locshift]{(LocInfShift)}} of Proposition~\ref{prop:inftoshift} is equivalent to the following:
    \begin{enumerate}
        \item[2.] \label{item:locd} \emph{\textbf{(LocInfD):}}
        For every $v \in \Rns, A \in \GL(n, \Z),$ such that $v A^{-\top} = (0, \ldots, 0, v_{d+1}, \ldots, v_n)$, $0 \leq d \leq n-1$ with $v_{d+1}, \ldots, v_n$ being $\Q$-linearly independent,
        there exists $\balpha \in \left(\sum_{k = d+1}^n \Z v_k\right)^K$ and $\bff^d \in \init_{v A^{-\top}, \balpha}^{d}(\varphi_A(\mM))$ satisfying the following properties:
        \begin{enumerate}
            \item[(a)] $\bff^d \in \left(\A_{d}^+\right)^K$.
            \item[(b1)] Denote $I' \coloneqq \{i \in I \mid \alpha_{i} = \min_{i' \in I} \alpha_{i'}\}, J' \coloneqq \left(O_{v} \cup J\right)$, we have
            \begin{equation}\label{eq:condirr}
                J' \cap I' \neq \emptyset.
            \end{equation}
            \item[(b2)] Denote by $\pi_d \coloneqq \Z^n \rightarrow \Z^d$ the projection onto the first $d$ coordinates.
            For $u \in (\R^{d})^*$, define $O'_u \coloneqq \{i \in \{1, \ldots, K\} \mid \pi_d(A a_i) \not\perp u\}$, we have
            \begin{equation}\label{eq:condirrD}
            \left(O'_{u} \cup J'\right) \cap M_{u}(I', \bff^d) \neq \emptyset \quad \text{ for every $u \in (\R^{d})^*$}.
            \end{equation}
        \end{enumerate}
    \end{enumerate}
    As in Lemma~\ref{lem:dimred}, Property~\eqref{eq:condirrD} is considered trivially true when $d = 0$.
\end{prop}
\begin{proof}
    Fix a $v = (v_1, \ldots, v_n) \in \Rns$.
    Take any $A \in \GL(n, \Z)$ with $v A^{-\top} = (0, \ldots, 0, v'_{d+1}, \ldots, v'_n)$ such that $v'_{d+1}, \ldots, v'_n$ are $\Q$-linearly independent.
    Note that $\sum_{k = 1}^n \Z v_k = \sum_{k = d+1}^n \Z v'_k$ because $A \in \GL(n, \Z)$.
    Therefore, we can apply Lemma~\ref{lem:dimred} to the super Gr\"{o}bner basis $\varphi_A(\bg_1), \ldots, \varphi_A(\bg_m)$, the vector $v A^{-\top} = (0, \ldots, 0, v'_{d+1}, \ldots, v'_n)$ and the vectors $a_1 A, \ldots, a_K A \in \Z^n$.
    This shows that there exists $\balpha \in \left(\sum_{k = d+1}^n \Z v'_k\right)^K$ and $\bff^d \in \init_{v A^{-\top}, \balpha}^{d}(\varphi_A(\mM))$ satisfying \hyperref[item:locd]{(LocInfD)}(a)(b1)(b2) if and only if there exists $\balpha \in \left(\sum_{k = 1}^n \Z v_k\right)^K$ and $\bff \in \mM$ satisfying \hyperref[item:locshift]{(LocInfShift)}(a)(b).
\end{proof}

\subsection{Local condition at infinity: computing cells \hyperref[item:loccell]{(LocInfCell)}}\label{subsec:cell}
In this subsection we further reduce the Condition~\hyperref[item:locd]{(LocInfD)} to a Condition~\hyperref[item:loccell]{(LocInfCell)} which consists of verifying a \emph{finite} number of $v \in \Rns$ for each coordinate-change matrix $A \in \GL(n, \Z)$.

Let $v \in \Rns, \balpha \in \R^K$.
Denote by $e_1, \ldots, e_K$ the canonical basis of the $\A$-module $\A^K$.
We introduce the new variables $T_1, \ldots, T_K$ and define an $\A$-module homomorphism
\[
\phi: \A^K \rightarrow \R[X_1^{\pm}, \ldots, X_n^{\pm}, T_1^{\pm}, \ldots, T_K^{\pm}], \quad \oX^{u} e_i \mapsto \oX^{u} T_i.
\]
We have $\phi(\init_{v, \balpha}(\bff)) = \init_{(v, \balpha)}(\phi(\bff))$ for every $\bff \in \A^K$.

As in the previous subsections let $\bg_1, \ldots, \bg_m$ be a super Gr\"{o}bner basis of $\mM$.
Since $\phi(\bg_i)$ is a polynomial in $\R[X_1^{\pm}, \ldots, X_n^{\pm}, T_1^{\pm}, \ldots, T_K^{\pm}]$, there exists a partition of $\Rns \times \R^K$ such that for any two directions in the same partition element the initial parts of $\phi(\bg_i)$ are the same.
Let $\mL_{\mM}$ be the common refinement of the partitions associated to the polynomials $\phi(\bg_1), \ldots, \phi(\bg_m)$.

From now on we use the term ``\emph{cell}'' to call elements of a given partition.
Fix $I \subseteq \{1, \ldots, K\}$.
There exists a partition $\mL_I$ of $\R^K$ such that for any two vectors $(\alpha_1, \ldots, \alpha_K)$, $(\alpha'_1, \ldots, \alpha'_K)$ in the same cell, we have $\alpha_i > \alpha_j \iff \alpha'_i > \alpha'_j$ and $\alpha_i < \alpha_j \iff \alpha'_i < \alpha'_j$ for all $i, j \in I$.
Define the partition $\mL'_I \coloneqq \Rns \times \mL_I$ of $\Rns \times \R^K$ where each cell is of the form $\Rns \times P, P \in \mL_I$.

Finally, there exists a partition $\mL_O$ of $\Rns$ such that any two vectors $v, v'$ in the same cell satisfy $v \perp a_i \iff v' \perp a_i$ for all $i \in \{1, \ldots, K\}$.
By subdividing $\mL_O$ we can suppose that each cell is a convex polyhedron.
Similar to the definition of $\mL'_I$, we define the partition $\mL'_O \coloneqq \mL_O \times \R^K$ of $\Rns \times \R^K$.

For any two partition $\mA, \mB$ of the same set $S$, define $\mA \vee \mB$ to be the partition of $S$ whose elements are of the form $A \cap B, A \in \mA, B \in \mB$.
Consider the partition $\mL$ of $\Rns \times \R^K$ defined by
\[
\mL \coloneqq \mL_{\mM} \vee \mL'_I \vee \mL'_O.
\]
We point out that from the definition of the partitions $\mL_{\mM}, \mL'_I, \mL'_O$, the cells of $\mL$ are invariant under scaling by a positive real, meaning $x \in Q \iff r \cdot x \in Q$ for all cells $Q \in \mL$ and $r \in \Rpp$.

Let $\pi \colon \Rns \times \R^K \rightarrow \Rns, (v, \balpha) \mapsto v$ be the canonical projection.
For each $Q \in \mL$, define the two-element partition $\{\pi(Q), \Rns \setminus \pi(Q)\}$, and define
\[
\mP \coloneqq \bigvee_{Q \in \mL} \{\pi(Q), \Rns \setminus \pi(Q)\}.
\]
By this definition, take any $P \in \mP$ and $Q \in \mL$ with $\pi^{-1}(P) \cap Q \neq \emptyset$; then for $v, v' \in P$, there exists $\balpha \in \R^K$ with $(v, \balpha) \in Q$ if and only if there exists $\balpha' \in \R^K$ with $(v', \balpha') \in Q$.

It is important to note that the partitions $\mL_{\mM}, \mL_{I}, \mL_{O}$ are all defined using equalities and inequalities with \emph{rational} coefficients.
Also, each inequality is strict, so every cell $Q \in \mL$ and $P \in \mP$ is relatively open (a polyhedron is called relative open if it is open in the smallest linear space containing it).
In other words, each cell is defined by a combination of equalities and \emph{strict} inequalities.
We also point out that, like the cells of $\mL$, the cell of $\mP$ are invariant under scaling by a positive real, meaning $x \in P \iff r \cdot x \in P$ for all cells $P \in \mP$ and $r \in \Rpp$.


Let $Q \in \mL$. For $(v, \balpha), (v', \balpha') \in Q$, we have
\[
    \init_{v, \balpha}(\bg_j) = \init_{v', \balpha'}(\bg_j)
\]
for all $j = 1, \ldots, d$.
Thus, if $v = (0, \ldots, 0, v_{d+1}, \ldots, v_n)$ is such that $v_{d+1}, \ldots, v_n$ are $\Q$-linearly independent, then $\init^{d}_{v, \balpha}(\mM)$ depends only on the cell $Q \in \mL$ containing $(v, \balpha)$.
Hence, we can denote
\[
\init^{d}_{Q}(\bg_j) \coloneqq \init^{d}_{v, \balpha}(\bg_j), \quad j = 1, \ldots, m, \quad \init^{d}_{Q}(\mM) \coloneqq \init^{d}_{v, \balpha}(\mM), \quad \text{ where } (v, \balpha) \in Q.
\]

For any coordinate change $A \in \GL(n, \Z)$, we similarly define the partitions $\mL A^{-\top}$ and $\mP A^{-\top}$ based on the super Gr\"{o}bner basis $\varphi_A(\bg_1), \ldots, \varphi_A(\bg_m)$ and the vectors $a_1 A, \ldots, a_K A$.
In particular, each cell of $\mL A^{-\top}$ is of the form $Q \cdot diag(A^{-\top}, I_K), Q \in \mL$, and each cell of $\mP A^{-\top}$ is of the form $P \cdot A^{-\top}, P \in \mP$.
If $v A^{- \top} = (0, \ldots, 0, v_{d+1}, \ldots, v_n)$ is such that $v_{d+1}, \ldots, v_n$ are $\Q$-linearly independent, then $\init^{d}_{v A^{- \top}, \balpha}(\varphi_A(\mM))$ depends only on the cell $Q \in \mL A^{- \top}$ containing $(v A^{- \top}, \balpha)$.
Similarly, for $j = 1, \ldots, m$, we can denote
\[
\init^{d}_{Q}(\varphi_A(\bg_j)) \coloneqq \init^{d}_{v A^{- \top}, \balpha}(\varphi_A(\bg_j)), \; \init^{d}_{Q}(\varphi_A(\mM)) \coloneqq \init^{d}_{v A^{- \top}, \balpha}(\varphi_A(\mM)), \; \text{ where } (v A^{- \top}, \balpha) \in Q.
\]

The inputs in Theorem~\ref{thm:dec} are generators for modules $\mM$ over $\A = \R[X_1^{\pm}, \ldots, X_n^{\pm}]$, vectors $a_1, \ldots, a_K$ in $\Z^n$ and two sets $I, J$.
Our strategy is to use induction on $n$ to prove Theorem~\ref{thm:dec}.
The base case $n = 0$ reduces to linear programming.
Indeed, when $n = 0$, $\A = \R, \A^+ = \Rpp$, the Property~\eqref{eq:deccond} is trivially true; and the problem becomes the following: given an $\R$-submodule $\mM$ of $\R^K$, decide whether $\mM \cap \Rpp^K$ contains an element.
Since the given generators of $\mM$ all have integer coefficients, this is decidable using linear programming.

The following observation shows that a decision procedure for Theorem~\ref{thm:dec} with smaller $n$ can help us decide which cells $Q \in \mL$ contain $\bff^d$ satisfying the Properties \hyperref[item:locd]{(LocInfD)}(a)(b1)(b2).

\begin{lem}\label{lem:compop}
    Fix a change of coordinates $A \in \GL(n, \Z)$ and a number $0 \leq d \leq n-1$.  
    Suppose Theorem~\ref{thm:dec} is true for all $n_0$, $0 \leq n_0 \leq n - 1$.
    Then for each cell $Q \in \mL A^{- \top}$, we can decide whether $\init_{Q}^{d}(\varphi_A(\mM))$ contains an element $\bff^d$ satisfying the Properties \emph{\hyperref[item:locd]{(LocInfD)}}(a)(b1) and (b2).
\end{lem}
\begin{proof}
    Suppose Theorem~\ref{thm:dec} is true for all $0 \leq n_0 \leq n - 1$.
    Fix a cell $Q \in \mL A^{- \top}$.
    
    By the definition of the partition $\mL$, the sets $I' \coloneqq \{i \in I \mid \alpha_{i} = \min_{i' \in I} \alpha_{i'}\}, J' \coloneqq O_{v} \cup J$ only depend on the cell $Q$ containing $(v, \balpha)$.
    Hence we can compute $I', J'$ and verify whether the Property \hyperref[item:locd]{(LocInfD)}(b1), $J' \cap I' \neq \emptyset$, is satisfied.

    In \hyperref[item:locd]{(LocInfD)}, the $\A_d$-submodule $\init_{v A^{- \top}, \balpha}^{d}(\varphi_A(\mM)) = \init_{Q}^{d}(\varphi_A(\mM))$ of $\A_d^K$ is generated by $\init^{d}_{Q}(\varphi_A(\bg_j)), j = 1, \ldots, m$.
    Recall that $\pi_d \coloneqq \Z^n \rightarrow \Z^d$ denotes the projection onto the first $d$ coordinates.
    Therefore using Theorem~\ref{thm:dec}, replacing $n$ by $d < n$, replacing the $\A$-module $\mM$ by the $\A_d$-module $\init_{Q}^{d}(\varphi_A(\mM))$, and replacing the vectors $a_1, \ldots, a_K$ by the vectors $\pi_d(a_1 A), \ldots, \pi_d(a_K A)$, we can decide whether $\init_{Q}^{d}(\varphi_A(\mM))$ contains an element $\bff^d$ satisfying the Properties~\hyperref[item:locd]{(LocInfD)}(a) and (b2).
\end{proof}

Denote by $Op(A, d)$ the union of all cells $Q \in \mL A^{- \top}$ such that $\init_{Q}^{d}(\varphi_A(\mM))$ contains an element $\bff^d$ satisfying the Properties~\hyperref[item:locd]{(LocInfD)}(a)(b1)(b2).
By Lemma~\ref{lem:compop}, the set $Op(A, d)$ is effectively computable as a finite union of polyhedra defined over rational coefficients (supposing Theorem~\ref{thm:dec} is true for all $0 \leq n_0 \leq n - 1$).
See Figure~\ref{fig:L} for an illustration of $Op(A, d)$.

\begin{prop}\label{prop:dtocell}
    Condition~\emph{\hyperref[item:locd]{(LocInfD)}} of Proposition~\ref{prop:shifttod} is equivalent to the following:
    \begin{enumerate}
        \item[2.] \label{item:loccell} \emph{\textbf{(LocInfCell):}}
        For every $A \in \GL(n, \Z)$ and every number $0 \leq d \leq n-1$, the following is true:
        \begin{enumerate}
            \item[(a)] For every $v = (0, \ldots, 0, v_{d+1}, \ldots, v_n) \in \{0\}^d \times (\R^{n-d})^*$ with $v_{d+1}, \ldots, v_n$ being $\Q$-linearly independent, there exists $\balpha \in \left(\sum_{k = d+1}^n \Z v_k\right)^K$ with $(v, \balpha) \in Op(A, d)$.
        \end{enumerate}
    \end{enumerate}
\end{prop}
\begin{proof}
This follows directly from the definition of $Op(A, d)$. 
\end{proof}

\begin{figure}[h!]
        \centering
        \includegraphics[width=0.8\textwidth,height=1.0\textheight,keepaspectratio, trim={0cm 0cm 0cm 0cm},clip]{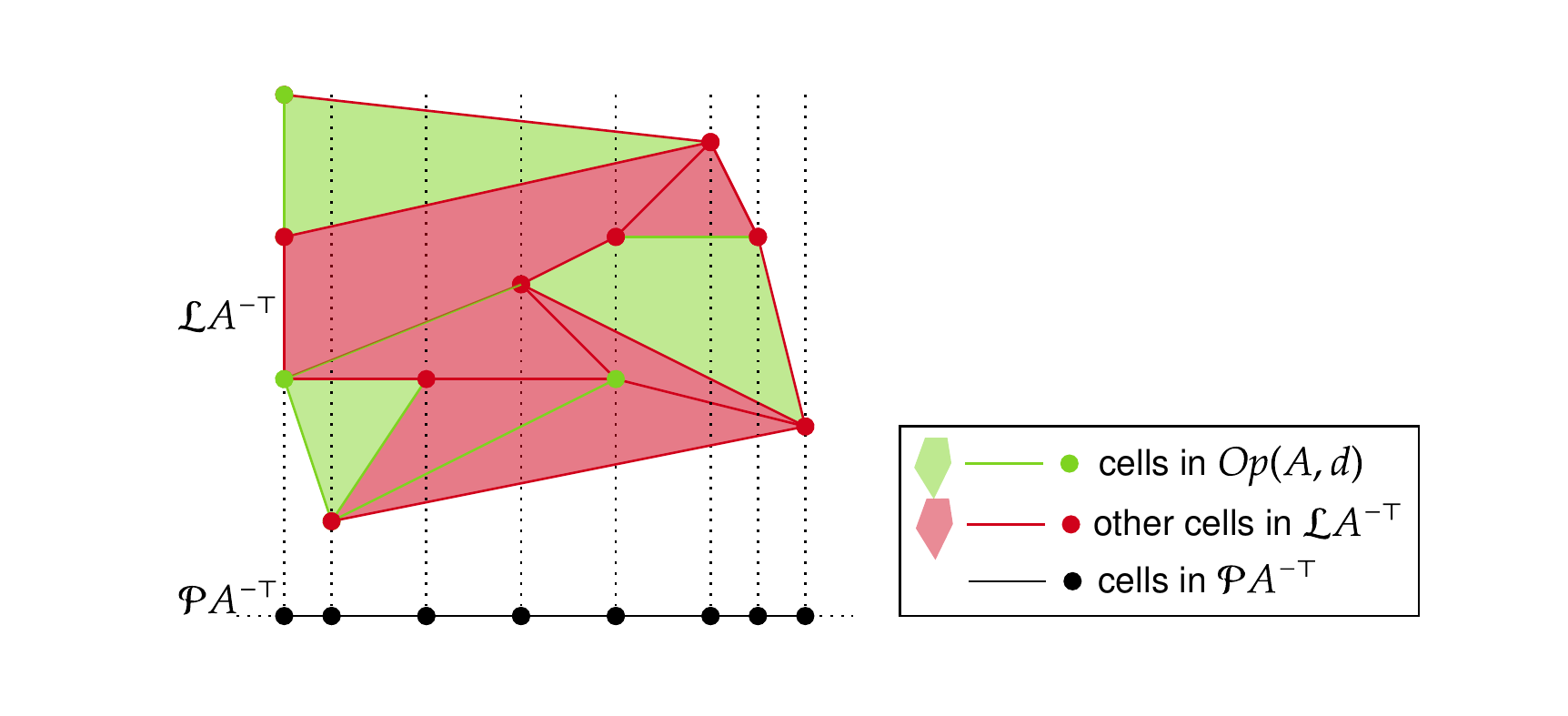}
        \caption{Illustration of $Op(A, d)$.}
        \label{fig:L}
\end{figure}   
\begin{figure}[h!]
        \centering
        \includegraphics[width=0.8\textwidth,height=1.0\textheight,keepaspectratio, trim={0cm 0cm 0cm 0cm},clip]{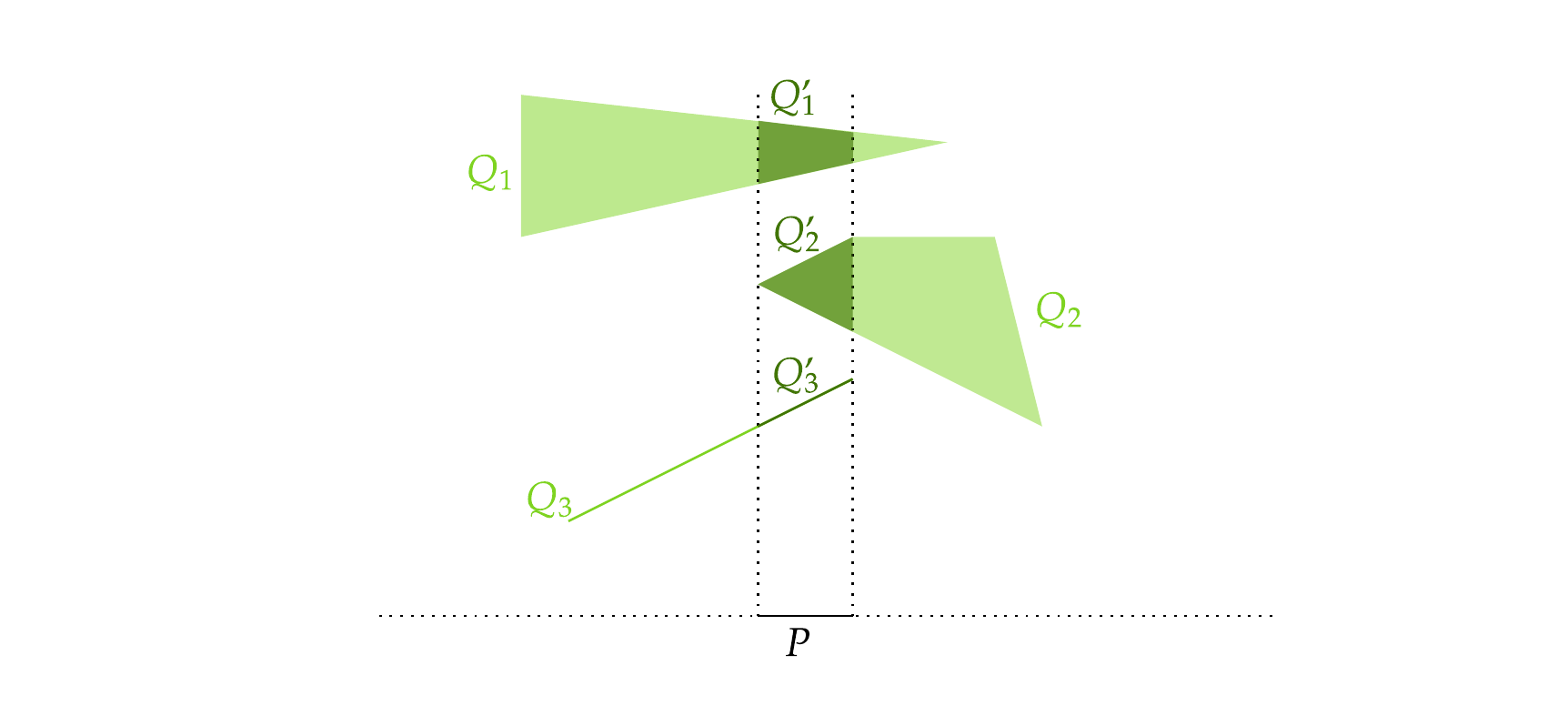}
        \caption{Illustration of Lemma~\ref{lem:checkcell}}
        \label{fig:cells}
\end{figure}

\begin{lem}\label{lem:checkcell}
    Given $A \in \GL(n, \Z), d \in \N$ and given $Op(A, d)$ as a finite union of polyhedra defined over rational coefficients, it is decidable whether the statement \emph{\hyperref[item:loccell]{(LocInfCell)}(a)} is true.
\end{lem}
\begin{proof}
    Replace each cell $Q$ in $\mL A^{- \top}$ with its intersection with $\{0\}^d \times (\R^{n-d})^* \times \R^K$; and replace each cell $P \in \mP A^{- \top}$ with its intersection with $\{0\}^d \times (\R^{n-d})^*$.
    We can suppose $\mP A^{- \top}$ is a partition of $(\R^{n-d})^*$, $\mL A^{- \top}$ is a partition of $(\R^{n-d})^* \times \R^K$, $Op(A, d) \subseteq (\R^{n-d})^* \times \R^K$ is a union of cells in $\mL A^{- \top}$, and that $v \in \left(\R^{n-d}\right)^*$ for all $v$ we consider.
    We separate two cases.

    \begin{enumerate}[noitemsep, label = (\arabic*)]
        \item
    \textbf{When $d = n-1$.}
    In this case, since the partitions are invariant under scaling, we can suppose $v_n \in \{1, -1\}$.
    Then for each case $v_n = 1$ and $v_n = -1$, decide whether there exists $\balpha \in \left(\Z v_n\right)^K = \Z^K$ with $(v, \balpha) \in Op(A, d)$.
    Since $Op(A, d)$ is a finite union of polyhedra defined using rational coefficients, this is decidable using integer programming.

        \item
    \textbf{When $d \leq n-2$.}
    See Figure~\ref{fig:cells} for an illustration in this case.
    Whenever $v = (v_{d+1}, \ldots, v_n)$ with $v_{d+1}, \ldots, v_n$ being $\Q$-linearly independent, $v$ must fall in a cell $P \in \mP A^{- \top}$ of dimension $n-d$.
    For each cell $P \in \mP A^{- \top}$ of dimension $n-d$, consider all cells $Q \subseteq Op(A, d)$ such that $\pi(Q) \cap P \neq \emptyset$.
    If there is no such cell $Q$ then statement \hyperref[item:loccell]{(LocInfCell)}(a) is false.
    In fact, since $P \in \mP A^{- \top}$ is of dimension $n-d$, it contains an element $v = (v_{d+1}, \ldots, v_n)$ with $v_{d+1}, \ldots, v_n$ being $\Q$-linearly independent.
    Then for this $v$, there does not exist any $\balpha \in \R^K$ such that $(v, \balpha) \in Op(A, d)$, so statement \hyperref[item:loccell]{(LocInfCell)}(a) is false.
    
    Suppose now that for every cell $P \in \mP$ of dimension $n-d$ there exist cells $Q \subseteq Op(A, d)$ such that $\pi(Q) \cap P \neq \emptyset$.
    Fix a cell $P \in \mP$, let $Q_1, \ldots, Q_{\ell}$ denote all cells in $Op(A, d)$ such that $\pi(Q) \cap P \neq \emptyset$.
    Define $Q'_t \coloneqq Q_t \cap \pi^{-1} (P)$ for $t = 1, \ldots, \ell$.
    Each $Q'_t$ is a relatively open polyhedron.

    Take an arbitrary $Q'_t$, it is define by the following equations and inequalities:
    \begin{align}
        & (v_{d+1}, \ldots, v_n) \in P \label{eq:P1} \\
        & \beta_{j, 1} \alpha_1 + \cdots + \beta_{j, K} \alpha_K = \gamma_{j, d+1} v_{d+1} + \cdots + \gamma_{j, n} v_n, \quad j = 1, \ldots, m_0, \label{eq:eq1}\\
        & \delta_{j, 1} \alpha_1 + \cdots + \delta_{j, K} \alpha_K < \epsilon_{j, d+1} v_{d+1} + \cdots + \epsilon_{j, n} v_n, \quad j = 1, \ldots, m_1.  \label{eq:ineq1}
    \end{align}
    Where $\beta_{j, i}, \gamma_{j, i}, \delta_{j, i}, \epsilon_{j, i}$ are all rational numbers.
    Note that by the definition of $P$, for $v, v' \in P$, there exists $\balpha \in \R^K$ with $(v, \balpha) \in Q$ if and only if there exists $\balpha' \in \R^K$ with $(v', \balpha') \in Q$.
    Therefore $\pi(Q'_t) = P$ and we can suppose that the left hand sides of \eqref{eq:eq1} and \eqref{eq:ineq1} do not vanish (so that no extra constraint on $(v_{d+1}, \ldots, v_n)$ other than \eqref{eq:P1} is imposed).
    
    Using Gaussian pivoting and possibly exchanging the orders of $\alpha_i, i = 1, \ldots, K$, we can rewrite the above equations and inequalities into the form
    \begin{align}
        & (v_{d+1}, \ldots, v_n) \in P \\
        & \alpha_{i} = \beta_{i, D+1} \alpha_{D+1} + \cdots + \beta_{i, K} \alpha_K + \gamma_{i, d+1} v_{d+1} + \cdots + \gamma_{i, n} v_n, \quad i = 1, \ldots, D, \label{eq:eqq}\\
        & \delta_{j, D+1} \alpha_{D+1} + \cdots + \delta_{j, K} \alpha_K < \epsilon_{j, d+1} v_{d+1} + \cdots + \epsilon_{j, n} v_n, \quad j = 1, \ldots, m_1. 
    \end{align}
    In particular, the number $D \in \N$ is such that $Q'_t$ is a polyhedron of dimension $n-d + K - D$.

    Let $M \in \N$ be a common denominator of all $\beta_{i, j}, \gamma_{i,k}, i = 1, \ldots, D, j = D+1, \ldots, K, k = d+1, \ldots, n$.
    We multiply both sides of the Equations~\eqref{eq:eqq} by $M$, and suppose $Q'_t$ is defined by
    \begin{align}
        & (v_{d+1}, \ldots, v_n) \in P \label{eq:inP}\\
        & M \alpha_{i} = \beta_{i, D+1} \alpha_{D+1} + \cdots + \beta_{i, K} \alpha_K + \gamma_{i, d+1} v_{d+1} + \cdots + \gamma_{i, n} v_n, \quad i = 1, \ldots, D, \label{eq:eq}\\
        & \delta_{j, D+1} \alpha_{D+1} + \cdots + \delta_{j, K} \alpha_K < \epsilon_{j, d+1} v_{d+1} + \cdots + \epsilon_{j, n} v_n, \quad j = 1, \ldots, m_1 \label{eq:ineq},
    \end{align}
    where $\beta_{i, j}, \gamma_{i,k}, i = 1, \ldots, D, j = D+1, \ldots, K, k = d+1, \ldots, n$, are \emph{integers}.

    Fix any $v = (v_{d+1}, \ldots, v_n) \in P$ with $v_{d+1}, \ldots, v_n$ being $\Q$-linearly independent.
    We claim the following.
    There exists $\balpha \in \left(\sum_{k = d+1}^n \Z v_k\right)^K$ such that $(v, \balpha) \in Q'_t$, if and only if the following system of $(n-d)D$ equations has integer solutions $z_{i, k}, i = 1, \ldots, D, D+1, \ldots, K, k = d+1, \ldots, n$.

    \begin{align}\label{eq:intsys}
        & M z_{i, d+1} = \beta_{i, D+1} z_{D+1, d+1} + \cdots + \beta_{i, K} z_{K, d+1} + \gamma_{i, d+1}, \quad i = 1, \ldots, D, \nonumber\\
        & M z_{i, d+2} = \beta_{i, D+1} z_{D+1, d+2} + \cdots + \beta_{i, K} z_{K, d+2} + \gamma_{i, d+2}, \quad i = 1, \ldots, D, \nonumber\\
        & \quad \vdots \nonumber\\
        & M z_{i, n} = \beta_{i, D+1} z_{D+1, n} + \cdots + \beta_{i, K} z_{K, n} + \gamma_{i, n}, \quad i = 1, \ldots, D.
    \end{align}

    We now prove this claim.
    For the ``only if'' implication, suppose there exists $\balpha \in \left(\sum_{k = d+1}^n \Z v_k\right)^K$ such that $(v, \balpha) \in Q'_t$.
    For each $i \in \{1, \ldots, K\}$, we write $\alpha_i = \sum_{k = d+1}^n z_{i, k} v_k$, then the Equations~\eqref{eq:eq} become
    \begin{multline}\label{eq:vsys}
        M \sum_{k = d+1}^n z_{i, k} v_k = \beta_{i, D+1} \sum_{k = d+1}^n z_{D+1, k} v_k + \cdots + \beta_{i, K} \sum_{k = d+1}^n z_{K, k} v_k + \gamma_{i, d+1} v_{d+1} + \cdots + \gamma_{i, n} v_n, \\ i = 1, \ldots, D.
    \end{multline}
    Since $v_{d+1}, \ldots, v_n$ are $\Q$-linearly independent, Equations~\eqref{eq:vsys} hold if and only if for all $k = d+1, \ldots, n$, the coefficients of $v_{k}$ on both sides are equal.
    That is,
    \[
    M z_{i, k} = \beta_{i, D+1} z_{D+1, k} + \cdots + \beta_{i, K} z_{i, k} + \gamma_{i, k}, \quad i = 1, \ldots, D, \quad k = d+1, \ldots, n.
    \]
    This is exactly the system~\eqref{eq:intsys}.

    For the ``if'' implication, suppose the system of equations~\eqref{eq:intsys} has integer solutions $z_{i, k}, i = 1, \ldots, D, D+1, \ldots, K, k = d+1, \ldots, n$.

    For each tuple of integers $c_{i, k}, i = D+1, \ldots, K, k = d+1, \ldots, n$, we can construct a new solution of the system~\eqref{eq:intsys} by letting $z'_{i, k} \coloneqq z_{i, k} + M c_{i, k}, i = D+1, \ldots, K, k = d+1, \ldots, n$, and
    \begin{align}
        & z'_{i, d+1} \coloneqq z_{i, d+1} + \beta_{i, D+1} c_{D+1, d+1} + \cdots + \beta_{i, K} c_{K, d+1}, \quad i = 1, \ldots, D, \nonumber\\
        & z'_{i, d+2} = z_{i, d+2} + \beta_{i, D+1} c_{D+1, d+2} + \cdots + \beta_{i, K} c_{K, d+2}, \quad i = 1, \ldots, D, \nonumber\\
        & \quad \vdots \nonumber\\
        & z'_{i, n} = z_{i, n} + \beta_{i, D+1} c_{D+1, n} + \cdots + \beta_{i, K} c_{K, n}, \quad i = 1, \ldots, D.
    \end{align}
    Since $z'_{i, k}, i = 1, \ldots, D, D+1, \ldots, K, k = d+1, \ldots, n,$ is a solution for \eqref{eq:intsys}, it is easy to verify that $\alpha_i \coloneqq \sum_{k = d+1}^n z'_{i, k} v_k, i = 1, \ldots, K$ constitute a solution for \eqref{eq:eq}.
    We now show that for every tuple $(v_{d+1}, \ldots, v_n) \in P$, we can actually find integers $c_{i, k}, i = D+1, \ldots, K, k = d+1, \ldots, n$ such that $\alpha_i = \sum_{k = d+1}^n (z_{i, k} + M c_{i, k}) v_k, i = D+1, \ldots, n,$ satisfy also \eqref{eq:ineq}.

    Since $Q'_t$ is relatively open and non-empty, the set $\pi^{-1}((v_{d+1}, \ldots, v_n)) \cap Q'_t$ is also non-empty.
    Therefore, the solution set $\mA$ for $(\alpha_{D+1}, \ldots, \alpha_K) \in \R^{K-D}$ of the inequalities~\eqref{eq:ineq} is non-empty.
    This solution set $\mA$ is an open subset of $\R^{K-D}$ since it is define by strict inequalities.
    Since $v_{d+1}, \ldots, v_n$ are $\Q$-linearly independent, the set 
    \[
    \left\{\alpha_i \coloneqq \sum_{k = d+1}^n (z_{i, k} + M c_{i, k}) v_k \;\middle|\; c_{i, d+1}, \ldots, c_{i, n} \in \Z \right\}
    \]
    is dense in $\R$ for every $i \in \{D+1, \ldots, K\}$.
    Thus we can find $c_{D+1, d+1}, \ldots, c_{D+1, n} \in \Z$ such that $\alpha_{D+1} \coloneqq \sum_{k = d+1}^n (z_{D+1, k} + M c_{D+1, k}) v_k$ satisfies $(\alpha_{D+1}, x_{D+2}, \ldots ,x_{n}) \in \mA$ for some $x_{D+2}, \ldots ,x_{n} \in \R$.
    Similarly, by the openness of $\mA$, we can then find $c_{D+2, d+1}, \ldots, c_{D+2, n} \in \Z$ such that $\alpha_{D+2} \coloneqq \sum_{k = d+1}^n (z_{D+2, k} + M c_{D+2, k}) v_i$ satisfies $(\alpha_{D+1}, \alpha_{D+2}, x_{D+3}, \ldots ,x_{n}) \in \mA$ for some $x_{D+3}, \ldots ,x_{n} \in \R$.
    Continue this way and we will find integers $c_{i, k}, i = D+1, \ldots, K, k = d+1, \ldots, n$ such that $(\alpha_{D+1}, \ldots ,\alpha_{n}) \in \mA$.
    This tuple $(\alpha_{D+1}, \ldots ,\alpha_{n})$ satisfies \eqref{eq:ineq}.
    Since $\alpha_i = \sum_{k = d+1}^n z'_{i, k} v_k, i = 1, \ldots, K,$ is a solution for \eqref{eq:eq} regardless of the choice of $c_{i, k}$, both \eqref{eq:eq} and \eqref{eq:ineq} are now satisfied.
    We have proved the claim.
    
    Note that whether the system~\eqref{eq:intsys} has integer solutions depend only on the coefficients $\beta_{i, k}, \gamma_{i,k}$, $i = 1, \ldots, D, j = D+1, \ldots, K, k = d+1, \ldots, n$.
    These coefficients are determined by the polyhedron $Q'_t$, but not on the choice of $v$.
    For each $Q'_t, t = 1, \ldots, \ell$, we can decide whether its system~\eqref{eq:intsys} has integer solutions.
    If for some $t \in \{1, \ldots, \ell\},$ its system~\eqref{eq:intsys} has integer solutions, then the above claim shows that for all $v = (v_{d+1}, \ldots, v_n) \in P$ with $v_{d+1}, \ldots, v_n$ being $\Q$-linearly independent, there exists $\balpha \in \left(\sum_{k = d+1}^n \Z v_k\right)^K$ such that $(v, \balpha) \in Q'_t$.
    Otherwise, if for all $t \in \{1, \ldots, \ell\},$ its system~\eqref{eq:intsys} has no integer solutions, then for any $v = (v_{d+1}, \ldots, v_n) \in P$ with $v_{d+1}, \ldots, v_n$ being $\Q$-linearly independent, there does not exist $\balpha \in \left(\sum_{k = d+1}^n \Z v_k\right)^K$ such that $(v, \balpha) \in Q'_t$.

    To summarize, in order to decide whether the statement \hyperref[item:loccell]{(LocInfCell)}(a) is true, it suffices to enumerate all cells $P \in \mP A^{- \top}$ of dimension $n-d$.
    For a cell $P$, if there is no cell $Q \subseteq Op(A, d)$ such that $\pi(Q) \cap P \neq \emptyset$. then statement \hyperref[item:loccell]{(LocInfCell)}(a) is false; otherwise for each $Q'_1, \ldots, Q'_{\ell}$ check whether system~\eqref{eq:intsys} has integer solutions, in case an integer solution exists for some $Q'_t$ we call the cell $P$ ``operational''.
    If every cell $P \in \mP$ of dimension $n-d$ is operational, then statement \hyperref[item:loccell]{(LocInfCell)}(a) is true, otherwise it is false.
    \end{enumerate}
\end{proof}



\subsection{Proving Theorem~\ref{thm:dec}: induction and a double procedure}
In this subsection we finally prove Theorem~\ref{thm:dec}.
The overall strategy is to use induction on $n$, while deciding the Conditions \hyperref[item:locr]{(LocR)} and \hyperref[item:locinf]{(LocInf)} from Theorem~\ref{thm:locglob}.
\thmdec*

\begin{proof}
    We use induction on $n$.
    As remarked in Subsection~\ref{subsec:cell}, the base case $n = 0$ degenerates into linear programming (given an $\R$-submodule $\mM$ of $\R^K$, decide whether $\mM \cap \Rpp^K$ contains an element).
    Suppose we have a decision procedure for all $n_0 < n$, we now construct a procedure for $n$.
    
    By Theorem~\ref{thm:locglob} it suffices to decide whether the two conditions \hyperref[item:locr]{(LocR)} and \hyperref[item:locinf]{(LocInf)} are both satisfied.
    First we check if \hyperref[item:locr]{(LocR)} is true using Proposition~\ref{prop:decr}.
    If \hyperref[item:locr]{(LocR)} is false then we return False and conclude there is no $\bff \in \mM \cap \ApK$ satisfying \eqref{eq:deccond}.
    If \hyperref[item:locr]{(LocR)} is true we proceed.

    We now run the two following procedures \emph{in parallel}:

    \begin{enumerate}
        \item \textbf{Procedure A:} We enumerate all elements of the $\Z[X_1^{\pm}, \ldots, X_n^{\pm}]$-module:
        \[
        \mM_{\Z} \coloneqq \left\{\sum_{j = 1}^m h_j \cdot \bg_j \;\middle|\; h_1, \ldots, h_m \in \Z[X_1^{\pm}, \ldots, X_n^{\pm}]\right\}.
        \]
        For each element $\bff \in \mM_{\Z}$, check if $\bff$ is in $\ApK$ and satisfies Property~\eqref{eq:deccond}.
        If $\bff \in \ApK$, then Property~\eqref{eq:deccond} can be checked by looking at the corresponding $\mG$-graph $\Gamma$ of $\bff$.
        Consider the graph $\Gamma_I$ obtained by keeping only the edges of $\Gamma$ whose label is in the set $I$.
        Then $\bff$ satisfies Property~\eqref{eq:deccond} if and only if every strict face $F$ of $\conv(V(\Gamma_I))$ contains the starting point of either an edge with label in $J$ or an edge going out of $F$.
        If some element $\bff \in \mM_{\Z}$ is in $\ApK$ and satisfies Property~\eqref{eq:deccond}, we stop the procedure and return True.
        
        \item \textbf{Procedure B:} We enumerate all $A \in \GL(n, \Z)$ and $d \in \{0, 1, \ldots, n-1\}$.
        For each $A$ and $d$, compute $Op(A, d)$ using Lemma~\ref{lem:compop} and the induction hypothesis.
        Using Lemma~\ref{lem:checkcell}, we check if the statement~\hyperref[item:loccell]{(LocInfCell)}(a) from Proposition~\ref{prop:dtocell} is false.
        If for some $A, d$, statement~\hyperref[item:loccell]{(LocInfCell)}(a) is false, then we stop the procedure and return False.
    \end{enumerate}
    We claim that one of the two above procedures must stop.
    
    Indeed, if $\mM$ contains an element of $\ApK$ satisfying Property~\eqref{eq:deccond}, then there exists an element $\bff \in \mM_{\Z} \cap \ApK$ satisfying Property~\eqref{eq:deccond} (see Lemma~\ref{lem:M}).
    In this case, Procedure A terminates by finding an element of $\mM_{\Z} \cap \ApK$ satisfying Property~\eqref{eq:deccond}.
    
    If $\mM$ does not contain an element of $\ApK$ satisfying Property~\eqref{eq:deccond}, then by Theorem~\ref{thm:locglob}, Condition~\hyperref[item:locinf]{(LocInf)} must be false (since we have already checked \hyperref[item:locr]{(LocR)} to be true).
    By the chain of Propositions~\ref{prop:inftoshift}, \ref{prop:shifttod} and \ref{prop:dtocell}, the statement~\hyperref[item:loccell]{(LocInfCell)}(a) must be false for some $A \in \GL(n, \Z)$ and $d \in \{0, 1, \ldots, n-1\}$.
    In this case, Procedure B terminates by finding $A \in \GL(n, \Z)$ and $d \in \{0, 1, \ldots, n-1\}$ where statement~\hyperref[item:loccell]{(LocInfCell)}(a) is false.
    
    Therefore, by running Procedure A and Procedure B in parallel, we obtain an algorithm that always terminates for $n$.
\end{proof}

\bibliography{metabelian}

\appendix

\section{Proof of Lemma~\ref{lem:subsume} and \ref{lem:grptoeul}}\label{app:prelim}

\lemsubsume*
\begin{proof}
    Fix the group $G$.
    Let $\mG$ be a finite subset of $G$.
    For the Identity Problem, we claim that the semigroup $\sgmG$ contains the neutral element $e$ if and only if there is some non-empty subset $\mH$ of $\mG$ such that $\langle \mH \rangle = \langle \mH \rangle_{grp}$.
    Indeed, if $\sgmG$ contains the neutral element $e$, suppose $e$ is represented by the word $w \in \mG^*$.
    Let $\mH \subseteq \mG$ be the set of letters appearing in $w$, then $w$ is a full-image word in the alphabet $\mH$ and hence $\langle \mH \rangle = \langle \mH \rangle_{grp}$ by Lemma~\ref{lem:word}.
    For the opposite implication, if $\langle \mH \rangle = \langle \mH \rangle_{grp}$ then $e \in \langle \mH \rangle_{grp} = \langle \mH \rangle \subseteq \langle \mG \rangle$.

    Therefore, to decide the Identity Problem, it suffices to check the Group Problem for every non-empty subset $\mH$ of $\mG$.

    For the Inverse Problem, let $\mG = \{g_1, \ldots, g_K\}$.
    Without loss of generality suppose we want to decide whether $g_1^{-1} \in \sgmG$.
    We claim that $g_1^{-1} \in \sgmG$ if and only if there is some subset $\mH$ of $\mG$, such that $\langle \mH \cup \{g_1\} \rangle = \langle \mH \cup \{g_1\} \rangle_{grp}$.
    Indeed, if $g_1^{-1} \in \sgmG$, suppose $g_1^{-1}$ is represented by the word $w \in \mG^*$.
    Let $\mH$ be the set of letters appearing in $w$, then $g_1 w$ is a full-image word in the alphabet $\mH \cup \{g_1\}$ representing the neutral element, and hence $\langle \mH \cup \{g_1\} \rangle = \langle \mH \cup \{g_1\} \rangle_{grp}$ by Lemma~\ref{lem:word}.
    For the opposite implication, if $\langle \mH \cup \{g_1\} \rangle = \langle \mH \cup \{g_1\} \rangle_{grp}$, then $g_1^{-1} \in \langle \mH \cup \{g_1\} \rangle_{grp} = \langle \mH \cup \{g_1\} \rangle \subseteq \sgmG$.

    Therefore, to decide the Inverse Problem, it suffices to check the Group Problem for every subset $\mH \cup \{g_1\}$ of $\mG$.
\end{proof}

\lemgrp*
\begin{proof}
    If the semigroup $\langle\mG\rangle$ is a group, then by Lemma~\ref{lem:word} there exists a full-image word $w$ representing the neutral element.
    Then its associated $\mG$-graph $\Gamma(w)$ is full-image, Eulerian and represents the neutral element.

    If $\Gamma$ is a full-image Eulerian $\mG$-graph representing the neutral element.
    Let $z \in V(\Gamma)$ be any vertex of $\Gamma$. Then consider the translation $\Gamma - z$: it represents the element $(\oX^z \cdot 0, 0) = (0, 0)$ and contains an Eulerian circuit starting from 0.
    We read from this Eulerian circuit a word $w$, then $w$ represents the neutral element.
    Furthermore, $w$ is full-image because $\Gamma - z$ is full-image.
    Therefore by Lemma~\ref{lem:word} the semigroup $\sgmG$ is a group.
\end{proof}

\section{Proof of Proposition~\ref{prop:metatoZ}}\label{app:metatoZ}

We now give a full proof of Proposition~\ref{prop:metatoZ}.
We start by stating a lemma of Baumslag.

\begin{restatable}[{\cite[Theorem~3.3]{baumslag1994algorithmic}, \cite[Theorem~9.5.3]{lennox2004theory}}]{lem}{lemsubgrp}\label{lem:subgroppres}
There is an algorithm which, when a finitely metabelian
presentation of $G$ is given, together with a finite subset $\mG \subseteq G$, finds a finite metabelian presentation of the subgroup $\sgmG_{grp}$.
\end{restatable}

The proof of~\cite[Theorem~9.5.3]{lennox2004theory} shows that the set $\mG$ is given as the generators of the finite metabelian presentation of $\langle\mG\rangle_{grp}$.\footnote{Indeed, let $H = \sgmG_{grp}$ and $G' = [G, G]$. The proof of~\cite[Theorem~9.5.3]{lennox2004theory} computes a finite presentation of $H \cap G'$ as a $\Z[HG'/G']$-module in terms of the generators of $H \cap G'$.
It then uses this presentation to compute a finite metabelian presentation of $H$.
This computation is done using the procedure outlined at the start of~\cite[Section~9.5]{lennox2004theory}, where the generators of the finite metabelian presentation of $H$ is the generating set $\mG$ of $H$.}
Since our goal is to decide the Group Problem (whether $\sgmG = \sgmG_{grp}$), we can without loss of generality suppose $G = \sgmG_{grp}$ by Lemma~\ref{lem:subgroppres}.

We now recall the definition of the \emph{wreath product}.
Given two groups $A, T$, their (restricted) wreath product $A \wr T$ is defined as a semidirect product $A^T \rtimes T$.
Here, $A^T$ is the direct sum of $A$ over the index set $T$ and is called the \emph{base group}. That is, $A^T$ is the set of sequences $(a_s)_{s \in T}, a_s \in A$ where $a_s = 0$ for all but finitely many $s \in T$.
It is a group by pointwise multiplication.
The wreath product $A \wr T = A^T \rtimes T$ is the set of pairs $((a_s)_{s \in T}, t)$ with $(a_s)_{s \in T} \in A^T, t \in T$, where multiplication is defined by 
\[
\left((a_s)_{s \in T}, t \right) \cdot \left((a'_s)_{s \in T}, t' \right) = \left((a_s a'_{t^{-1}s})_{s \in T}, t t' \right).
\]
The wreath product $A \wr T$ canonically contains as subgroups $T \cong \{(e_{A^T}, t) \mid t \in T\}$, where $e_{A^T}$ is the neutral element of $A^T$, as well as the base group $A^T \cong \{((a_s)_{s \in T}, e_T) \mid (a_s)_{s \in T} \in A^T\}$, where $e_T$ is the neutral element of $T$.

An important special case of the wreath product is when $A = \Z^n$ and $T$ is abelian.
In this case, the base group $A^T$ is isomorphic to the direct power $\left( \Z[T] \right)^n$ of the group ring 
\[
\Z[T] \coloneqq \left\{\sum_{t \in T} z_t t \;\middle|\; z_t \in \Z, \text{ and $z_t = 0$ except for a finite number of $t$} \right\}.
\]
Here, $\sum_{t \in T} z_t t$ denotes a formal sum with finite support.
The wreath product $A \wr T$ then becomes the semidirect product $\left( \Z[T] \right)^n \rtimes T$ consisting of the pairs $(y, t)$, where $y \in \left( \Z[T] \right)^n, t \in T$, with multiplication given by $(y, t) \cdot (y', t') = (y + t \cdot y', t t')$.

Furthermore, if $A = \Z^n$ and $T = \Z^d$, then the wreath product $A \wr T$ is simply the semidirect product $\left( \Z[X_1^{\pm}, \ldots, X_d^{\pm}] \right)^n \rtimes \Z^d$ defined in Equation~\eqref{eq:defsemi}.
The following classic result of Magnus gives an explicit embedding of the quotient of a free group into a wreath product.

\begin{lem}[Magnus Embedding Theorem~{\cite[Lemma~2]{baumslag1973subgroups}, \cite{magnus1939theorem}}]\label{lem:magnusemb}
    Let $F$ be a free group, freely generated by a set $X = \{x_i \mid i \in I\}$, and let $R$ be a normal subgroup of $F$. 
    Let the mapping $x_i R \mapsto t_i, i \in I,$ define an isomorphism from $F /R$ to a group $T$ generated by $t_i, i \in I$. 
    Furthermore, let $A$ be a free abelian group, freely generated by the elements $a_i, i \in I$. 
    Then the mapping 
    \[
    x_i [R, R] \mapsto a_i t_i, \;\; i \in I,
    \]
    defines an injection of $F /[R, R]$ into the wreath product $W = A \wr T$.
\end{lem}

Recall that the semidirect product $\mY \rtimes \Z^n$ canonically contains the subgroup $\Z^n \cong \{(0, a) \mid a \in \Z\}$.
The next lemma shows that a finitely generated metabelian group can be effectively embedded in a quotient $\left(\mY \rtimes \Z^n\right)/H$, where $H$ is a subgroup of $\Z^n \cong \{(0, a) \mid a \in \Z\}$.

\begin{restatable}[Corollary of~{\cite[Lemma~3]{baumslag1973subgroups}}]{lem}{lemembed}\label{lem:embed}
Let $G$ be a finitely generated metabelian group.
Then $G$ is isomorphic to a subgroup $\widetilde{G}$ of the quotient $\left(\mY \rtimes \Z^n\right)/H$, where
\begin{enumerate}[nosep, label = (\roman*)]
    \item $n \in \N$ and $\mY$ is a finitely presented $\Z[X_1^{\pm}, \ldots, X_n^{\pm}]$-module.
    \item $H$ is a subgroup of $\Z^n \leq \mY \rtimes \Z^n$, and elements of $H$ commute with all elements in $\mY \rtimes \Z^n$.
    \item the image $\pi(\widetilde{G})$ under the projection $\pi: \left(\mY \rtimes \Z^n\right)/H \rightarrow \Z^n / H$ is equal to $\Z^n / H$.
\end{enumerate}
Furthermore, given a finite metabelian presentation of $G$, the integer $n$, the finite presentation of $\mY$, the generators of $H$ and the isomorphism $G \xrightarrow{\sim} \widetilde{G}$ can all be effectively computed.
\end{restatable}
\begin{proof}
    This lemma is a simple extension of~\cite[Lemma~3]{baumslag1973subgroups}.
    We give a recount of its proof to show effectiveness and the conditions (ii) and (iii).

    Let $\{g_1, \ldots, g_n\}$ be the generators of $G$, let $F$ be the free group freely generated by $x_1, \ldots, x_n$, such that $M_n = F/[[F, F], [F, F]]$ and $G = M_n / \ncl_{M_n}(\widetilde{r_1}, \ldots, \widetilde{r_m})$ is the given finite metabelian presentation of $G$.
    Here, $\widetilde{r_i} = r_i [[F, F], [F, F]], i = 1, \ldots, m$ where $r_i$ is given as an element of the free group $F$.
    Let $\phi$ be the epimorphism from $F$ to $G$ defined by
    \[
    \phi \colon x_i \mapsto g_i, \quad (i = 1, \ldots, n).
    \]
    Let $K$ be the kernel of $\phi$ and let $R$ be the inverse image of $[G, G]$ under $\phi$.
    Since $\phi([F, F]) = [G, G]$ we have $R = \phi^{-1}([G, G]) = K[F, F]$.
    Also, $K$ is the normal subgroup of $F$ generated by $r_1, \ldots, r_m$ and $[[F, F], [F, F]]$.
    Then $R/K \cong [G, G]$ and hence is abelian.
    Therefore $[R, R] \leq K$, which means that
    \[
    [R, R] \leq K \leq R.
    \]

    Now let $A$ be a free abelian group on $a_1, \ldots, a_n$ and consider the wreath product $A \wr T$, where $T = F/R$.
    The structure of the abelian group $T$ can be effectively computed by $T = F/R = F/\left(K [F, F]\right) \cong \left( F/[F, F] \right)/\left( K[F, F]/[F, F] \right)$.
    In other words, writing $r_j = x_{i_{j_1}}^{e_{j_1}} \cdots x_{i_{j_{\ell_j}}}^{e_{j_{\ell_j}}}, j = 1, \ldots, m$, and writing $t_i = x_i R, i = 1, \ldots, n$, then $T$ is the quotient $F_{ab}/H$, where $F_{ab} = F/[F, F]$ is the free abelian group generated by $t_1, \ldots, t_n$, and $H = K[F, F]/[F, F]$ is the subgroup of $F_{ab}$ generated by $t_{i_{j_1}}^{e_{j_1}} \cdots t_{i_{j_{\ell_i}}}^{e_{j_{\ell_i}}}, j = 1, \ldots, m$.
   
    
    Let $\psi$ be the homomorphism of $F$ into $A \wr T$ defined by
    \[
    \psi(x_i) = a_i t_i, \quad (i = 1, \ldots, n).
    \]
    By Lemma~\ref{lem:magnusemb} the kernel of $\psi$ is $[R, R]$.
    Hence $\psi$ induces an isomorphism $\psi_*: x_i[R, R] \mapsto a_i t_i$, of $F/[R, R]$ into $A \wr T$.

    We put $N = \psi_*(K/[R, R])$. Now $N \leq \psi_*(F/[R, R])$. Therefore $N$ is normalized by the elements $a_i t_i, i = 1, \ldots, n$.
    But it follows from the definition of $\psi_*$ that $N$ is contained in the base group $B = A^T$ of $A \wr T$.
    Since $B$ is abelian, $N$ is normalized by $B$ and hence by the elements $t_i$, and therefore by all of $A \wr T$.
    In other words $N$ is normal, and is normally generated by $\psi_*(r_1[R, R]), \ldots, \psi_*(r_m[R, R])$.
    This is because $N = \psi_*(K/[R, R])$ and $K$ is normally generated by $r_1, \ldots, r_m$ and the set $[[F, F], [F, F]] \leq [R, R]$.
    Note that
    \begin{equation}\label{eq:genN}
    \psi_*(r_j[R, R]) = \left(a_{i_{j_1}}t_{i_{j_1}}\right)^{e_{j_1}} \cdots \left(a_{i_{j_{\ell_j}}}t_{i_{j_{\ell_j}}}\right)^{e_{j_{\ell_j}}}, j = 1, \ldots, m.
    \end{equation}
    Now $G$ is isomorphic to the subgroup $\widetilde{G} \coloneqq \psi_*(F/[R, R])/N$ of $\left(A \wr T\right)/N$ by the map $g_i \mapsto a_i t_i N$.
    
    Note that $\left(A \wr T\right)/N = \left(A^T/N\right) \rtimes T$ where $N$ is normal, so $N$ is the $\Z[T]$-module generated by $\psi_*(r_1[R, R]), \ldots, \psi_*(r_m[R, R])$.
    Furthermore, since $\widetilde{G}$ contains the elements $a_i t_i N, i = 1, \ldots, n$, its projection onto $T$ contains the elements $t_i, i = 1, \ldots, n$ and hence is $T$ itself.

    We write $T = F_{ab}/H = \Z^n / H$ since $F_{ab}$ is the free abelian group over $n$ generators.
    The canonical projection $\Z^n \rightarrow T$ induces a ring homomorphism $\Z[\Z^n] \rightarrow \Z[T]$, so the $\Z[T]$-module $A^T/N$ is naturally also a $\Z[X_1^{\pm}, \ldots, X_n^{\pm}] = \Z[\Z^n]$-module, where elements of the form $\oX^h, h \in H$ act trivially.
    Taking $\mY \coloneqq A^T/N$, we define the semidirect product $\mY \rtimes \Z^n$ by considering $\mY = A^T/N$ as a $\Z[X_1^{\pm}, \ldots, X_n^{\pm}]$-module.
    We show that elements of $H$ commute with every element in $\mY \rtimes \Z^n$.
    This is rather straightforward: since elements of the form $\oX^h, h \in H$ act trivially on $\mY = A^T/N$, we have $(0, h) (y, a) (0, h^{-1}) = (\oX^h \cdot y, a) = (y, a)$.
    Therefore, elements of $H$ commute with every element in $\mY \rtimes \Z^n$, proving (ii).
    A fortiori, this shows that $H$ is a normal subgroup of $\mY \rtimes \Z^n$.
    
    Finally, we have
    \[
    \widetilde{G} = \left(A^T/N\right) \rtimes T = \mY \rtimes \left(\Z^n / H\right) = \left(\mY \rtimes \Z^n\right)/H,
    \]
    and (iii) follows directly from the fact that the projection $\widetilde{G} = \left(A^T/N\right) \rtimes T \rightarrow T = \Z^n / H$ has full image.
    We now show that $\mY = A^T/N$ can be effectively written as a finitely presented $\Z[X_1^{\pm}, \ldots, X_n^{\pm}]$-module.
    First, we write $A^T$ as a finitely presented $\Z[X_1^{\pm}, \ldots, X_n^{\pm}]$-module in the following way.
    Let $h_1, \ldots, h_m$ the generators of $H$ as a subgroup of $\Z^n$.
    Since $A \cong \Z^n$, we have $A^T = \left(\Z[T]\right)^n$, and $\Z[T]$ is the quotient of $\Z[X_1^{\pm}, \ldots, X_n^{\pm}] = \Z[\Z^n]$ by the ideal generated by elements $\oX^{h_i} - 1, i = 1, \ldots, m$.
    Hence we obtain a finite presentation of the $\Z[X_1^{\pm}, \ldots, X_n^{\pm}]$-module $\Z[T]$, and from it a finite presentation of the module $A^T = \left(\Z[T]\right)^n$.
    The generators of $N \subseteq A^T$ as a $\Z[X_1^{\pm}, \ldots, X_n^{\pm}]$-module is the same as its generators as a $\Z[T]$-module, which are given by the elements in \eqref{eq:genN}.
    Therefore, we obtain a finite presentation of the $\Z[X_1^{\pm}, \ldots, X_n^{\pm}]$-module $\mY = A^T/N$, and (i) follows.
\end{proof}

We can hence suppose $G$ is given as a subgroup of $\left(\mY \rtimes \Z^n\right)/H$ and the generator set $\mG$ is given as a subset of $\left(\mY \rtimes \Z^n\right)/H$.
Writing $\mG = \{g_1 H, \ldots, g_k H\}$ where $g_1, \ldots, g_k$ are elements of $\mY \rtimes \Z^n$, and let $h_1, \ldots, h_M$ be the generators of $H \subseteq \mY \rtimes \Z^n$ as a \emph{semi}group.
Then
\begin{restatable}{lem}{lemGH}\label{lem:GH}
    The semigroup $\sgmG$ is a group if and only if the semigroup generated by $g_1, \ldots, g_k$, $h_1, \ldots, h_M \in \mY \rtimes \Z^n$ is a group.
\end{restatable}
\begin{proof}
    Suppose $\sgmG$ is a group.
    Then by Lemma~\ref{lem:word} there exists a full-image word 
    \[
    w = g_{i_1} H g_{i_2} H \cdots g_{i_p} H
    \]
    over the alphabet $\mG$ representing the neutral element in $\left(\mY \rtimes \Z^n\right)/H$.
    Since elements of $H$ commute with every element of $\mY \rtimes \Z^n$, this means $g_{i_1} g_{i_2} \cdots g_{i_p} \in H$.
    Therefore there exists a word $v$ over the alphabet $\mH \coloneqq \{h_1, \ldots, h_M\}$ such that $g_{i_1} g_{i_2} \cdots g_{i_p} \cdot v$ represents the neutral element in $\mY \rtimes \Z^n$.
    Since $\mH$ generates the group $H$ as a semigroup, there exists a full-image word $v'$ over the alphabet $\mH$ representing the neutral element of $H \leq \mY \rtimes \Z^n$.
    Then, the word $g_{i_1} g_{i_2} \cdots g_{i_p} \cdot v \cdot v'$ is full-image over the alphabet $\{g_1, \ldots, g_k, h_1, \ldots, h_M\}$ and it represents the neutral element in $\mY \rtimes \Z^n$.
    By Lemma~\ref{lem:word}, the semigroup generated by $\{g_1, \ldots, g_k, h_1, \ldots, h_M\}$ is a group.

    For the other implication, suppose now the semigroup generated by $\{g_1, \ldots, g_k, h_1, \ldots, h_M\}$ is a group.
    By Lemma~\ref{lem:word} there exists a full-image word $\widetilde{w}$ over the alphabet $\{g_1, \ldots, g_k, h_1, \ldots, h_M\}$ representing the neutral element.
    Since the elements $h_1, \ldots, h_M$ commute with all other elements, we can move them to the rightmost side of $\widetilde{w}$ and suppose
    \[
    \widetilde{w} = g_{i_1} g_{i_2} \cdots g_{i_p} h_{j_1} h_{j_2} \cdots h_{j_q}.
    \]
    Then the word $g_{i_1} H \cdot g_{i_2} H \cdots \cdot g_{i_p} H$ is full-image over the alphabet $\mG$ and represents the neutral element of $\left(\mY \rtimes \Z^n\right)/H$.
    By Lemma~\ref{lem:word}, the semigroup $\sgmG$ is a group.
\end{proof}

Proposition~\ref{prop:metatoZ} follows from Lemmas~\ref{lem:subgroppres}, \ref{lem:embed} and \ref{lem:GH}:

\propmetatoZ*
\begin{proof}
    By Lemma~\ref{lem:subgroppres}, we can compute a finitely metabelian presentation for the group $\sgmG_{grp}$.
    Since the generators $\mG$ are explicitly given under this presentation, we can without loss of generality suppose $G = \sgmG_{grp}$.
    By Lemma~\ref{lem:embed}, $G$ can be effectively embedded as a subgroup of a quotient $\left(\mY \rtimes \Z^n\right)/H$, where $H$ is a subgroup of $\Z^n \leq \mY \rtimes \Z^n$, and elements of $H$ commute with all elements of $\mY \rtimes \Z^n$.
    We can hence suppose $G = \widetilde{G}$ is a subgroup of $\left(\mY \rtimes \Z^n\right)/H$ and the generator set $\mG$ is given as $\{g_1 H, \ldots, g_k H\}$ where $g_1, \ldots, g_k \in \mY \rtimes \Z^n$.
    Let $h_1, \ldots, h_M$ be the generators of $H \subseteq \mY \rtimes \Z^n$ as a \emph{semi}group.

    Let $\widetilde{\mG} \coloneqq \{g_1, \ldots, g_k, h_1, \ldots, h_M \} \subseteq \mY \rtimes \Z^n$.
    By Lemma~\ref{lem:embed}, the image of $\sgmG_{grp} = \widetilde{G}$ under the projection $\pi: \left(\mY \rtimes \Z^n\right)/H \rightarrow \Z^n / H$ is equal to $\Z^n / H$.
    Since $h_1, \ldots, h_M$ generate $H$ as a semigroup, the group generated by $\widetilde{\mG}$ admits image $\left(\Z^n/H\right) + H = \Z^n$ under the canonical projection $\mY \rtimes \Z^n \rightarrow \Z^n$.
    Finally, by Lemma~\ref{lem:GH}, the semigroup $\sgmG$ is a group if and only if the semigroup generated by $\widetilde{\mG}$ is a group.
    This proves the proposition.
\end{proof}

\section{Omitted proofs from Section~\ref{sec:main}}\label{app:main}
\propgtoe*
\begin{proof}
    By Lemma~\ref{lem:grptoeul} and Observation~\ref{obs:contoacc}, it suffices to prove the ``if'' direction.
    Suppose $\Gamma$ is a full-image symmetric face-accessible $\mG$-graph representing the neutral element, we construct a full-image Eulerian $\mG$-graph that represents the neutral element.
    By Theorem~\ref{thm:acctocon}, there exist $z_1, \ldots, z_m \in \Z^n$, such that $\hG \coloneqq \bigcup_{i = 1}^m \left( \Gamma + z_i \right)$ is an Eulerian graph.
    Since $\hG$ contains a translation of $\Gamma$, it is full-image.
    Each translation $\Gamma + z_i$ represents the element $\oX^{z_i} \cdot 0 = 0$, so the union $\hG$ also represents the neutral element.
\end{proof}

\propeulertoeq*
\begin{proof}
    (i) $\Gamma$ is full-image if and only if each label appears at least once, meaning $f_i \neq 0$ for all $i$.

    (ii) We have 
    \begin{multline*}
        \sum_{i = 1}^K f_i \cdot (\oX^{a_i} - 1) = \sum_{i = 1}^K \sum_{e \in E(\Gamma), \ell(e) = i} \oX^{s(e)} (\oX^{a_i} - 1)
        = \sum_{i = 1}^K \sum_{e \in E(\Gamma), \ell(e) = i} (\oX^{s(e) + a_i} - \oX^{s(e)}) \\
        = \sum_{e \in E(\Gamma)} (\oX^{d(e)} - \oX^{s(e)}) = \sum_{e \in E(\Gamma)} \oX^{d(e)} - \sum_{e \in E(\Gamma)} \oX^{s(e)}.
    \end{multline*}
    This is equal to zero if and only if the in-degree equals the out-degree at every vertex.

    (iii) Let $C$ be the convex hull of $V(\Gamma)$. 
    For every strict face $F$ of $C$ there is a vector $v \in \Rns$ such that $F$ consists of all points $x$ in $C$ where $v \cdot  x$ is maximal.
    Conversely, for every vector $v \in \Rns$, the set $F$ of all points $x$ in $C$ such that $v \cdot  x$ is maximal forms a strict face of $C$.

    Let $F$ be a strict face, then $F$ is accessible if and only if some edge starting in $F$ does not end in $F$.
    Let $e$ be an edge starting in $F$, with label $\ell(e)$.
    Then $v \cdot  s(e)$ is maximal among all $e \in E(\Gamma)$.
    Since the monomial $\oX^{s(e)}$ is contained in $f_{\ell(e)}$, this means $\ell(e) \in M_{v}(\{1, \ldots, K\}, \bff)$.
    
    Observe that $d(e) \in F$ if and only if $a_{\ell(e)} = d(e) - s(e)$ is orthogonal to $v$, which is equivalent to $\ell(e) \not\in O_{v}$.
    Therefore, $F$ is accessible if and only if an edge $e$ exists such that $\ell(e) \in M_{v}(\{1, \ldots, K\}, \bff)$ and $\ell(e) \in O_{v}$;
    that is, $O_{v} \cap M_{v}(\{1, \ldots, K\}, \bff) \neq \emptyset$.

    By the definition of face-accessibility, $\Gamma$ is face-accessible if and only if $O_{v} \cap M_{v}(\{1, \ldots, K\}, \bff) \neq \emptyset$ holds for every $v \in \Rns$.

    (iv) Suppose $\Gamma$ is symmetric, then $\sum_{e \in E(\Gamma)} a_{\ell(e)} = 0$. By Equation~\eqref{eq:edges}, $\Gamma$ represents the element 
    \[
    \left(\sum_{e \in E(\Gamma)} \oX^{s(e)} \cdot y_{\ell(e)}, \sum_{e \in E(\Gamma)} a_{\ell(e)}\right) = \left(\sum_{i = 1}^K \sum_{e \in E(\Gamma), \ell(e) = i} \oX^{s(e)} \cdot y_{i}, 0\right) = \left(\sum_{i = 1}^K f_i \cdot y_i, 0\right),
    \]
    which is the neutral element if and only if $\sum_{i = 1}^K f_i \cdot y_i = 0$.
\end{proof}

\lemMZ*
\begin{proof}
    Recall that $\mY$ is given as a quotient $M/N$ where $N$ and $M$ are $\Z[\oX^{\pm}]$-submodules of $\Z[\oX^{\pm}]^d$ respectively generated by $\bm_1, \ldots, \bm_{\ell'}$ and $\bn_1, \ldots, \bn_{\ell} \in \Z[\oX^{\pm}]^d$.
    For $i = 1, \ldots, K$, the element $y_i$ is given as $y_i = \tby_i + N$ where $\tby_i \in M \subseteq \Z[\oX^{\pm}]^d$.
    
    The equation $\sum_{i = 1}^K f_i \cdot y_i = 0$ can be written as $\sum_{i = 1}^K f_i \cdot \tby_i \in N$, which is equivalent to
    \begin{equation}
        \sum_{i = 1}^K f_i \cdot \tby_i = \sum_{j = 1}^{\ell} g_j \cdot \bn_j \text{ for some } g_1, \ldots, g_{\ell} \in \Z[\oX^{\pm}].
    \end{equation}
    Let $\widetilde{\mM}$ be the set of solutions $(f_1, \ldots, f_K, g_1, \ldots, g_{\ell}) \in \Z[\oX^{\pm}]^{K + \ell}$ of the following system of homogeneous linear equations:  
    \begin{align*}
        & \sum_{i = 1}^K f_i \cdot \tby_i - \sum_{j = 1}^{\ell} g_j \cdot \bn_j = 0\\
        & \sum_{i = 1}^K f_i \cdot (\oX^{a_i} - 1) = 0.
    \end{align*}
    The set $\widetilde{\mM}$ is also known as the syzygy module.
    It is a classic result from linear algebra over Noetherian rings that a finite set of generators $(\bff_1, \bg_1), \ldots, (\bff_s, \bg_s)$ for $\widetilde{\mM}$ can be effectively computed (see~\cite{berkesch2015syzygies, schreyer1980berechnung} or \cite[Theorem~15.10]{eisenbud2013commutative}).
    Let $\pi \colon \Z[\oX^{\pm}]^{K + \ell} \rightarrow \Z[\oX^{\pm}]^{K}, (f_1, \ldots, f_K, g_1, \ldots, g_{\ell}) \mapsto (f_1, \ldots, f_K)$ be the projection onto the first $K$ coordinates.
    Then $\mM_{\Z} = \pi(\mM)$, and a finite set of generators for $\mM_{\Z}$ is simply $\left\{\pi((\bff_1, \bg_1)), \ldots, \pi((\bff_s, \bg_s))\right\} = \left\{\bff_1, \ldots, \bff_s \right\}$.
\end{proof}

\lemM*
\begin{proof}
    An element $\tbf \in \mM_{\Z} \cap \left(\N[\oX^{\pm}]^*\right)^K$ satisfying Property~\eqref{eq:gen} is obviously an element in $\mM \cap \left(\Rp[\oX^{\pm}]^*\right)^K$.
    Therefore it suffices to prove the ``if'' implication.

    Suppose we have an element $\bff \in \mM \cap \left(\Rp[\oX^{\pm}]^*\right)^K$ satisfying Property~\eqref{eq:gen}, we show that there is an element $\tbf \in \mM_{\Z} \cap \left(\N[\oX^{\pm}]^*\right)^K$ satisfying Property~\eqref{eq:gen}.

    Write $\bff = (f_1, \ldots, f_K)$ where for $i = 1, \ldots, K$,
    \[
    f_i = \sum_{b \in B_i} c_{i, b} \oX^{b}.
    \]
    Here, the \emph{support} $B_i$ is a non-empty finite subset of $\Z^n$, and $c_{i, b} \in \Rpp$ for all $b \in B_i$.
    Since Property~\eqref{eq:gen} depends only on the supports $B_1, \ldots, B_K$, it suffices to show that there exists $\tbf = (\tf_1, \ldots, \tf_K) \in \mM_{\Z} \cap \left(\N[\oX^{\pm}]^*\right)^K$ where
    \[
    \tf_i = \sum_{b \in B_i} \tc_{i, b} \oX^{b}
    \]
    and $\tc_{i, b} \in \Zpp$ for all $b \in B_i$.
    
    Since $\bff \in \mM$, we have $\bff = \sum_{j = 1}^m h_j \cdot \bg_j$ for some $h_1, \ldots, h_m \in \R[\oX^{\pm}]$.
    For each $j \in \{1, \ldots, m\}$, write $h_j = \sum_{b \in H_j} h_{j, b} \oX^{b}$, where $H_j$ is a finite subset of $\Z^n$.
    Then the equation $\bff = \sum_{j = 1}^m h_j \cdot \bg_j$ can be rewritten as a finite system of linear equations over $\R$, where the left hand sides are $0$ or $c_{i, b}, b \in B_i, i = 1, \ldots, K$, and the right hand sides are $\Z$-linear combinations of the variables $h_{j, b}, j \in \{1, \ldots, m\}, b \in H_j$ (because the coefficients of $\bg_j$ are integers for all $j$).

    Since this system of linear equations is homogeneous and the coefficients are all in $\Z$, it has a solution $h_{j, b} \in \R, j \in \{1, \ldots, m\}, b \in H_j$ and $c_{i, b} \in \Rpp, b \in B_i, i = 1, \ldots, K,$ if and only if it has a solution with $h_{j, b} \in \Q, c_{i, b} \in \Qpp$ for all $i, j, b$.
    By multiplying all $h_{j, b}, c_{i, b}$ with their common denominator, we obtain a solution $\widetilde{h}_{j, b} \in \Z, \tc_{i, b} \in \Zpp$ for all $i, j, b$.
    Then, $\tf_i \coloneqq \sum_{b \in B_i} \tc_{i, b} \oX^{b}, i = 1, \ldots, K$ and $\widetilde{h}_j = \sum_{b \in H_j} \widetilde{h}_{j, b} \oX^{b}, j = 1, \ldots, m,$ satisfy $\tbf = \sum_{j = 1}^m \widetilde{h}_j \cdot \bg_j$.
    Hence, $\tbf = (\tf_1, \ldots, \tf_K) \in \mM_{\Z} \cap \left(\N[\oX^{\pm}]^*\right)^K$.
    The element $\tbf$ satisfies Property~\eqref{eq:gen} since the condition depends only on the supports $B_1, \ldots, B_K$.
\end{proof}

\end{document}